\documentclass{article}

\usepackage[margin=1in]{geometry} 
\usepackage{amsmath,amssymb,amsthm}
\usepackage{amsmath}
\usepackage{amsfonts}
\usepackage{hyperref}
\usepackage{mathrsfs}
\usepackage{mathtools}
\usepackage{graphicx}
\usepackage{tikz-cd}
\usepackage{xcolor}
\usepackage{indentfirst}
\usepackage[shortlabels]{enumitem}
\usepackage[capitalize]{cleveref}
\usepackage[overload]{empheq}
\usepackage[title]{appendix}
\usepackage[T1]{fontenc}

\newcommand{\ZZ}{\mathbb{Z}}
\newcommand{\NN}{\mathbb{N}}
\newcommand{\QQ}{\mathbb{Q}}
\newcommand{\FF}{\mathbb{F}}
\newcommand{\EE}{\mathbb{E}}

\DeclareMathOperator{\Hom}{Hom}
\newcommand{\catC}{\mathcal{C}}
\newcommand{\op}{\text{op}}
\newcommand{\cond}[1]{\underline{#1}}

\newcommand{\condHom}{\cond{\Hom}}

\newcommand{\Space}{\mathcal{S}}
\newcommand{\ProVect}{\operatorname{Pro} (\Vect_k^\omega)}
\newcommand{\ProVectsigma}{\ProVect_{< \sigma}}
\newcommand{\LSym}{\mathbb{L}\Sym}
\renewcommand{\nu}{\operatorname{nu}}

\makeatletter
\newcommand{\colim@}[2]{%
  \vtop{\m@th\ialign{##\cr
    \hfil$#1\operator@font colim$\hfil\cr
    \noalign{\nointerlineskip\kern-\ex@}\cr}}%
}

\newcommand{\colim}{%
  \mathop{\mathpalette\colim@{\scriptscriptstyle}}\nmlimits@
}

\DeclareMathOperator{\Solid}{Solid}
\DeclareMathOperator{\RHom}{RHom}
\DeclareMathOperator{\ED}{ED}
\DeclareMathOperator{\Fun}{Fun}
\DeclareMathOperator{\Sp}{Sp}
\DeclareMathOperator{\Map}{Map}
\DeclareMathOperator{\Ind}{Ind}
\DeclareMathOperator{\Cond}{Cond}
\DeclareMathOperator{\Lie}{Lie}
\DeclareMathOperator{\im}{Im}
\DeclareMathOperator{\CAlg}{CAlg}
\DeclareMathOperator{\Poly}{Poly}
\DeclareMathOperator{\cofib}{cofib}
\DeclareMathOperator{\fib}{fib}

\DeclareMathOperator{\End}{End}
\DeclareMathOperator{\coeq}{coeq}

\DeclareMathOperator{\Set}{Set}
\DeclareMathOperator{\coker}{coker}
\DeclareMathOperator{\Vect}{Vect}
\DeclareMathOperator{\Pro}{Pro}
\DeclareMathOperator{\pro}{pro}

\DeclareMathOperator{\Gr}{Gr}
\DeclareMathOperator{\Tot}{Tot}
\DeclareMathOperator{\forget}{forget}

\DeclareMathOperator{\ft}{ft}

\renewcommand{\Bar}{\operatorname{Bar}}
\DeclareMathOperator{\id}{id}
\DeclareMathOperator{\Tor}{Tor}

\DeclareMathOperator{\Sym}{Sym}
\DeclareMathOperator{\art}{art}
\DeclareMathOperator{\LMod}{LMod}
\DeclareMathOperator{\ProFin}{ProFin}
\DeclareMathOperator{\Shv}{Shv}
\DeclareMathOperator{\Mod}{Mod}
\DeclareMathOperator{\Der}{Der}
\DeclareMathOperator{\aperf}{aperf}
\DeclareMathOperator{\cN}{cN}
\DeclareMathOperator{\perf}{perf}

\DeclareMathOperator{\ds}{ds}
\DeclareMathOperator{\Alg}{Alg}
\DeclareMathOperator{\Comm}{Comm}
\DeclareMathOperator{\cn}{cn}
\DeclareMathOperator{\gr}{gr}
\DeclareMathOperator{\Moduli}{Moduli}
\DeclareMathOperator{\hyper}{hyp}

\renewcommand{\epsilon}{\varepsilon}

\usepackage{blindtext}

\usepackage{thmtools,thm-restate}

\newtheorem{theorem}{Theorem}
\newtheorem{lemma}[theorem]{Lemma}
\newtheorem{prop}[theorem]{Proposition}
\newtheorem{cor}[theorem]{Corollary}

\theoremstyle{definition}
\newtheorem{defi}[theorem]{Definition}
\newtheorem{rem}[theorem]{Remark}
\newtheorem{example}[theorem]{Example}
\newtheorem{construction}[theorem]{Construction}

\numberwithin{equation}{section}
\numberwithin{lemma}{section}
\numberwithin{theorem}{section}
\numberwithin{cor}{section}
\numberwithin{defi}{section}
\numberwithin{rem}{section}
\numberwithin{example}{section}
\numberwithin{prop}{section}
\numberwithin{construction}{section}
\numberwithin{notation}{section}
\numberwithin{conjecture}{section}
\numberwithin{problem}{section}

\begin{document}

\title{Ultrasolid Homotopical Algebra}
\author{Sofía Marlasca Aparicio}
\date{}

\maketitle

\begin{abstract}
	Solid modules over $\QQ$ or $\FF_p$, introduced by Clausen and Scholze, are a well-behaved variant of complete topological vector spaces that forms a symmetric monoidal Grothendieck abelian category. For a discrete field $k$, we construct the category of ultrasolid $k$-modules, which specialises to solid modules over $\QQ$ or $\FF_p$. In this setting, we show some commutative algebra results like an ultrasolid variant of Nakayama's lemma. We also explore higher algebra in the form of animated and $\EE_\infty$ ultrasolid $k$-algebras, and their deformation theory. We focus on the subcategory of complete profinite $k$-algebras, which we prove is contravariantly equivalent to equal characteristic formal moduli problems with coconnective tangent complex.
\end{abstract}

\tableofcontents

\section{Introduction}

Condensed mathematics is a new framework recently introduced by Clausen and Scholze (see \cite{condensed} and \cite{analytic}) that makes topology and algebra interact in a more well-behaved manner. For example, in the case of condensed modules (\cite{condensed}, Definition 1.2), this provides a new generalisation of topological modules that forms a complete and cocomplete abelian category.

\begin{example}
	If $R$ is a commutative ring and $S$ is a compactly generated Hausdorff space, we can form the condensed $R$-module $R[S]$, which is the free condensed $R$-module on the space $S$.

	If $S = \varprojlim S_i$ is a profinite space, we can form this free $R$-module in a way that takes into account the profinite structure of $S$. Define
	\begin{align*}
		R[S]^\blacksquare := \varprojlim R[S_i]
	\end{align*}
	which should be thought of as the profinite completion of $R[S]$. Note that this is independent of the inverse limit chosen to represent $S$.
\end{example}

Condensed modules are too general for some constructions, so it is convenient to move to a suitable subcategory of "complete" objects.

\begin{defi}
	A condensed $R$-module $M$ is \emph{solid} if for every profinite space $S$ the map $R[S] \to R[S]^\blacksquare$ induces an isomorphism
	\begin{align*}
		\Hom (R[S]^\blacksquare , M) \xrightarrow{\simeq} \Hom (R[S], M)
	\end{align*}
\end{defi}

\begin{rem}
	One can define other "completions" of $R[S]$, which gives rise to the more general notion of analytic rings and solid modules over an analytic ring (cf.\ \cite{condensed}, \S VII).

	A \emph{preanalytic ring structure on $R$} is a functor from extremally disconnected spaces to condensed $R$-modules, $S \mapsto \widehat{R[S]}$, taking finite disjoint unions to products, and a natural transformation $R[S] \to \widehat{R[S]}$, so we can think of a choice of $\widehat{R[S]}$ as a choice of completion of $R[S]$. In the cases $R = \ZZ, \FF_p$ or $ \QQ$, the preanalytic ring defined above forms an analytic ring, which means that many desirable properties are satisfied and the theory of solid modules is very well-behaved: they form a reflexive subcategory of condensed $R$-modules, it is compactly generated by projectives of the form $R[S]^\blacksquare$  for $S$ extremally disconnected, and it is stable under all limits and colimits computed in condensed $R$-modules.

	However, for general $R$, the preanalytic ring defined above is not necessarily an analytic ring, making the desired category of solid modules not as well-behaved, since the category of solid modules is not necessarily closed under colimits. An example of this is when $R$ is a field that is not the localisation of a finitely generated $\ZZ$-algebra (\cite{goodtheorysolid}).

	For our purposes, we will now fix $R = k$ a field. This is to make some aspects of the theory more convenient. For example, for $R = \ZZ$, it was proved by Efimov that projective solid abelian groups need not be flat, which creates problems for some homological constructions. Additionally, working with a field gives many convenient properties of complexes with profinite homology (\cref{prop-complexesprohomology}), which we use, for example, to embed pro-Artinian $k$-algebras into the category of ultrasolid $k$-algebras (\cref{theorem-proartinian}).

	Our work builds a category that captures all of the desirable properties of the theory of solid $k$-modules, even when $k$ is not the localisation of a finitely generated $\ZZ$-algebra. That is, we can form an abelian category with the expected compact projective generators. 
\end{rem}

\begin{defi}
	We write $\Pro (\Vect_k ^\omega)$ for the pro-completion (\cite{categories}, Definition 6.1.1) of finite-dimensional vector spaces. This is the category obtained by formally adding cofiltered limits to $\Vect_k^\omega$. By duality, this is equivalent to $\Vect_k ^\op$, so its objects are of the form $\prod_I k$ for some set $I$ and $\Hom (\prod_I k, k) = \bigoplus_I k$. The product $\prod_I k$ carries a profinite structure by writing it as $\varprojlim _{J \subset I} \bigoplus_J k$, where the limit is taken over all finite subsets of $I$.

	The category $\Solid_k ^\heartsuit$ of \emph{ultrasolid $k$-modules} is the ($1$-categorical) sifted cocompletion (\cite{sifted}, \S 2) of $\ProVect$. That is, $\Solid_k^\heartsuit$ is the category obtained by freely adding filtered colimits and reflexive coequalizers to $\Pro(\Vect_k^\omega)$. This can be identified with finite-product-preserving functors $(\ProVect)^{\op} \to \Vect_k$ that satisfy a left Kan extension condition (\cref{defi-ultrasolid}).
\end{defi}

\begin{rem}
	We will later see that every object in $\ProVect$ being projective implies that $\Solid_k^\heartsuit$ is the $\Ind$-completion of $\ProVect$ (\cref{lemma-ultrasolidind}).
\end{rem}

Our goal will be to develop this theory of ultrasolid modules and show settings where it can be useful.

\begin{example}
	The category of ultrasolid modules can be endowed with a symmetric monoidal structure. This tensor product should be thought of as some kind of completed tensor product. This is exhibited, for example, by the fact that
	\begin{align*}
		k[[x]] \otimes k[[y]] \cong k[[x,y]]
	\end{align*}
	where the tensor product is taken in the category of ultrasolid modules (\cref{section-tensorproduct}). This is not the case if we work with $k$-algebras, as the power series $\sum_i x^i y^i$ belongs to the right-hand side but not the left-hand side.
\end{example}

This tensor product leads to the notion of ultrasolid $k$-algebras and their derived variants: $\EE_\infty$ ultrasolid $k$-algebras and animated ultrasolid $k$-algebras. We can then talk about categories of modules and perform several constructions from deformation theory. 

We will write $\CAlg_k^\blacksquare$ for the $\infty$-category of $\EE_\infty$ ultrasolid $k$-algebras; that is, chain complexes of ultrasolid modules equipped with a multiplication that is associative and commutative up to coherent homotopy.

Additionally, let $\CAlg_k^\blacktriangle$ be the $\infty$-category of animated ultrasolid $k$-algebras. The more classical setup of animated $k$-algebras goes back to the work of Quillen on simplicial commutative rings (\cite{quillen}), and in this paper we follow a more modern treatment with $\infty$-categories, akin to that in (\cite{SAG}, \S 25). Animated ultrasolid $k$-algebras are the $\infty$-categorical sifted cocompletion (\cite{HTT}, \S 5.5.8) of the category spanned by the free ultrasolid $k$-algebras on profinite vector spaces. 

We have many results that mirror the classical theory:

\begin{theorem}
	Let $k$ be a field.
	\begin{enumerate}
		\item There is a conservative, monadic and comonadic functor $\CAlg_k^\blacktriangle \to \CAlg_k^\blacksquare$, $A \mapsto A^\circ$.
		\item Given an $n$-connective map $A \to B$ in $\CAlg_k^\blacktriangle$, there is an $(n+2)$-connective map $L_{B^{\circ}/A^{\circ}} \to L_{B/A}$ between the topological and the algebraic cotangent complex.
		\item Let $A \in \CAlg_k^\blacktriangle$. Then, every step in the Postnikov tower 
			\begin{align*}
				\dots \to \tau_{\leq 2} A \to \tau_{\leq 1} A \to \tau_{\leq 0} A
			\end{align*}
		is a square-zero extension. The same holds for the spectral variant.
	\end{enumerate}
\end{theorem}

We will consider \emph{complete profinite animated ultrasolid $k$-algebras}, which are a generalisation of complete local Noetherian $k$-algebras. Remember that a complete local Noetherian animated $k$-algebra is an augmented animated $k$-algebra $R$ such that $\pi_0 (R)$ is complete, local and Noetherian, and $\pi_i (R)$ is a finitely generated $\pi_0 (R)$-module for all $i$.

\begin{defi}
	A \emph{complete profinite animated ultrasolid $k$-algebra} is an augmented object $R \in \CAlg_{k/ /k}^\blacktriangle$ such that $\pi_0 (R)$ is complete with respect to its augmentation ideal and $\pi_i (R) \in \Solid_k^\heartsuit$ is a profinite vector space for all $i$.
\end{defi}

\begin{rem}
	By $\pi_0 (R)$ being complete we mean that if $\mathfrak{m}$ is its augmentation ideal, then $\pi_0 (R) \cong \varprojlim \pi_0 (R)/\mathfrak{m}^n$. We will see later that we can take this inverse limit either in ultrasolid modules or chain complexes and the result is the same due to inverse limits being exact in profinite vector spaces (\cref{prop-complexesprohomology}).
\end{rem}

\begin{example}
	An example of a complete profinite ultrasolid animated $k$-algebra is $\widehat{\LSym}^* (\prod_\NN k)$. This consists of all possible power series on variables $x_1,x_2, x_3, \dots $ (in particular, this ring includes power series of the form $\sum_n x_n$).

	This is not complete local Noetherian. Indeed, we have the infinite chain of ideals
	\begin{align*}
		(x_1) \subsetneq (x_1, x_2) \subsetneq (x_1, x_2, x_3) \subsetneq \dots
	\end{align*}
\end{example}

These complete profinite objects satisfy many nice properties:

\begin{prop}
	Let $\widehat{\CAlg}_{k/ /k}^\blacktriangle$ be the $\infty$-category of complete profinite animated ultrasolid $k$-algebras.
	\begin{enumerate}
		\item There is a fully faithful embedding from the $\infty$-category of complete local Noetherian animated $k$-algebras into $\widehat{\CAlg}_{k/ /k}^\blacktriangle$. Its essential image are those objects $R$ such that $\cot (R) := k \otimes_R L_{R/k}$  satisfies that $\pi_i (\cot(R))$ is a finite-dimensional $k$-vector space for all $i$.
		\item The $\infty$-category $\widehat{\CAlg}_{k/ /k}^\blacktriangle$ is equivalent to the pro-completion of augmented Artinian animated $k$-algebras.
		\item Let $A \to B$ be a map in $\widehat{\CAlg}_{k/ /k}^\blacktriangle$. If $L_{B/A}$ vanishes, then $A \to B$ is an equivalence.
	\end{enumerate}
\end{prop}

Remember that an (equal characteristic) \emph{formal moduli problem} is a functor $\catC_{\art} \to \Space$ that preserves certain pullbacks and satisfies $X(k) \simeq *$, where $\catC_{\art}$ is the $\infty$-category of augmented Artinian animated $k$-algebras and $\Space$ is the $\infty$-category of spaces (c.f. DAGX, Definition 1.1.14). Notice that the objects in $\catC_{\art}$ don't carry a condensed or ultrasolid structure.

We call a formal moduli problem \emph{corepresentable} if it is equivalent to $\Map(R,-)$ for some $R \in \CAlg_{k/ /k}^\blacktriangle$. We obtain the following extension of the Lurie-Schlessinger criterion in equal characteristic.

\begin{theorem}
	A formal moduli problem is corepresentable by a complete profinite ultrasolid $k$-algebra if and only if its tangent fibre is coconnective.
\end{theorem}

\begin{rem}
	It was already known that any equal characteristic formal moduli problem with coconnective tangent fibre is corepresentable by a pro-Artinian $k$-algebra (\cite{DAGX}, Corollary 2.3.6 and \cite{galatius}, Theorem 4.33). Our framework gives a new description of this pro-completion.
\end{rem}

We have the following layout for this paper:

\begin{itemize}
	\item In \cref{section-basicproperties} we perform some basic constructions with ultrasolid modules, endow them with a symmetric monoidal structure and study the $\infty$-category of chain complexes of ultrasolid $k$-modules.
	\item In \cref{section-commalg} we establish some commutative algebra results in the ultrasolid setting, like a variant of Nakayama's lemma, and discuss coherence of ultrasolid $k$-algebras.
	\item In \cref{section-EEinfty} and \cref{section-animated} we construct the spectral and derived variants of ultrasolid $k$-algebras. This includes some results about the cotangent complex, as well as a comparison between the spectral and derived ultrasolid theory.
	\item In \cref{section-fmp} we finally prove our generalisation of the Lurie-Schlessinger criterion.
\end{itemize}

\section*{Acknowledgements}

We are immensely grateful to Lukas Brantner for his guidance throughout this project and all of his helpful feedback. We would also like to thank Dustin Clausen, for first suggesting working with the category of ultrasolid modules and his very insightful advice.

We wish to thank Juan Esteban Rodríguez Camargo, whose help in computing the cotangent complex of completed free algebras was essential for this project; as well as Rhiannon Savage, for our lively discussions about the more difficult parts of this paper, and Adrià Marín Salvador, for discussing every aspect of this project from its beginning.

Finally, we want to express our gratitude to the Mathematical Institute, University of Oxford, for their funding and support.

\section{Basic properties of ultrasolid modules}
\label{section-basicproperties}

\subsection{Sifted cocompleting a large category}

Sifted cocompletions are only known to exist for small categories. By (\cite{sifted}, Corollary 2.8), if $\mathcal{A}$ is a small category admitting finite coproducts, its sifted cocompletion can be described as finite-product-preserving functors $\mathcal{A}^{\op} \to \Set$. However, if $\mathcal{A}$ is large, this notion does not necessarily work. We can fix this by writing $\mathcal{A}$ as a suitable filtered colimit.

Many of the arguments in this section are paraphrased from \cite{condensed} and suitably adapted to our situation.

\begin{defi}
	Let $\sigma$ be an uncountable strong limit cardinal. Write $\Pro(\Vect_k^\omega)_{< \sigma} \subset \ProVect$ for the full subcategory consisting of objects of the form $\prod_I k$ with $|I| < \sigma$.

	The category $\Solid_k^{\heartsuit, \sigma}$ of \emph{$\sigma$-ultrasolid $k$-modules} is the sifted cocompletion of $\ProVectsigma$.

	That is, $\Solid_k ^{\heartsuit, \sigma}$ can be identified with functors $(\ProVectsigma)^\op \to \Vect_k$ that preserve finite products.
\end{defi}

We now see that we could work with $\Ind$-completions instead of sifted cocompletions.

\begin{lemma}
	\label{lemma-ultrasolidind}
	Let $\sigma$ be an uncountable strong limit cardinal. Then, there is an equivalence
	\begin{align*}
		\Ind(\ProVectsigma) \cong \Solid_k^{\heartsuit, \sigma}
	\end{align*}
\end{lemma}

\begin{proof}
	By (\cite{HTT}, Proposition 5.3.5.11), we only need to show that $\ProVectsigma$ generates all of $\Solid_k^{\heartsuit, \sigma}$ under filtered colimits. Let $X : \ProVectsigma ^{\op} \to \Vect_k$ be a $\sigma$-ultrasolid $k$-module. Then, for $V \in \ProVectsigma$, by the Yoneda embedding, we can identify $X(V)$ with maps $V \to X$ (here we are identifying $V$ with its image in $\sigma$-ultrasolid $k$-modules under the Yoneda embedding). Then, it is clear that
	\begin{align*}
		X \cong \colim_{V \to X} V
	\end{align*}
	where the colimit runs over all $V \in \ProVectsigma$ with a map $V \to X$. Hence, it suffices to show that this category is filtered. Given a finite diagram $\{V_i \to X\}$, where each $V_i \in \ProVectsigma$, we can let $W := \colim V_i$, where the colimit is taken in the category of $\sigma$-ultrasolid $k$-modules. Since every object in $\ProVectsigma$ is projective, the Yoneda embedding commutes with finite colimits that exist in $\ProVectsigma$, so $W \in \ProVectsigma$. This concludes the proof.
\end{proof}

Clearly, each category $\Solid_k^{\heartsuit, \sigma}$ is abelian, complete and cocomplete, since we can just take limits and colimits pointwise (this is possible due to finite products commuting with all limits and colimits in $\Vect_k$).

\begin{rem}
	If $\sigma < \sigma'$ we have a pair of adjoint functors $\Solid_k^{\heartsuit, \sigma} \leftrightarrows \Solid_k ^{\heartsuit, \sigma'}$. The right adjoint is given by restriction and the left adjoint is given by left Kan extension from $\ProVectsigma^{\op}$.
\end{rem}

\begin{defi}
	\label{defi-ultrasolid}
	The category $\Solid_k^\heartsuit$ of \emph{ultrasolid $k$-modules} is the filtered colimit of $\Solid_k^{\heartsuit, \sigma}$ over all uncountable strong limit cardinals $\sigma$.

	Equivalently, $\Solid_k^\heartsuit$ is the category of finite-product-preserving functors $V: (\ProVect)^\op \to \Vect_k$ such that $V$ is the left Kan extension of its own restriction to $\ProVectsigma^{\op}$ for some uncountable strong limit cardinal $\sigma$.
\end{defi}

Hence, for $V, W \in \Solid_k^\heartsuit$,
\begin{align*}
	\Hom_{\Solid_k ^\heartsuit} (V,W) = \Hom _{\Solid_k^{\heartsuit, \sigma}} (V_{|\sigma}, W_{ | \sigma})
\end{align*}
for some uncountable strong limit cardinal $\sigma$.

We now check that profinite vector spaces do embed into our construction.

\begin{lemma}
	The Yoneda embedding gives a well-defined functor $\ProVect \to \Solid_k ^\heartsuit$.
\end{lemma}

\begin{proof}
	One only needs to check the left Kan extension condition. This amounts to showing that if $V \in \ProVect$, there is a fixed cardinal $\sigma$ such that for any map $W \to V$ of profinite vector spaces, there exists $W' \in \ProVectsigma$ such that the map factorises $W \to W' \to V$. Any $\sigma > \dim V$ works, where we define the dimension of a profinite vector space to be the dimension of its dual.
\end{proof}

\begin{prop}
	The category $\Solid_k^\heartsuit$ of ultrasolid $k$-modules is abelian, complete and cocomplete, satisfying the same Grothendieck axioms as $\Vect_k$. Additionally, all small limits and colimits can be computed pointwise.
\end{prop}

\begin{proof}
	We will prove that limits and colimits are computed pointwise. Then, the statements about the category being abelian and the Grothendieck axioms follow from the equivalent results about $\Vect_k$.

	The statements about colimits are obvious because if $\{F_i : \ProVect^{\op} \to \Vect_k\}_{i \in I}$ are left Kan extended from $\ProVectsigma^{\op}$, then $\colim F_i$ is left Kan extended from $\ProVectsigma^{\op}$. Hence, we can just compute them pointwise to obtain the desired colimit.

	For limits, let $\sigma < \sigma'$ be uncountable strong limit cardinals, and let $\lambda$ be the cofinality of $\sigma$. We will show that the left Kan extension functor $\Solid_k^{\heartsuit, \sigma} \to \Solid_k ^{\heartsuit, \sigma'}$ commutes with $\lambda$-small limits.

	Let $X \in \Solid_k^{\heartsuit, \sigma}$ and let $X' \in \Solid_k^{\heartsuit, \sigma'}$ be its left Kan extension. Then, for $V \in \Pro(\Vect_k^\omega)_{< \sigma'}$,
	\begin{align*}
		X'(V) = \colim _{V \to W} X(W)
	\end{align*}
	where the colimit runs over all $W \in \Pro(\Vect_k^\omega)_{< \sigma}$ with a map from $V$ (\cite{maclane}, \S X, Corollary 4). Given that $\lambda$-small limits commute with $\lambda$-filtered colimits in $\Vect _k$ (\cite{adamek_rosicky}, Theorem 1.59) and hence in $\Solid_k^{\heartsuit, \sigma'}$, we only need to show that the category of all $W \in \Pro(\Vect_k^\omega)_{< \sigma}$ with a map $V \to W$, where $V \in \Pro(\Vect_k^\omega)_{< \sigma'}$, is $\lambda$-cofiltered.

	This reduces to showing that if for some $\lambda$-small index category $I$, we have a diagram $\{W_i\}_{i \in I}$ of $\sigma$-small profinite vector spaces with compatible maps $V \to W_i$ then there is a $\sigma$-small profinite vector space $U$ with a map $V \to U$ over which all $V \to W_i$ factor compatibly. It suffices to choose $U = \lim W_i$. This is a subspace of $\prod_I W_i$ which has dimension $< \sigma$ due to our choice of $\lambda$.

	Hence, we see that if we have a $\sigma$-ultrasolid vector space, a pointwise limit will still be the left Kan extension from some $\ProVectsigma$. This shows that ultrasolid modules have small limits and they can be computed pointwise.
\end{proof}

\begin{rem}
	Given that limits and colimits are computed pointwise, we have that a map $V \to W$ in $\Solid_k^\heartsuit$ is a monomorphism/epimorphism if and only if for all $U \in \ProVect$, the map $V(U) \to W(U)$ is injective/surjective. Hence, we will often talk about injections or surjections of ultrasolid $k$-modules.
\end{rem}

This gives us a category that is the sifted cocompletion of $\Pro (\Vect_k^\omega)$.

\begin{cor}
	Let $\catC$ be a $1$-category that admits sifted colimits. Then, restriction gives an equivalence
	\begin{align*}
		\Fun_\Sigma (\Solid_k ^\heartsuit, \catC) \xrightarrow{\simeq} \Fun (\Pro (\Vect_k ^\omega), \catC)
	\end{align*}
	where $\Fun_\Sigma$ refers to the full subcategory of functors that preserve sifted colimits. The inverse is given by left Kan extension.

	The Yoneda embedding $\Pro (\Vect_k ^\omega) \to \Solid_k ^\heartsuit$ commutes with all small limits and finite colimits. Its essential image are precisely the compact objects, which are also projective.
\end{cor}

\begin{proof}
	For the first part, we can just take the filtered colimit of the usual result with sifted cocompletions of small categories (\cite{sifted}, Corollary 2.8). That is, for an uncountable strong limit cardinal $\sigma$, we have an equivalence
	\begin{align*}
		\Fun_\Sigma (\Solid_k^{\heartsuit, \sigma}, \catC) \cong \Fun (\ProVectsigma, \catC)
	\end{align*}
	The result now follows by noting that $\Solid_k^\heartsuit = \colim \Solid_k^{\heartsuit, \sigma}$ and $\ProVect = \colim \ProVectsigma$, so we can take filtered colimits on each category of functors. That is,
	\begin{align*}
		\Fun_\Sigma (\Solid_k^{\heartsuit}, \catC) \cong \varprojlim \Fun_\Sigma (\Solid_k^{\heartsuit, \sigma}, \catC) \cong \varprojlim \Fun (\ProVectsigma, \catC) \cong \Fun(\ProVect, \catC)
	\end{align*}

	Remember that we can think of ultrasolid $k$-vector spaces as finite-product-preserving functors $\Pro (\Vect_k^\omega)^\op \to \Vect_k$. For $V \in \ProVect$, the functor $\Hom_{\Solid_k^\heartsuit} (V, -)$ is precisely evaluation at $V$. Since all limits and colimits are computed pointwise, it follows that all profinite vector spaces are compact projective.

	The category $\Pro(\Vect_k^\omega)$ is obviously closed under all limits. It is also closed under finite colimits because every object in $\Pro (\Vect_k^\omega)$ is injective so the Yoneda embedding commutes with reflexive coequalizers. Also, since $\Pro (\Vect_k ^\omega)$ is idempotent complete, the image of the Yoneda embedding is closed under retracts.

	We must now show that every compact object is of this form. Let $V \in \Solid_k^\heartsuit$ and suppose that $V$ is left Kan extended from $\Pro(\Vect_k^\omega)^{\op}_{< \sigma}$ for some $\sigma$. Then,
	\begin{align*}
		V = \colim _{W \to V} W
	\end{align*}
	where the colimit is taken over all $W \in \ProVectsigma$ with a map $W \to V$.

	This category is obviously filtered since given any finite collection of maps $\{W_i \to V\}$ we can assemble them into a map $\colim W_i \to V$. Hence, if $V$ is compact, the identity map $V \to V$ factors through some $W  \in \Pro(\Vect_k^\omega)_{< \sigma}$, so it is a retract of $W$. We are now done since this category is closed under retracts.
\end{proof}

\subsection{Comparison with condensed vector spaces}

When $k = \QQ$ or $\FF_p$, there is a well-behaved theory of solid $k$-modules that fits in the condensed framework. Solid $k$-modules can be described as the Bousfield localisation at the maps $k[S] \to k[S]^\blacksquare := \varprojlim k[S_i]$, where $S = \varprojlim S_i$ is profinite. They then form a reflexive subcategory of condensed $k$-modules that is stable under all limits and colimits so there is a pair of adjoint functors $\Cond(\Vect_k) \leftrightarrows \Solid_k^\heartsuit$ (\cite{condensed}, Proposition 7.5 and Theorem 8.1).

For more general fields $k$, we can't quite fit this in the same way within $\Cond(\Vect_k)$ but it is possible to still obtain part of this comparison.

\begin{prop}
	There is a fully faithful embedding
	\begin{align*}
		\Pro (\Vect_k^\omega) \to  \Cond(\Vect_k)
	\end{align*}
\end{prop}

\begin{proof}
	This map is given by the inverse-limit-preserving extension that is the identity on finite-dimensional vector spaces. We now need to compute the relevant $\Hom$-sets in condensed vector spaces. It is enough to show that
	\begin{align*}
		\Hom_{\Cond(\Vect_k)} (\prod_I k, k) \cong \bigoplus _I k
	\end{align*}
	We need to check that any map $\prod_I k \to k$ factors through $\prod_J k \to k$ where $J \subset I$ is finite. Using the left adjoint of the inclusion from topological spaces to condensed sets (\cite{condensed}, Proposition 1.7),
	\begin{align*}
		\Hom _{\Cond(\Set)} (\prod_I k, k) = C\left( ( \prod_I k ) ^\sim , k \right)
	\end{align*}
	where the superscript $\sim$ means the $k$-ification of a space (a set is open in $(\prod_I k)^\sim$ if and only if its intersection with every compact subset of $\prod_I k$ is open). 

	Hence, maps of condensed $k$-modules $\prod_I k \to k$ are in one-to-one correspondence with maps of vector spaces $\prod_I k \to k$ that are continuous where $\prod_I k$ is has as its topology the $k$-ification of the product topology.

	Suppose we have a continuous map $f: (\prod_I k)^\sim \to k$. Then, the kernel is an open set containing $0$. The space $\prod_I \{0,1\}$ is compact. Since $\ker f$ is open, $\ker f \cap \prod_I \{0,1\}$ is open in $\prod_I \{0,1\}$ and contains $0$. Then, there is some finite set $J \subset I$ such that $\prod_{I \setminus J} \{0,1\} \subset \ker f$. We claim that $\prod_{I \setminus J} k \subset \ker f$.

	Firstly, it is clear that $\bigoplus_{I \setminus J} k \subset \ker f$ by extending linearly. Let $(a_i)_{i \in I} \in \prod_{I \setminus J} k$, so $a_j = 0$ for $j \in J$. We now consider the compact set
	\begin{align*}
		L := \left( \prod_J \{0\} \right) \times \left( \prod_{I \setminus J} \{0,a_i\} \right)
	\end{align*}
	Its intersection with the kernel is nonempty and open. Again by looking at basic opens, we see that there is some $(b_i)_{i \in I} \in \ker f$ with $b_j = 0$ for $j \in J$ and $b_i \neq a_i$ for finitely many $i$. But then $(a_i) - (b_i) \in \bigoplus_{I \setminus J} k \subset \ker f$. This implies $(a_i)_{i \in I} \in \ker f$.

	Hence, $\prod_I k \to k$ factors through $\prod_J k$ for some finite $J \subset I$, as required.
\end{proof}

\begin{construction}[Functor from ultrasolid modules to condensed $k$-modules]
	The above implies that there is a map $\phi: \Solid_k^\heartsuit \to \Cond(\Vect_k)$ by extending sifted colimits. We can also describe it as follows. There is a functor from profinite sets to profinite vector spaces defined as
	\begin{align*}
		\varprojlim S_i \mapsto \varprojlim k[S_i] 
	\end{align*}
	so composition gives a functor $\Solid_k^\heartsuit \to \Cond(\Vect_k)$. Limits and colimits are computed pointwise in both functor categories, so this functor commutes with small limits and colimits. In particular, it commutes with sifted colimits so it is left Kan extended from $\ProVect$. Since it coincides with $\phi$ on $\ProVect$, we see that they are the same functor.

	This functor is also conservative, so it is monadic and comonadic (\cite{categories}, Theorem 4.3.8). However, the author doesn't know a description of this monad or comonad, or whether the functor $\Solid_k^\heartsuit \to \Cond(\Vect_k)$ is fully faithful.
\end{construction}

\begin{rem}
	When $k = \FF_p$ or $\QQ$, the functor $\Solid_k^\heartsuit \to \Cond(\Vect_k)$ is fully faithful and its image are precisely solid $k$-modules.

	To see this, remember that the preanalytic ring $\varprojlim S_i \mapsto \varprojlim k[S_i]$ is analytic (\cite{condensed}, Theorem 8.13), so the category of solid $k$-modules is generated under colimits by profinite vector spaces, which are all compact (\cite{condensed}, Proposition 7.5). By \cref{lemma-ultrasolidind} and (\cite{HTT}, Proposition 5.3.5.11), the functor from ultrasolid $k$-modules to condensed $k$-modules is fully faithful, and to check that its image are solid $k$-modules we only need to check that profinite vector spaces generate under filtered colimits. For a solid $k$-module $X$, we can write
	\begin{align*}
		X \cong \colim _{V \to X} X	
	\end{align*}
	where the colimit is taken over all $V \in \ProVect$ with a map $V \to X$ (restricting to a suitable cardinality). We only need to check that this category is filtered, but this is true since any finite colimit of profinite vector spaces is again a profinite vector space.
\end{rem}

\begin{construction}[The solidification functor]
	Write $\catC_0 \subset \Cond(\Vect_k)$ for the full subcategory of condensed vector spaces spanned by objects of the form $k[S]$ for $S$ extremally disconnected. Then, for extremally disconnected $S,T$ we have that
	\begin{align*}
		\Hom_{\Cond(\Vect_k)} (k[S], k[T]) \cong k[T] (S) = \colim _{\sqcup S_i = S} \prod_i k[C(S_i, T)]
	\end{align*}
	where the colimit runs over all decompositions of $S$ into finitely many disjoint clopens, ordered by refinement, and $C(S_i, T)$ is the set of continuous functions $S_i \to T$. The condensed $k$-module $k[S]$ is the sheafification of the functor $S \mapsto k[C(S,T)]$.

	We now construct a functor $\catC_0 \to \Solid_k^\heartsuit$, by mapping $k[S] \mapsto C(S,k)^\vee \in \ProVect \subset \Solid_k^\heartsuit$. Now for any map $k[S] \to k[T]$ we need to construct a map $C(S,k)^\vee \to C(T,k)^\vee$. By the above characterization of the mapping sets in $\catC_0$, it is enough to construct compatible such maps for any continuous map $S_i \to T$, where $S_i \subset S$ is clopen.

	Given a map $S_i \to T$, we have a map $C(T,k) \to C(S,k)$ as follows: $f \in C(T,k)$ is mapped to $\tilde{f}$, where $\tilde{f}$ maps $S \setminus S_i$ to $0$, and maps $S_i$ as $S_i \to T \xrightarrow{f} k$. By taking duals, we get the desired map $C(S,k)^\vee \to C(T,k)^\vee$.

	We have constructed a functor $\catC_0 \to \Solid_k^\heartsuit$, and by left Kan extending we obtain a sifted-colimit-preserving functor $\Cond(\Vect_k) \to \Solid_k^\heartsuit$, which we will call \emph{solidification}.
\end{construction}

\subsection{Completed tensor product}
\label{section-tensorproduct}

We now describe a symmetric monoidal structure on ultrasolid $k$-vector spaces.

\begin{prop}
	The category $\Solid_k^\heartsuit$ is symmetric monoidal with a tensor product that commutes with small colimits in each variable. It is given on compact objects by
	\begin{align*}
		\prod_I k \otimes \prod_J k = \prod_{I \times J} k
	\end{align*}
\end{prop}

\begin{proof}
	Remember that $\Pro (\Vect_k ^\omega) \cong \Vect_k^\op$, and the functor defined on compact objects is precisely the one inherited from vector spaces. We can now extend this functor in the unique sifted-colimit-preserving way to obtain a bifunctor $- \otimes - : \Solid_k ^\heartsuit \times \Solid_k ^\heartsuit \to \Solid_k ^\heartsuit$.

	Hence, the tensor product commutes with sifted colimits on each variable, so we only need to check that it commutes with finite sums. It clearly commutes with finite sums on $\Pro (\Vect_k ^\omega)$, and this will be preserved by left Kan extension.
\end{proof}

\begin{prop}
	Given $V \in \Solid_k^\heartsuit$, there is an adjunction
	\begin{align*}
		- \otimes V : \Solid_k^\heartsuit \leftrightarrows \Solid_k^\heartsuit : \condHom (V, -)
	\end{align*}
	such that
	\begin{align*}
		\condHom (V, W) (\prod_I k) = \Hom (\prod_I k \otimes V, W)
	\end{align*}
\end{prop}

\begin{proof}	
	The tensor product commutes with small colimits in each variable so by the adjoint functor theorem (\cite{HTT}, Corollary 5.5.2.9), we have such an adjunction in $\Solid_k^{\heartsuit, \sigma}$ for any uncountable strong limit cardinal $\sigma$ (we are doing this one cardinal at a time so that our categories are presentable and we can apply the adjoint functor theorem). Taking filtered colimits, we get the desired right adjoint.

	The rest follows by the adjunction and the Yoneda embedding.
\end{proof}

Although the category of ultrasolid modules is more complicated than the category of vector spaces, we have similar behaviour with the tensor product.

\begin{lemma}
	Every $U \in \ProVect$ is flat as an object in $\Solid_k^{\heartsuit}$.
\end{lemma}

\begin{proof}
	Fix $U \in \ProVect$. Suppose we have a map of ultrasolid modules $V \to W$. Without loss of generality, we may restrict ourselves to some cardinal $\sigma$. We have that $\Solid_k^{\heartsuit, \sigma} \cong \Ind (\ProVectsigma)$ (\cref{lemma-ultrasolidind}), and by (\cite{HTT}, Proposition 5.3.5.15) $\Fun (\Delta^1, \Solid_k^{\heartsuit, \sigma}) \simeq \Ind (\Fun(\Delta^1, \ProVectsigma))$. Then, we may write $V \to W$ as a filtered colimit of maps $V_i \to W_i$ where each $V_i, W_i \in \ProVectsigma$. Given that finite limits commute with filtered colimits,
	\begin{align*}
		\ker (V \to W) \cong \colim \ker (V_i \to W_i)
	\end{align*}
	Since $U$ is flat in $\Pro(\Vect_k^\omega)$ and the tensor product commutes with filtered colimits, we get that
	\begin{align*}
		\ker (U \otimes V \to U \otimes W) & \cong \colim_i \ker (U \otimes V_i \to U \otimes W_i)	\cong \colim_i U \otimes (\ker (V_i \to W_i)) \\
		& \cong U \otimes \colim_i \ker (V_i \to W_i) \cong U \otimes \ker (V \to W)
	\end{align*}
\end{proof}

\begin{prop}
	\label{prop-flat}
	Every ultrasolid module is flat.
\end{prop}

\begin{proof}
	Every ultrasolid module can be written as a filtered colimit of profinite vector spaces, all of which are flat. Hence, it suffices to see that filtered colimits of flat objects are flat. But this follows from the fact that the tensor product commutes with colimits and filtered colimits are exact.
\end{proof}

\begin{example}
	This completed tensor product is well-behaved with respect to profinite constructions. For example, we can consider the ultrasolid $k$-algebra (see \cref{section-rings}) $k[[x]]$ of formal power series in one variable. The underlying ultrasolid module is $\prod_{n \geq 0} k$.

	One can now see that $k[[x]] \otimes k[[y]] = k[[x,y]]$. This is because we can write $k[[x]] = \prod_{i \geq 0} k x^i $, so that $k[[x]] \otimes k[[y]] = \prod_{i,j \geq 0} k x^i y^j = k[[x,y]]$. This equality is not true if one just takes the usual tensor product of algebras, since the power series $\sum_n x^n y^n$ would be on the right-hand side but not on the left-hand side.
\end{example}

We finally establish some nice properties of the subcategory $\Pro(\Vect_k^\omega) \subset \Solid_k^\heartsuit$.

\begin{prop}
	\label{prop-profinite}
	The category of profinite modules $\Pro(\Vect_k^\omega)$ is closed under all small limits and finite colimits, and inverse limits are exact in this category.

	Additionally, the tensor product commutes with limits in each variable when restricted to profinite modules.
\end{prop}

\begin{proof}
	The first part is clear. To see that inverse limits are exact, notice that $\Pro(\Vect_k^\omega) \cong \Vect_k^{\op}$, so the result follows from the fact that filtered colimits of vector spaces are exact.

	For the last part, it is easy to see that the monoidal structure on $\Pro(\Vect_k^\omega)$ is induced by the one on $\Vect_k^{\op}$, so the result follows from the fact that the tensor product of vector spaces commutes with all colimits in each variable.
\end{proof}

\subsection{Chain complexes of ultrasolid modules}

The category of ultrasolid $k$-modules admits enough projectives, given that for any $V \in \Solid_k^\heartsuit$ we have a surjection $\bigoplus_{\prod_I k \to V} \prod_I k \to V$ (restricting to a suitable cardinality). We can then form the derived category $\Solid_k$, which we will always treat as an $\infty$-category (\cite{HA}, Definition 1.3.5.8).

We will also write $\Solid_k^\sigma$ for the derived category of $\Solid_k^{\heartsuit, \sigma}$. That is, the $\infty$-category of chain complexes of $\sigma$-ultrasolid $k$-modules.

\begin{rem}
	\label{rem-Modk}
	Write $\Mod_k$ for the $\infty$-category of $k$-module spectra, which can also be identified with chain complexes of $k$-modules (\cite{HA}, Proposition 7.1.1.15). Then, there is a fully faithful symmetric monoidal colimit-preserving functor $\Mod_k \to \Solid_k$, given by $k[0] \mapsto k[0]$. This is true on connective chain complexes by (\cite{HTT}, Proposition 5.5.8.22) and then we can stabilise.

	Alternatively, we can also describe $\Mod_k$ as finite-product-preserving functors $(\Vect_k^\omega)^{\op} \to \Sp$, where $\Sp$ is the $\infty$-category of spectra.
\end{rem}

For connective chain complexes we have the following, which is a consequence of (\cite{HA}, Proposition 1.3.3.14) and (\cite{HTT}, Proposition 5.5.8.15) by taking filtered colimits. For a reference on $\infty$-categorical sifted cocompletions, also known as the $\mathcal{P}_\Sigma$ construction, see (\cite{HTT}, \S 5.5.8).

\begin{prop}
\label{prop-ultrasolidPSigma}
	Let $\sigma$ be an uncountable strong limit cardinal. Then, there is an equivalence of $\infty$-categories
	\begin{align*}
		\mathcal{P}_\Sigma (\ProVectsigma) \simeq \Solid_{k, \geq 0}^\sigma
	\end{align*}

	Let $\catC$ be an $\infty$-category with sifted colimits. Then, there is an equivalence of $\infty$-categories
	\begin{align*}
		\Fun_\Sigma (\Solid_{k, \geq 0}, \catC) \xrightarrow{\simeq} \Fun (\Pro(\Vect_k), \catC)
	\end{align*}
	where $\Fun_\Sigma$ refers to the full subcategory of sifted-colimit-preserving functors. The inverse is given by left Kan extension.
\end{prop}

\begin{rem}
	\label{rem-Solidasfunctors}
	In particular, we can identify $\Solid_{k, \geq 0}$ with finite-product-preserving functors $(\Pro(\Vect_k^\omega))^{\op} \to \Space$ that are the left Kan extension from $(\ProVectsigma)^{\op}$ for some uncountable strong limit cardinal $\sigma$.

	Then, $\Solid_k \simeq \Sp (\Solid_{k, \geq 0})$ is the stabilization of $\Solid_{k, \geq 0}$ (see \cite{HA}, \S 1.4.2 for more on stabilization), so it consists of finite-product-preserving functors $(\ProVect)^{\op} \to \Sp$ that are left Kan extended from $(\ProVectsigma)^{\op}$ for some uncountable strong limit cardinal $\sigma$.  

	The $k$-linear structure on $\ProVect$ means that if $X : (\ProVect)^{\op} \to \Sp$ is an object in $\Solid_k$, the functor $X$ lands on $k$-module spectra. To see this, we can associate to $X$ a functor $(\ProVect)^{\op} \to \Fun ((\Vect_k^\omega)^{\op}, \Sp)$ given by $V \mapsto X(V \otimes -)$. It is clear that each functor $X(V \otimes -)$ preserves finite products, so by \cref{rem-Modk} we just constructed a functor $\tilde{X} : (\ProVect)^{\op} \to \Mod_k$ given by $\tilde{X} (V) = X(V \otimes -)$. It is clear that we obtain $X$ by forgetting $\tilde{X}$ from $k$-module spectra to spectra. Hence, $X$ can be promoted to a finite-product-preserving functor to $k$-module spectra in a canonical way.

	Hence, $\Solid_k$ can be identified with finite-product-preserving functors $(\ProVect)^{\op} \to \Mod_k$ that are left Kan extended from $(\ProVectsigma)^{\op}$ for some $\sigma$.

	Since $\Solid_k$ is a stable $\infty$-category, for $V, W \in \Solid_k$, the mapping space $\Map (V,W)$ can be promoted to a spectrum $\Map^S (V,W)$ with $\Omega^\infty (\Map^S (V,W)) \simeq \Map (V,W)$. The $k$-linear structure on $\Solid_k$ implies that $\Map^S (V,W)$ has the structure of a $k$-module spectrum, which we will write as $\RHom (V,W)$.

	For $X \in \Solid_k$ and $V \in \ProVect$, the $k$-module spectrum $X(V)$ is the mapping spectrum $\RHom (V[0], X)$.
\end{rem}

\begin{prop}
	\label{prop-filteredhomology}
	Filtered colimits commute with homology in $\Solid_k$.
\end{prop}

\begin{proof}
	Let $X \in \Solid_k$, which we can consider as a functor $(\ProVect)^{\op} \to \Mod_k$. Then, the $n$th homology is precisely $\pi_n \circ X : (\ProVect)^{\op} \to \Vect_k$. The result now follows from the fact that homology commutes with filtered colimits in $\Mod_k$, and colimits of ultrasolid modules are computed pointwise.
\end{proof}

Compact objects in the $\infty$-category $\Solid_k$ will be particularly simple.

\begin{defi}
	Write $\Solid_k^{\perf}$ for the full subcategory of $V \in \Solid_k$ such that $\pi_i (V) \in \Pro(\Vect_k^\omega)$ for all $i$ and $\pi_i (V) = 0$ for $|i| \gg 0$. These are the \emph{perfect} ultrasolid $k$-modules.

	Since all of the objects in $\Pro(\Vect_k^\omega)$ are projective, this can be identified with bounded complexes of profinite vector spaces. This is also the smallest stable subcategory of $\Solid_k$ containing $V [0]$ for all $V \in \ProVect$ that is closed under retracts.

	Finally, write $\Solid_k^{\perf, \sigma}$ for the perfect ultrasolid modules all of whose homology is $\sigma$-ultrasolid.

	A chain complex $V$ of ultrasolid modules is \emph{almost perfect} if for every $n$ it admits an $n$-connective map $M \to V$ from a perfect chain complex $M$. We will write $\Solid_k^{\aperf}$ for the full subcategory of almost perfect chain complexes. By definition, any such complex is bounded below and has profinite homology in every degree. Conversely, any bounded below complex with profinite homology is split, since all of its homology is projective, so it is clear that it is almost perfect.
\end{defi}

\begin{rem}
	By taking duals, we can see that $\Solid_k^{\perf}$ is contravariantly equivalent to the $\infty$-category of bounded chain complexes of $k$-vector spaces (since all complexes are split in both categories, the dual is straightforward to compute). Similarly, $\Solid_k^{\aperf}$ is contravariantly equivalent to bounded above chain complexes of $k$-vector spaces.
\end{rem}

\begin{prop}
	\label{prop-complexesInd}
	Let $\sigma$ be an uncountable strong limit cardinal. Then, there is an equivalence of $\infty$-categories
	\begin{align*}
		\Ind(\Solid_k^{\perf, \sigma}) \simeq \Solid^\sigma_k
	\end{align*}
	and an object is compact if and only if it is perfect.
\end{prop}

\begin{proof}
	First we show that if $V \in \ProVectsigma$, then $V[0]$ is compact. By \cref{rem-Solidasfunctors}, for $X \in \Solid_k$, $\RHom (V[0], X) = X(V)$. Filtered colimits of finite-product-preserving functors $(\ProVectsigma)^{\op} \to \Mod_k$ are computed pointwise, so it follows that $\RHom (V[0], -)$ commutes with filtered colimits. Then, we have that $\Omega^{\infty} \circ  \RHom (V[0], -) \simeq \Map (V[0], -): \Solid_k^\sigma \to \Space$, and $\Omega^{\infty}$ commutes with filtered colimits, so $V[0]$ is compact.

	Since compact objects are closed under finite limits, colimits and retracts in a stable $\infty$-category, all perfect objects are compact.

	By (\cite{HTT}, Proposition 5.3.5.11), there is a fully faithful functor $\Ind (\Solid_k ^{\perf, \sigma}) \to \Solid_k^\sigma$ that preserves filtered colimits. To show it is an equivalence, we only need to show that the image generates the entire category under filtered colimits.

	For $V \in \Solid_k$, we can write $V = \colim_{ n} \tau_{\geq n} V$, so we only need to show that perfect objects generate connective chain complexes. We have that $\Solid_{k, \geq 0} \simeq \mathcal{P}_\Sigma (\Pro(\Vect_k^\omega))$ (\cref{prop-ultrasolidPSigma}) and by (\cite{HTT}, Proposition 5.5.8.10) this category is compactly generated by retracts of finite colimits of objects of the form $V[0]$ for $V \in \ProVect$, all of which are perfect. 

	Finally, we must show that any compact object is perfect. For this, if $V \in \Solid_k^\sigma$ is compact, we can write $V$ as a filtered colimit $V \simeq \colim V_i$, where each $V_i$ is perfect. Then, $\Map(V,V) \simeq \colim \Map(V, V_i)$. This means that the identity map $V \to V$ factors through some $V_i$, so $V$ is a retract of a perfect chain complex. We are now done since perfect chain complexes are closed under retracts.
\end{proof}

We can take filtered colimits of the $\infty$-categories of functors on (\cite{HTT}, Proposition 5.3.5.10) to obtain the following.

\begin{cor}
	Let $\catC$ be an $\infty$-category that admits filtered colimits. Then, restriction gives an equivalence of $\infty$-categories
	\begin{align*}
		\Fun_\kappa (\Solid_k , \catC) \xrightarrow{\simeq} \Fun (\Solid_k^{\perf}, \catC)
	\end{align*}
	where $\Fun_\kappa$ refers to the full subcategory of filtered-colimit-preserving functors. The inverse is given by left Kan extension.
\end{cor}

\begin{rem}
	$\Solid_k^\heartsuit$ has a symmetric monoidal structure that is inherited by its model category of chain complexes and then by $\Solid_k$ (\cite{HA}, Proposition 4.1.7.10). We can apply the Künneth spectral sequence (see \cref{lemma-spectralsequence}) and since every ultrasolid module is flat we get that for $A, B \in \Solid_k$,
	\begin{align*}
		\pi_* (A \otimes B) = \bigoplus_{i + j = *} \pi_i (A) \otimes \pi_j (B)
	\end{align*}
	where the tensor product is derived.

	This implies that there is no ambiguity between the derived and the underived tensor product of chain complexes, since they always give the same result.
\end{rem}

Chain complexes with profinite homology are very well-behaved.

\begin{prop}
	\label{prop-complexesprohomology}
	Let $\Solid_k^{\pro} \subset \Solid_k$ be the full subcategory of objects with profinite homology; that is, those $V \in \Solid_k$ with $\pi_i (V) \in \ProVect$ for all $i$. Then,
	\begin{enumerate}
		\item $\Solid_k^{\pro}$ is stable under small limits and finite colimits. 
		\item Inverse limits commute with homology in $\Solid_k^{\pro}$.
		\item The tensor product commutes in each variable with small limits in $\Solid_k^{\pro}$.
		\item $\Solid_{k, \geq 0}^{\pro} \simeq \Solid_{k, \geq 0}^{\aperf}$ is closed under geometric realisations and inverse limits.
	\end{enumerate}
\end{prop}

\begin{proof}
	Any complex in $\Solid_k^{\pro}$ is split and has projective homology, so we can see that for $V \in \Solid_k^{\pro}$ its dual satisfies $V^\vee \simeq \bigoplus_n \pi_n (V)^\vee[-n]$, so $V^\vee$ is in $\Mod_k \subset \Solid_k$ and $\pi_n (V^\vee) = \pi_{-n} (V)^\vee$. We can take duals again to obtain $V$ as its own double dual. Objects being split with projective homology in both $\Solid_k^{\pro}$ and $\Mod_k$ makes the computation of mapping spaces trivial, so we can see that duality gives an equivalence $\Solid_k^{\pro} \simeq \Mod_k^{\op}$. We can check directly on objects that this is symmetric monoidal as well.

	Duality takes colimits to limits, and since $\Mod_k$ is closed under colimits we see that $\Solid_k^{\pro}$ is closed under all limits. It is stable, so it is also closed under finite colimits. Statements $2$ and $3$ now follow from the dual statements in $\Mod_k$.

	For the last part, we saw that $\Solid_k^{\pro}$ is closed under inverse limits and that inverses limits commute with homology in $\Solid_k^{\pro}$, so $\Solid_{k, \geq 0}^{\pro}$ is closed under inverse limits. Let $V_\bullet$ be a simplicial object in $\Solid_{k, \geq 0}^{\aperf}$. Then, $\tau_{\leq n} |V_\bullet| = \tau_{\leq n} \colim_{\Delta_{\leq n+1}^{\op}} V_\bullet$ (\cite{HTT}, Lemma 6.5.3.10). Since $\Solid_{k, \geq 0}^{\pro}$ is closed under finite colimits, we see that $\colim_{\Delta_{\leq n+1}^{\op}} V_\bullet \in \Solid_{k, \geq 0}^{\pro}$. This is true for every $n$ so it has profinite homology and $|V_\bullet| \in \Solid_{k, \geq 0}^{\pro}$.
\end{proof}

The dual statement of the last part gives us that $\Mod_{k, \leq 0} \subset \Solid_k$ is closed under totalisations.

\begin{lemma}
	\label{prop-inversegeom}
	Filtered colimits commute with totalisations in $\Mod_{k, \leq 0}$. Dually, inverse limits commute with geometric realisations in $\Solid_{k, \geq 0}^{\aperf}$.
\end{lemma}

\begin{proof}
	Let $\{V^\bullet _i\}$ be a filtered system of cosimplicial objects in $\Mod_{k, \leq 0}$. It suffices to show that the map $ \tau_{\geq n} \varinjlim \Tot (V_i^\bullet) \to \tau_{\geq n} \Tot (\varinjlim V^\bullet _i)$ is an equivalence for every $n < 0$. Since homology commutes with filtered colimits in $\Mod_k$, we need to show that the natural map $\varinjlim \tau_{\geq n} \Tot (V_i^\bullet) \to \tau_{\geq n} \Tot(\varinjlim V^\bullet_i)$ is an equivalence.

	By (\cite{HTT}, Proposition 6.5.3.10), it suffices to show that the map
	\begin{align*}
		\varinjlim \lim_{\Delta_{\leq n+1}} V_i^\bullet \to \lim_{\Delta_{\leq n+1}} \varinjlim V_i^\bullet
	\end{align*}
	is $n$-coconnective for all $n$. This is true since colimits and finite limits commute in a stable $\infty$-category.
\end{proof}

\section{Ultrasolid commutative algebra}
\label{section-commalg}

\subsection{Ultrasolid $k$-algebras and their modules}
\label{section-rings}

\begin{defi}
	An \emph{ultrasolid $k$-algebra} is a commutative algebra object in the category $\Solid_k^\heartsuit$. We write $\CAlg_k^{\blacksquare, \heartsuit}$ for the category of ultrasolid $k$-algebras.
\end{defi}

\begin{rem}
	\label{rem-ringmonad}
	The category of ultrasolid $k$-algebras is the category of algebras over the free utrasolid $k$-algebra monad on $\Solid_k^\heartsuit$ (\cite{HA}, Example 4.7.3.11), which we will write as $\LSym^*$. By (\cite{HA}, Proposition 3.1.3.13), this monad is given on $V \in \Solid_k^\heartsuit$ by
	\begin{align*}
		\LSym^* (V) := \bigoplus_{n \geq 0} V^{\otimes n}_{\Sigma_n}
	\end{align*}
	We will write $\LSym^n : \Solid_k^\heartsuit \to \Solid_k^\heartsuit$ for the functor $V \mapsto V^{\otimes n} _{\Sigma_n}$.

	For $V = \prod_I k$, we have that
	\begin{align*}
		\LSym^* (V) = \bigoplus_{n \geq 0} \prod_{I^n_{\Sigma_n}} k
	\end{align*}
	This can be thought of as the power series ring on variables $\{x_i\}_{i \in I}$ where we allow arbitrary sums of same degree monomials. That is, we allow $\sum_{i \in I} x_i$ and $\sum_{i,j \in I} x_i x_j$ but we don't allow $\sum_{n \geq 0} x_i^n$. More precisely, this $k$-algebra has a natural grading, and \emph{as a graded $k$-algebra}, we can write it as $\varprojlim _{J \subset I} \LSym^* (\prod_J k)$, where the inverse limit is taken over all finite subsets $J \subset I$. Instead, we can take this inverse limit in the category of graded ultrasolid $k$-algebras, embedding each $\LSym^* (\prod_J k)$ into ultrasolid $k$-algebras via the symmetric monoidal embedding $\Mod_k^\heartsuit \xhookrightarrow{} \Solid_k^\heartsuit$. The underlying ultrasolid $k$-algebra is $\LSym^* (V)$.

	We can also consider completions of these algebras with respect to their maximal ideal $\bigoplus_{n > 0} \LSym^n (V)$.
	\begin{align*}
		\widehat{\LSym^*} (V) = \prod_{n \geq 0} V^{\otimes n}_{\Sigma_n}
	\end{align*}
	(see \cref{construction-freedeltacomplete} for the multiplicative structure on $\widehat{\LSym}^* (V)$). Similarly, this is the ring of formal power series on variables $\{x_i\}_{i \in I}$ where we allow any infinite sum of distinct monomials of any degree.
\end{rem}

\begin{rem}
	Even though we are not working with derived functors yet, we are still referring to the symmetric algebra monad as $\LSym^*$. This will be in order to not confuse it with the $\Sym^*$ monad used to define $\EE_\infty$ ultrasolid $k$-algebras.
\end{rem}

\begin{defi}
	Let $A$ be an ultrasolid $k$-algebra. Then, we will write $\Solid_A^\heartsuit$ for the category of $A$-modules in $\Solid_k^\heartsuit$.	
\end{defi}

\begin{prop}
	\label{prop-monadicadj}
	There is a monadic adjunction
	\begin{align*}
		- \otimes A : \Solid_k ^\heartsuit \leftrightarrows \Solid_A^\heartsuit : \forget
	\end{align*}
	The category $\Solid_k^\heartsuit$ is generated by compact projectives of the form $V \otimes A$, where $V \in \ProVect$, and there are enough projectives. It is abelian, complete and cocomplete and satisfies the same Grothendieck axioms as $\Solid_k^\heartsuit$.
\end{prop}

\begin{proof}
	The proof follows in the same way as the classical case. Since the forgetful functor commutes with small limits and colimits, it is a monadic adjunction.

	This adjunction makes it clear that $V \otimes A$, where $V \in \ProVect$, generate. Also, $\Hom_{\Solid_A^\heartsuit} (V \otimes A, -) \cong \Hom _{\Solid_k^\heartsuit} (V, -)$, and since $V$ is compact projective in $\Solid_k^\heartsuit$ and colimits of $A$-modules are computed in $\Solid_k^\heartsuit$, the objects $V \otimes A$ are compact and projective in $\Solid_A^\heartsuit$.

	For $M \in \Solid_A^\heartsuit$, we have a surjection $\bigoplus_{V \to M} A \otimes V$, where we are taking the sum over all $V \in \ProVect$ with a map $V \to M$ (by restricting to a suitable cardinality), so there are enough projectives.

	The last statement follows from the equivalent result in $\Solid_k^\heartsuit$ since limits and colimits can be computed there and the forgetful functor is conservative.
\end{proof}

\begin{defi}
	We write $\Solid_A$ for the $\infty$-category of chain complexes of $A$-modules.
\end{defi}

The $\infty$-category $\Solid_A$ inherits a symmetric monoidal structure from $\Solid_A^\heartsuit$. In order to compute $\Tor$ groups, we need to check that projectives are flat.

\begin{prop}
	Let $A$ be an ultrasolid $k$-algebra and $V \in \ProVect$. Then, $V \otimes A$ is flat as an $A$-module.
\end{prop}

\begin{proof}
	For $M \in \Solid_A^\heartsuit$, we can write $M \otimes_A (V \otimes A) = M \otimes V$. The forgetful functor $\Solid_A ^\heartsuit \to \Solid_k ^\heartsuit$ preserves small limits and colimits, so tensoring by $V \otimes A$ preserves exactness, since $V$ is flat as an ultrasolid $k$-module (\cref{prop-flat}).
\end{proof}

\begin{rem}
	In particular, to compute $M \otimes_A^L N$ for $M,N \in \Solid_A^\heartsuit$, we can resolve $M$ by projectives of the form $\bigoplus_I  V_i \otimes A$ with $V_i \in \ProVect$, and take the tensor product with $N$.
\end{rem}

\subsection{Coherent ultrasolid $k$-algebras}

\begin{example}
	Remember that a commutative ring $R$ is \emph{coherent} if every finitely generated ideal of $R$ is finitely presented.

	Finitely presented $R$-modules are precisely the compact objects in $\Mod_R^\heartsuit$, so another way to phrase coherence is that a ring is coherent if and only if the subcategory of compact objects of $\Mod_R^\heartsuit$ is closed under finite limits (\cite{coherent}). Notice that then $(\Mod_R^\heartsuit)^\omega$ forms an abelian category.
\end{example}

Coherence is arguably a more well-behaved notion than that of a Noetherian ring, which is often too strong. In this section we discuss coherent ultrasolid $k$-algebras and some of their properties.

\begin{defi}
	An ultrasolid $k$-algebra $R$ is \emph{coherent} if the subcategory of compact objects of $\Solid_R^\heartsuit$ is closed under finite limits.
\end{defi}

We will first try to characterise the compact objects in $\Solid_R^\heartsuit$ and then find some examples of ultrasolid $k$-algebras that are indeed coherent. Throughout the rest of this section we fix an ultrasolid $k$-algebra $R$.

\begin{defi}
	Let $R$ be an ultrasolid $k$-algebra.

	An $R$-module $M$ is \emph{profinitely generated} if it admits a surjection $R \otimes V \twoheadrightarrow M$ for some $V \in \ProVect$.

	An $R$-module $M$ is \emph{profinitely presented} if there is an exact sequence of the form
	\begin{align*}
		R \otimes V \to R \otimes W \to M \to 0
	\end{align*}
	for some $V,W \in \ProVect$.
\end{defi}

\begin{prop}
	An $R$-module is profinitely presented if and only if it is a compact object in $\Solid_R^\heartsuit$.
\end{prop}

\begin{proof}
	Let $V \in \ProVect$. Remember that from the adjunction with ultrasolid $k$-modules,
	\begin{align*}
		\Hom_{\Solid_R^\heartsuit} (R \otimes V, -) \cong \Hom _{\Solid_k^\heartsuit} (V, -)
	\end{align*}
	Colimits of $R$-modules are formed in $\Solid_k^\heartsuit$ and $V$ is compact in $\Solid_k^\heartsuit$. It follows that $R \otimes V$ is compact. Given that compact objects are closed under finite colimits, we get that all profinitely presented $R$-modules are compact.

	Conversely, suppose $M \in \Solid_R^\heartsuit$ is compact. We will first show that $M$ can be written as a filtered colimit of profinitely presented modules. For $V \in \ProVect$, any element in $M(V)$ can be identified with a map $V \otimes R \to M$, so $M(V)$ lies in the image of a profinitely generated $R$-module. Then, we see that $M$ can be written as a colimit of its profinitely generated submodules (restricting to a suitable cardinality). This colimit is filtered because if $\{M_i\}$ a finite diagram of profinitely generated modules admitting each a surjection from $V_i \otimes R$ with $V_i \in \ProVect$, then $\colim M_i$ admits a surjection from $ R \otimes \bigoplus V_i$. We may then just show that a profinitely generated module can be written as a filtered colimit of profinitely presented modules. 

	If $M$ is profinitely generated, then $M \cong \coker (N \to (V \otimes R))$ for some $N \in \Solid_R^\heartsuit$ and $V \in \ProVect$. Again, we can write $N$ as a filtered colimit of profinitely generated modules, which gives us a surjection $\varinjlim W_i \otimes R \to N$, where each $W_i \in \ProVect$. This expresses $M$ as a filtered colimit of profinitely presented modules. Hence, any $R$-module can be obtained as a filtered colimit of profinitely presented modules.

	$M$ is compact, so the identity map must factor through a profinitely presented module, which implies that $M$ is a retract of a profinitely presented $R$-module. Hence, it suffices to show that retracts of profinitely presented $R$-modules are profinitely presented. We will prove the slightly more general fact that if
	\begin{align*}
		0 \to M_1 \to M_2 \to M_3 \to 0
	\end{align*}
	is a short exact sequence of $R$-modules, $M_2$ is profinitely presented and $M_1$ is profinitely generated, then $M_3$ is profinitely presented. The argument is taken from \cite[\href{https://stacks.math.columbia.edu/tag/0517}{Tag 0517}]{stacks-project} and adapted to our situation.

	Choose a profinite presentation
	\begin{align*}
		R \otimes U \to R \otimes V \to M_2 \to 0
	\end{align*}
	and a surjection $R \otimes W \twoheadrightarrow M_1$ for some $U,V,W \in \ProVect$. By projectivity, we can factorise the composition $R \otimes W \to M_1 \to M_2$ as $R \otimes W \to R \otimes V \to M_2$. We now claim that the sequence 	\begin{align*}
		R \otimes (U \oplus W) \xrightarrow{f} R \otimes V \xrightarrow{g} M_3 \to 0
	\end{align*}
	is exact (we have labelled the arrows for ease of notation).

	$g$ is surjective because it is a composition of surjections $R \otimes V \to M_2 \to M_3$. It is clear that $g \circ f = 0$ by construction, so $\im f \subset \ker g$. It remains to show that $\im f \supset \ker g$. Since $R \otimes U$ is contained in the kernel, it is enough to show the surjectivity after quotienting by $R \otimes U$. We get a sequence $R \otimes W \xrightarrow{\tilde{f}} M_2 \to M_3$, but we have a factorisation $R \otimes W \twoheadrightarrow M_1 \to M_2 \to M_3$, so it is clear that $R \otimes W$ surjects onto the kernel of $M_2 \to M_3$.
\end{proof}

We now introduce a rich variety of coherent ultrasolid $k$-algebras.

\begin{lemma}
	\label{lemma-profinitecoherent}
	Let $R$ be an ultrasolid $k$-algebra such that the underlying ultrasolid $k$-module is profinite. Then, it is coherent.	

	Additionally, an $R$-module is profinitely presented if and only if the underlying ultrasolid module is profinite.
\end{lemma}

\begin{proof}
	We may only show the second part, since any limit of $R$-modules is computed in $k$-modules, and finite limits of profinite modules are profinite.

	Suppose $M \in (\Solid_{R}^\heartsuit)^\omega$ (that is, $M$ is a compact object in $\Solid_R^\heartsuit$). Then, there is an exact sequence
	\begin{align*}
		R \otimes V \to R \otimes W \to M \to 0
	\end{align*}
	where $V, W \in \ProVect$. Since $R$ is profinite, $R \otimes V$ must be profinite (\cref{prop-profinite}). Hence, $M$ is a cokernel of profinite modules, so it is profinite.

	Conversely, suppose $M \in \Solid_R^\heartsuit$ is such that the underlying ultrasolid $k$-module is profinite. Then, it admits a surjection of ultrasolid $k$-modules $V \to M$, where $V \in \ProVect$ (for example, by taking $V = M$), which extends to a surjection of $R$-modules $R \otimes V  \to M$. Clearly, the kernel is profinite so we can repeat the argument to get an exact sequence of the form $R \otimes W \to R \otimes V \to M \to 0$, where $W \in \ProVect$.
\end{proof}

\begin{example}
	In particular, if $R$ is an ordinary complete local Noetherian $k$-algebra with residue field $k$, we can write $R = \varprojlim R/\mathfrak{m}^i$, where $\mathfrak{m}$ is its maximal ideal. The fact that it is Noetherian implies that all of the $R/\mathfrak{m}^i$ are finite-dimensional $k$-vector spaces, so if we take this inverse limit as an ultrasolid $k$-module we get a profinite $k$-module. Hence, by \cref{lemma-profinitecoherent}, complete local Noetherian $k$-algebras with residue field $k$ are coherent as ultrasolid $k$-algebras (when given the appropriate ultrasolid structure coming from the maximal ideal).
\end{example}

\begin{rem}
	It is not known to the author whether a discrete Noetherian $k$-algebra is coherent when considered as a discrete ultrasolid $k$-algebra. If we consider the ring $k[x]$ as an ultrasolid $k$-algebra, the projectives of the category of ultrasolid $k[x]$-modules are of the form $k[x] \otimes \prod_I k$. We can now write this as
	\begin{align*}
		\bigoplus_{n \geq 0} \prod_{i \in I} k x_i^n
	\end{align*}
	and there is no obvious reason why the kernel of a map between $k[x]$-modules of this form should be profinitely generated.
\end{rem}

\subsection{Complete profinite ultrasolid $k$-algebras}

We have already seen in \cref{lemma-profinitecoherent} that a good class of examples of coherent ultrasolid $k$-algebras are those whose underlying ultrasolid $k$-module is profinite. If we actually ask for completeness, we can have many properties similar to those of complete local Noetherian $k$-algebras.

\begin{defi}
	An \emph{augmented ultrasolid $k$-algebra} is an ultrasolid $k$-algebra $R$ equipped with a map $R \to k$ such that its composition with the unit $k \to R$ is the identity on $k$. We write $\CAlg_{k/ /k}^{\blacksquare, \heartsuit}$ for the category of augmented ultrasolid $k$-algebras.
	
	Given $R \in \CAlg^{\blacksquare, \heartsuit}_{k/ /k}$, there is an \emph{augmentation ideal} $\mathfrak{m}$, defined as the kernel of the map $R \to k$. The  \emph{adic filtration} on $R$ given by the image of $\mathfrak{m}^{\otimes_R n}$ in $R$. This induces a sequence of successive quotients
	\begin{align*}
		\dots \to R/\mathfrak{m}^3 \to R/\mathfrak{m}^2 \to R/\mathfrak{m} \cong k
	\end{align*}
	An augmented ultrasolid $k$-algebra is \emph{complete} if the natural map $R \to \varprojlim R/\mathfrak{m}^n$ is an isomorphism.

	A \emph{complete profinite ultrasolid $k$-algebra} is a complete augmented ultrasolid $k$-algebra whose underlying ultrasolid $k$-module is profinite.
\end{defi}

\begin{rem}
	Each map $R \to R/\mathfrak{m}^n$ and $R/\mathfrak{m}^{n+1} \to R/\mathfrak{m}^n$ is the corresponding quotient map of algebras, so the map $R \to R/\mathfrak{m}^n$ is a map of ultrasolid $k$-algebras. Hence, if $R$ is complete, $R \cong \varprojlim R/\mathfrak{m}^n$ \emph{as an ultrasolid $k$-algebra}.
\end{rem}

\begin{prop}
	\label{prop-mapsfromcompletediscrete}
	Suppose $R \in \CAlg_{k/ /k}^{\blacksquare, \heartsuit}$ is complete. Then, for $V \in \Solid_k^\heartsuit$, the map $\LSym^* (V) \to \widehat{\LSym}^* (V)$ induces an equivalence
	\begin{align*}
		\Hom_{\CAlg_{k/ /k}^{\blacksquare, \heartsuit}} (\widehat{\LSym}^* (V) , R) \xrightarrow{\cong} \Hom_{\CAlg_{k/ /k}^{\blacksquare, \heartsuit}} (\LSym (V), R)
	\end{align*}
\end{prop}

\begin{proof}
	Let $\mathfrak{m}$ be the augmentation ideal of $R$. Then,
	\begin{align*}
		\Hom_{\CAlg_{k/ /k}^{\blacksquare, \heartsuit}} (\widehat{\LSym}^* (V), R) = \varprojlim \Hom_{\CAlg_{k/ /k}^{\blacksquare, \heartsuit}} (\widehat{\LSym}^* (V), R/\mathfrak{m}^n)
	\end{align*}
	Since we are restricting to maps of augmented $k$-algebras, the augmentation ideal $I$ of $\widehat{\LSym}^* (V)$ maps to $\mathfrak{m}$. This ideal is nilpotent in each $R/\mathfrak{m}^n$, so any map $\widehat{\LSym}^* (V) \to R/\mathfrak{m}^n$ factors through $\widehat{\LSym}^* (V)/I^n$. This gives an isomorphism
	\begin{align*}
		\Hom_{\CAlg_{k/ /k}^{\blacksquare, \heartsuit}} (\widehat{\LSym}^* (V), R/\mathfrak{m}^n) & \cong \Hom_{\CAlg_{k/ /k}^{\blacksquare, \heartsuit}} (\widehat{\LSym}^* (V)/I^n , R/\mathfrak{m}^n) \\
																																					 & \cong \Hom_{\CAlg_{k/ /k}^{\blacksquare, \heartsuit}} (\LSym^* (V)/J^n , R/\mathfrak{m}^n)
	\end{align*}
	where $J$ is the augmentation ideal of $\LSym^* (V)$. Here we used that $\LSym^* (V)/J^n \cong \bigoplus_{m < n} \LSym^m (V) \cong \widehat{\LSym}^* (V)/I^n$. By a similar argument,
	\begin{align*}
		\Hom_{\CAlg_{k/ /k}^\blacksquare} (\LSym^* (V)/J^n, R/\mathfrak{m}^n) \cong \Hom_{\CAlg_{k/ /k}^\blacksquare} (\LSym^* (V), R/\mathfrak{m}^n)	
	\end{align*}
	The result follows.
\end{proof}

\begin{rem}
	If $R$ is a complete local Noetherian $k$-algebra with residue field $k$ and augmentation ideal $\mathfrak{m}$, we can construct the ultrasolid $k$-algebra $\varprojlim R/\mathfrak{m}^i$. For a complete local Noetherian $k$-algebra with residue field $k$ and augmentation ideal $\mathfrak{n}$, we have that
	\begin{align*}
		\Hom_{\CAlg^{\blacksquare, \heartsuit}_{k/ /k}} (\varprojlim R/\mathfrak{m}^i, \varprojlim S/\mathfrak{n}^j) \cong \varprojlim \Hom_{\CAlg_{k/ /k}^{\blacksquare, \heartsuit}} (R/\mathfrak{m}^i, S/\mathfrak{n}^i)
	\end{align*}
	The same argument from the previous proof gives the same mapping space computed in the category of complete local Noetherian $k$-algebras with residue field $k$. Thus, we have an embedding from complete local Noetherian $k$-algebras with residue field $k$ into ultrasolid $k$-algebras whose image is precisely those complete ultrasolid $k$-algebras whose successive quotients of the adic filtration are all finite-dimensional vector spaces.
\end{rem}

\begin{prop}
	\label{prop-compactcomplete}
	Let $A$ be an augmented ultrasolid $k$-algebra that is complete and profinite and let $M$ be an $A$-module that is profinitely presented. Then, $M$ is complete.

	That is, if $\mathfrak{m}$ is the augmentation ideal of $A$, the natural map $M \xrightarrow{\simeq} \varprojlim_n M/\mathfrak{m}^n M$ is an isomorphism.
\end{prop}

\begin{proof}
	We first show this for $M = A \otimes V$, where $V \in \Pro(\Vect_k^\omega)$. Then, $M/\mathfrak{m}^n M \cong V \otimes (A/\mathfrak{m}^n)$. Each piece $A/\mathfrak{m}^n$ is profinite, since it is a finite colimit of profinite modules. The tensor product of profinite ultrasolid modules commutes with inverse limits (\cref{prop-profinite}), so the result follows by the completeness of $A$.

	For general $M$, it can be written as a finite colimit of $A$-modules of the form $A \otimes V$ for $V \in \Pro(\Vect_k^\omega)$. We checked the result for this class of $A$-modules, and inverse limits are exact and commute with the tensor product in profinite ultrasolid modules (\cref{prop-profinite}), so we are done.
\end{proof}

\begin{prop}[Ultrasolid Nakayama]
	Let $A \in \CAlg_{k/ /k}^{\blacksquare, \heartsuit}$ be complete and profinite. Suppose that $M \in \Solid_A^\heartsuit$ is profinite and $\mathfrak{m} \otimes_A M = M$, where $\mathfrak{m}$ is the augmentation ideal of $A$. Then, $M = 0$.
\end{prop}

\begin{proof}
	By \cref{lemma-profinitecoherent}, being profinite is equivalent to being profinitely presented, and by \cref{prop-compactcomplete} $M$ must be complete. We have $\mathfrak{m} \otimes_A M = M$, so $\mathfrak{m} \otimes_A M /\mathfrak{m}^n = M /\mathfrak{m}^n$. Inductively $ M/\mathfrak{m}^n \cong \mathfrak{m}^n \otimes_A M/\mathfrak{m}^n \cong 0$ so we can conclude that $M/\mathfrak{m}^n = 0$. Hence, $M = \varprojlim M/\mathfrak{m}^n = 0$.
\end{proof}

\section{$\EE_\infty$ ultrasolid $k$-algebras}
\label{section-EEinfty}

\subsection{Free algebras and profinite chains of condensed anima}
\label{subsection-Einftyfree}

We have seen that $\Solid_k$ is a symmetric monoidal $\infty$-category. Hence, we can do the usual constructions.

\begin{defi}
	An \emph{$\EE_\infty$ ultrasolid $k$-algebra} is a commutative algebra object in $\Solid_k$ (\cite{HA}, Definition 2.1.3.1).  We will write $\CAlg_k^{\blacksquare}$ for the $\infty$-category of $\EE_\infty$ ultrasolid $k$-algebras, and $\CAlg_k^{\blacksquare, \cn}$ for the full subcategory of connective objects.

	Given an $\EE_\infty$ ultrasolid $k$-algebra $A$, the $\infty$-category of modules over $A$ is written as $\Solid_A$.
\end{defi}

\begin{rem}
	\label{rem-monadicadjmodules}
	An analogue of \cref{prop-monadicadj} is true for modules over $\EE_\infty$ ultrasolid $k$-algebras, with an identical proof.
\end{rem}

\begin{rem}
	The $\infty$-category of $\EE_{\infty}$ ultrasolid $k$-algebras is the $\infty$-category of algebras over the $\Sym^*$  monad (\cite{HA}, Example 4.7.3.11). For $V \in \Solid_k$ we have
	\begin{align*}
		\Sym^* (V) = \bigoplus_{n \geq 0} V^{\otimes n}_{h \Sigma_n}
	\end{align*}
	And we will refer to each functor $V \mapsto V^{\otimes n}_{h \Sigma_n}$ as $\Sym^n$.

	We will write $\CAlg_A^\blacksquare := (\CAlg_k)_{A/}^\blacksquare$ for the $\infty$-category of $\EE_\infty$ ultrasolid $k$-algebras under $A \in \CAlg_k^\blacksquare$. This $\infty$-category can also be written as algebras over the $\Sym_A^*$ monad on $\Solid_{A, \geq 0}$ (\cite{HA}, Example 4.7.3.11) given by
	\begin{align*}
		\Sym_A^* (V) = \bigoplus_{n \geq 0} V^{\otimes_A n}_{h \Sigma_n}
	\end{align*}
	We can apply the Barr-Beck-Lurie Theorem (\cite{HA}, Theorem 4.7.3.5) to see that $\CAlg_A^\blacksquare$ is also monadic over $\Solid_k$, and the monad is given by $A \otimes \Sym^*$.
\end{rem}

In the classical setting, $\Sym^* (k[0])$ can be described via the homology of the symmetric groups, since $\Sym^n (k[0]) \simeq C_* (B\Sigma_n, k)$. We can use condensed anima to obtain a similar description of free $\EE_\infty$ ultrasolid $k$-algebras. See \cref{example-condani} for a definition of the category of condensed anima.

\begin{construction}
	Let $S$ be an extremally disconnected space, which we may write as an inverse limit of finite discrete spaces $S = \varprojlim S_i$.

	We define the \emph{profinite chains} of $S$ to be
	\begin{align*}
		\widehat{C}_\bullet (S,k) := \varprojlim C_\bullet (S_i, k)
	\end{align*}
	where the inverse limit is taken in $\Solid_k$ and the functor $C_\bullet(-,k)$ refers to $k$-valued chains on a space.

	By \cref{example-condani}, this uniquely extends to a sifted-colimit-preserving functor
	\begin{align*}
		\widehat{C}_\bullet (-,k) : \Cond(\Space) \to \Solid_{k, \geq 0}
	\end{align*}
	where $\Cond(\Space)$ is the $\infty$-category of condensed anima.

	Actually, profinite chains commute with small colimits. This is because it clearly commutes with finite coproducts of extremally disconnected spaces, so it will commute with small colimits when we left Kan extend (\cite{HTT}, Proposition 5.5.8.15).
\end{construction}

\begin{rem}
	\label{rem-homologyprospaces}
	We will use throughout this section the fact that if $S = \varprojlim S_i$ is a profinite space, then $\varprojlim k[S_i] = C(S,k) ^\vee$. This is because any continuous map $S \to k$ has finite image, so it factors through one of the $S_i$. Hence, 
	\begin{align*}
		C(S,k)^\vee \cong (\varinjlim C(S_i,k))^\vee \implies C(S,k) ^\vee \cong \varprojlim k[S_i]
	\end{align*}
	This is also true at the level of chain complexes by \cref{prop-complexesprohomology}, so that $\widehat{C}_\bullet (S,k) \simeq C(S,k)^\vee [0]$.
\end{rem}

Remember that by \cref{example-condani}, there are fully faithful embeddings both from the $1$-category of compact Hausdorff spaces and the $\infty$-category of anima into condensed anima.

\begin{lemma}
	If $X \in \Cond(\Space)$ is given by an anima, then profinite chains coincide with usual $k$-valued chains.
\end{lemma}

\begin{proof}
	This immediately follows from the fact that it agrees on finite discrete spaces and it preserves sifted colimits.
\end{proof}

\begin{prop}
	\label{prop-prohomologycs}
	Let $S$ be a compact Hausdorff space $S$, and write $R\Gamma (S,k) \in \Mod_{k, \leq 0}$ for the sheaf cohomology of the locally constant sheaf with value $k$ on $S$. Then,
	\begin{align*}
		\widehat{C}_\bullet (S,k) \simeq R\Gamma (S, k)^\vee
	\end{align*}
\end{prop}

\begin{proof}
	First we prove this for extremally disconnected $S$. Then, we have defined $\widehat{C}_\bullet (S,k) \simeq (\varprojlim k[S_i]) [0]$ where $S = \varprojlim S_i$ and the $S_i$ are all finite and discrete. Also, $\widehat{C}_\bullet (S,k) \simeq C(S,k)^\vee [0]$ (\cref{rem-homologyprospaces}). By (\cite{Wiegand}, Theorem 5.1), this is precisely $R\Gamma (S,k)^\vee$.

	For general compact Hausdorff $S$, take a hypercover $T_\bullet \to S$ by extremally disconnected spaces. By (\cite{HTT}, Lemma 6.5.3.11), $|T_\bullet| \simeq S$ as condensed anima. Since profinite chains commute with sifted colimits, $\widehat{C}_\bullet (S,k) \simeq |\widehat{C}_\bullet (T_i,k)| \simeq |C(T_i,k)^\vee [0] |$.

	Given that $\Solid_{k, \geq 0} ^{\aperf}$ is closed under geometric realisations (\cref{prop-complexesprohomology}), this geometric realisation belongs to $\Solid_{k, \geq 0}^{\aperf}$ and it is equivalent to its double dual. Hence, $\widehat{C}_\bullet (S,k) \simeq \Tot (C(T_i,k)[0])^\vee$.

	We now claim that $\Tot(C(T_i,k)[0]) \simeq R \Gamma (S,k) $. This totalisation will live in $\Mod_{k, \leq 0} \subset \Solid_k$, since there is a totalisation-preserving embedding $\Mod_{k, \leq 0} \to \Solid_k$, from coconnective chain complexes of $k$-modules to chain complexes of ultrasolid $k$-modules. Hence, we can perform the computation in $\Mod_k$.

	For a condensed set $X$, we write $k[X]$ for the free condensed $k$-module on $X$, which is projective if $X$ is an extremally disconnected space (\cite{condensed}, Lecture II). Then, by projectivity, $\RHom_{\Cond(\Mod_k)} (k[T_i][0] , k[0] ) \simeq \Hom_{\Cond(\Mod_k^\heartsuit)} (k[T_i], k) [0] \simeq C(T_i,k)[0] $. Also, by \cref{lemma-hypercoverfree}, $k[S] [0] \simeq |k[T_\bullet][0]|$. Hence,
	\begin{align*}
		\Tot(C(T_i,k)[0]) & \simeq \Tot (\RHom_{\Cond(\Mod_k)} (k[T_i][0], k[0] )) \simeq \RHom_{\Cond(\Mod_k)} (k[S][0], k[0])
	\end{align*}
	We are now done by (\cite{condensed}, Theorem 3.1) since this is exactly the sheaf cohomology of $S$.
\end{proof}

\begin{lemma}
	There is a Künneth isomorphism for profinite chains. That is, for $X,Y \in \Cond(\Space)$, 
	\begin{align*}
		\widehat{C}_\bullet (X \times Y, k) \simeq \widehat{C}_\bullet (X, k) \otimes \widehat{C}_\bullet (Y,k)
	\end{align*}
\end{lemma}

\begin{proof}
	By expanding sifted colimits in each variable, we may just show this for $X$ and $Y$ extremally disconnected. We now have the problem that $X \times Y$ is not necessarily extremally disconnected, although it will always be profinite. 

	If $S= \varprojlim S_i$ is profinite, by \cref{prop-prohomologycs}, $\widehat{C}_\bullet (S,k) \simeq R\Gamma (S,k)^\vee$, and by (\cite{Wiegand}, Theorem 5.1) this is concentrated in degree zero, so $\widehat{C}_\bullet (S,k) \simeq C(S,k)^\vee  [0] \simeq \varprojlim k[S_i][0]$. Hence, we can see that the $\widehat{C}_\bullet: \ProFin \to \Solid_k$ coincides with the inverse-limit-preserving extension of the functor sending a finite set $S$ to $k[S][0] $. Writing $X = \varprojlim X_i$ and $Y = \varprojlim Y_i$ as inverse limits of finite sets,
	\begin{align*}
		\widehat{C}_\bullet (X \times Y, k) \simeq \widehat{C}_\bullet (\varprojlim X_i \times Y_i, k) \simeq \varprojlim \widehat{C}_\bullet (X_i \times Y_i, k) \simeq \varprojlim (\widehat{C}_\bullet (X_i, k) \otimes \widehat{C}_\bullet (Y_i, k)) \simeq \widehat{C}_\bullet (X,k) \otimes \widehat{C}_\bullet (Y,k)
	\end{align*}
	where in the last equality we used that inverse limits of profinite modules commute with the tensor product (\cref{prop-complexesprohomology}).
\end{proof}

\begin{example}
	Remember that if $V$ is a vector space, we can think of $V^{\otimes n}_{h \Sigma_n}$ as $k$-valued chains on the homotopy orbits of the space $((S^0)^{\vee \dim V})^{\wedge n}$. It is possible to make a similar interpretation for $V = \prod_I k [0]$.

	Given a set $I$, we can associate the profinite space $\widehat{I}$. This space is the inverse limit of $\{*\} \sqcup S_i$ where $S_i \subset I$ is finite. The transition maps are contravariant with respect to inclusion and project extra elements onto the base point.

	$V = \prod_I k$ is equivalent to profinite chains on the condensed anima $\widehat{I}$, due to our characterization of profinite chains on profinite spaces. By the Künneth isomorphism, $V^{\otimes n} = \widehat{C}_\bullet (\widehat{I}^{\wedge n}, k)$. Profinite chains commute with colimits, so $V^{\otimes n}_{h \Sigma_n}$ are profinite chains on the homotopy orbits $\widehat{I}^{\wedge n}_{h \Sigma_n}$.

	We can think of $\widehat{I}$ as a profinite wedge of $0$-spheres. When $V = \prod_I k[n]$ we can take the suspension of this construction to obtain a profinite wedge of spheres and have a higher-dimensional analogue.
\end{example}

We will now give an explicit description of $\pi_* (\Sym^* (V[0]))$ for $V \in \ProVect$ in terms of Dyer-Lashof operations. These operations and their relations were originally computed by Araki-Kudo \cite{Kudo}, Dyer-Lashof \cite{Dyer-Lashof}, Cohen-Lada-May \cite{homologyloopspaces} and Bruner-May-McClure-Steinberger \cite{Hinfty}.

Firstly, we will try to understand each piece $\Sym^n$ on profinite vector spaces.

\begin{lemma}
	\label{lemma-Symnlimits}
	Let $n \geq 0$.
	\begin{enumerate}
		\item $\Sym^n$ and $\Sym^*$ commute with sifted colimits.
		\item $\Sym^n$ preserves the category $\Solid_{k, \geq 0}^{\aperf}$ and commutes with inverse limits when restricted to this category.
	\end{enumerate}
\end{lemma}

\begin{proof}
	Remember that $\Sym^n (V) \simeq V^{\otimes n}_{h \Sigma_n}$ so it is clear that it commutes with sifted colimits, since the tensor product does. Hence, $\Sym^* = \bigoplus_{n \geq 0} \Sym^n$ also commutes with sifted colimits.

	Let $V \in \Solid_{k, \geq 0}^{\aperf}$. By applying the Bousfield-Kan formula (\cite{bousfield-kan}, \S 12), we can write $\Sym^n (V)$ as a geometric realisation of objects of the form $k[\Sigma_n^i] \otimes V^{\otimes n}$. By \cref{prop-complexesprohomology} each functor $V \mapsto k[\Sigma_n^i] \otimes V^{\otimes n}$ preserves $\Solid_{k, \geq 0}^{\aperf}$ and commutes with inverse limits. By \cref{prop-inversegeom}, inverse limits commute with geometric realisations in $\Solid_{k, \geq 0}^{\aperf}$, so $\Sym^n$ must commute with inverse limits.
\end{proof}

We now recall some classical results regarding the homology of free $\EE_\infty$-algebras. We will treat these homotopy groups as bigraded objects, since they have both homotopical and polynomial grading.

\begin{defi}
	Given an $\infty$-category $\catC$, we will write $\Gr \catC$ for the $\infty$-category of non-negatively graded objects in $\catC$. That is, functors $\ZZ_{\geq 0}^{\ds} \to \catC$, where $\ZZ_{\geq 0}^{\ds}$ is the category of non-negative integers with only identity morphisms.

	To ease notation, we will write $\Gr^2 \catC := \Gr \Gr \catC$ for the $\infty$-category of bigraded objects. Given $X \in \Gr^2 \catC$, there are two possible directions for grading, which we will call \emph{homotopical} and \emph{polynomial}.

	If $\catC$ is symmetric monoidal, $\Gr \catC$ inherits a symmetric monoidal structure by using Day convolution (\cite{dayconvolution}).
\end{defi}

\begin{example}
	The functor $\pi_* \circ \Sym^* : \Solid_{k, \geq 0} \to \Gr^2 \Solid_k^\heartsuit$ naturally lands on bigraded ultrasolid modules, where each piece $\pi_* \circ \Sym^n$ has polynomial grading $n$.	
\end{example}

\begin{defi}
	We will write $F_{\gr}: \Gr^2 \Solid_k^\heartsuit \to \Gr^2 \Solid_k^\heartsuit$ for the free commutative algebra functor. That is, for $V \in \Gr^2 \Solid_k^\heartsuit$,
	\begin{align*}
		F_{\gr}(V) = \bigoplus_{n \geq 0} V^{\otimes n}_{\Sigma_n}
	\end{align*}
\end{defi}

\begin{rem}
	Although we are calling some of the grading "homological", the previous construction is completely $1$-categorical. The reason for this naming is that the homology of the free algebras $\pi_* \Sym^* (-)$ is naturally bigraded. This has an obvious homological grading and the polynomial grading comes from the $\Sym^n$ pieces.
\end{rem}

\begin{construction}[Dyer-Lashof operations]
	\label{construction-DLoperations}
	Let $p$ be the characteristic of $k$.

	If $p > 0$, for any $n, r \in \ZZ_{\geq 0}$ there is a map $Q^r: \Sym^* (k[n+r]) \to \Sym^* (k[n])$, which we can consider as a homotopy operation on $\EE_\infty$ $k$-algebras $\pi_n (-) \to \pi_{n+r}(-)$. We also have the homotopy operation given by the Bockstein homomorphism $\pi_n (-) \to \pi_{n-1} (-)$. They satisfy the following relations:
	\begin{enumerate}
		\item (Additivity) $Q^r (x+y) = Q^r (x) + Q^r (y)$.
		\item (Instability) $Q^r x = 0$ if $r < |x|$ for $p = 2$ or $r < 2 |x|$ for $p > 2$.
		\item (Frobenius) $Q^r x = x^p$ if $r = |x|$.
		\item (Unit) $Q^r 1 = 0$ for $r \neq 0$.
		\item (Cartan formula) $Q^r (xy) = \bigoplus_{p+q = r} Q^p (x) Q^q (y)$.
		\item (Adem relations) $Q^r Q^s = \sum_i (-1)^{r+i}\binom{(p-1)(i-s)-1}{pi-r} Q^{r+s-i} Q^i$ for $r > ps$.
		\item (Stability) The natural suspension $\Sigma Q_r : \pi_{n+1} \to \pi_{n+r+1}$ coincides with $Q_r$.
		\item (Bockstein) If $p > 2$ and $r \geq ps$, 
			\begin{align*}
				Q^r \beta Q^s = \sum_i (-1)^{r+i} \binom{(p-1)(i-s)-1}{pi-r} \beta Q^{r+s+i} Q^i - (-1)^{r+i}\binom{(p-1)(i-s)-1}{pi-r-1}Q^{r+s-i}Q\beta Q^i
			\end{align*}
	\end{enumerate}
	Suppose $p = 2$. Let $V_0 $ be the bigraded $k$-vector space with basis elements $Q^J  := Q^{j_1} \dots Q^{j_r}$, where $J = (j_1,\dots , j_r)$ is a sequence of non-negative integers such that $j_i \leq 2j_{i+1}$ and $j_1 - j_2 - \dots - j_r > 0$. The element $Q^J$ has homological grading $j_1 + \dots + j_r$ and polynomial grading $p^r$.

	Suppose $p > 2$. Let $V_0$ be the bigraded $k$-vector space with basis elements $Q^J   := \beta^{\epsilon_1} Q^{j_1} \dots \beta^{\epsilon_r} Q^{j_r} $, where $J = (j_1, \epsilon_1, \dots , j_r, \epsilon_r)$ is a sequence with the $j_i$ non-negative integers and $\epsilon_i \in \{0,1\}$ such that $ps_i - \epsilon_i \geq s_{i-1}$ for $2 \leq j \leq r$ and
	\begin{align*}
		2s_1 - \sum_{i = 2}^r 2s_i (p-1) - \epsilon_i > 0
	\end{align*}
	$Q^Je_0$ has homological degree $ \sum_{i=1}^r 2s_i (p-1) - \epsilon_i$ and polynomial grading $p^r$.

	Notice that, for every $i$, the subspace $(V_0)_i$ of elements in homological degree $i$ is finite-dimensional.

	If $p = 0$, we set $V_0 = k$, with polynomial degree $1$ and homological degree $0$.
\end{construction}

\begin{theorem}[\cite{Hinfty}, IX, Theorem 2.1]
	\label{theorem-DLfinite}
	There is an isomorphism of bigraded $k$-algebras
	\begin{align*}
		\pi_* \Sym^* (k[0]) \cong F_{\gr}(V_0)
	\end{align*}
\end{theorem}

\begin{rem}
	This isomorphism is also true when the characteristic of $k$ is zero by \cref{prop-derivedcomparison} (3).
\end{rem}

\begin{rem}
	This gives us a good understanding of the functor $\Vect_k^\omega \to \Gr^2 \Vect_k$ given by $V \mapsto \pi_* \Sym^* (V[0])$. Since the free algebra functor takes coproducts to tensor products, we see that
	\begin{align*}
		\pi_* \Sym^* (V[0]) \cong F_{\gr}(V_0 \otimes V)
	\end{align*}
	However, in order to make this isomorphism natural we need to adjust the $k$-action on $V_0$ by an appropriate power of the Frobenius.
\end{rem}

\begin{construction}[Frobenius twist]
	We endow $V_0$ with the structure of a $k$-bimodule as follows: we give it its natural left $k$-action, and a right $k$-action where if $x$ is of polynomial degree $r$, $x \cdot \lambda := \lambda^{p^r} x$.

	To emphasize this Frobenius twist, for $V \in \Vect_k^\omega$, we will write $V_0 \otimes^{\Phi} V$ for the tensor product with respect to the right $k$-action described above. The $k$-action on $V_0 \otimes ^{\Phi} V$ is given by the left $k$-action on $V_0$.

	For $V \in \Vect_k^\omega$, $V_0 \otimes^{\Phi} V$ must be in $\Gr^2 \Vect_k^\omega$ (this is easily verifiable for $V = k$ and we can expand other finite-dimensional vector spaces as sums). Thus, by right Kan extending we obtain a functor $\ProVect \to \Gr^2 \ProVect$, which we will still write as $V \mapsto V_0 \otimes^{\Phi} V$.
\end{construction}

\begin{prop}
	\label{prop-DLprofinite}
	Let $V \in \ProVect$. Then, there is a natural isomorphism of bigraded ultrasolid algebras
	\begin{align*}
		F_{\gr}(V_0 \otimes^{\Phi} V) \cong \pi_* \Sym^* (V[0]) 	
	\end{align*}
	where the element $Q^J \otimes^{\Phi} v$ is sent to $Q^J (v)$.
\end{prop}

\begin{proof}
	We will first show the above natural isomorphism for $V \in \Vect_k^\omega$. Then it is true on objects by \cref{theorem-DLfinite}. We must now prove naturality.

	Write $\alpha: V_0 \otimes^{\Phi} V \to \pi_* \Sym^* (V[0])$ for the map given by $Q^J \otimes^{\Phi} v \mapsto Q^J (v)$. If $f: V \to W$ is a map in $\Vect_k^\omega$, we need to show that for $v \in V$, $Q^J (f(v)) = \alpha (f^* (Q^J \otimes^{\Phi} v))$, where $f^* : V_0 \otimes^{\Phi} V \to V_0 \otimes^{\Phi} W$ is induced by the map $V \to W$. Since Dyer-Lashof operations are additive and so is the tensor product, it is enough to show this for $V = W = k$. Hence, the map is given by multiplying by some $\lambda \in k$.

	Then, using the Cartan formula, instability and Frobenius, we get that $Q^J (\lambda v) = \lambda^{p^r} Q^J (v)$, where $r$ is the polynomial degree of $Q^J$. On the other hand, $Q^J \otimes^{\Phi} \lambda v = Q^J \cdot \lambda \otimes^{\Phi} v = \lambda^{p^r} (Q^J \otimes v)$. This shows the isomorphism is natural.

	To extend the natural isomorphism for all of $\ProVect$, it suffices to show that both functors commute with inverse limits. Each $\pi_* \Sym^n (-[0]): \ProVect \to \Gr^2 \Solid_k^\heartsuit$ commutes with inverse limits by \cref{lemma-Symnlimits} and \cref{prop-complexesprohomology}, and each of these pieces has different polynomial grading. It follows that $\pi_* \Sym^* (-[0])$ commutes with inverse limits.

	The functor $V \mapsto V_0 \otimes^{\Phi} V$ commutes with inverse limits by definition. Since the tensor product of profinite vector spaces commutes with inverse limits and finite colimits, each piece $(V_0 \otimes^{\Phi} V)_{\Sigma_n}$ commutes with inverse limits. Each of these pieces is made up of elements of polynomial degree at least $n$, so $F_{\gr}(V_0 \otimes^{\Phi} -): \ProVect \to \Gr^2 \Solid_k^\heartsuit$ must also commute with inverse limits.
\end{proof}

\subsection{The cotangent complex}

We will use the formalism of (\cite{HA}, Section 7) in order to be able to talk about the cotangent complex. We will need a few adjustments since usually presentability of categories is assumed. As usual, we will make the construction one cardinal at a time and then take a filtered colimit.

\begin{construction}
	Let $\sigma$ be an uncountable strong limit cardinal and let $\CAlg_k^{\blacksquare, \sigma}$ be the $\infty$-category of connective $\EE_\infty$ $\sigma$-ultrasolid $k$-algebras. Using (\cite{HA}, Theorem 7.3.4.18), there exists a tangent bundle $T_\sigma \to \CAlg_k^{\blacksquare, \sigma}$, where $T_\sigma$ consists of a category of pairs $(A,M)$ where $A \in \CAlg_k^{\blacksquare,\sigma}$ and $M \in \Sp (\CAlg_{/A}^{\blacksquare,\sigma}) \simeq \Solid_{A}^\sigma$ (\cite{HA}, Corollary 7.3.4.14).

	This leads to adjunctions
	\begin{align*}
		L: \CAlg_k^{\blacksquare, \sigma} \leftrightarrows T_\sigma: G
	\end{align*}
	where the left adjoint sends $A \in \CAlg_k^{\blacksquare,\sigma}$ to $L_A \in \Solid_A$ and $G$ sends $(A,M)$ to $A \oplus M$, the trivial square-zero extension of $A$ by $M$.

	We can form a relative version of the cotangent complex. That is, for any map $A \to B$ in $\CAlg_k^{\blacksquare, \sigma}$ we can assign $L_{B/A} \in \Solid_B$ such that there is always a cofibre sequence
	\begin{align*}
		B \otimes_A L_A \to L_B \to L_{B/A}
	\end{align*}
	The absolute cotangent $L_A$ is actually $L_{A/k}$, since $k$ is initial in the $\infty$-category $\CAlg_k^{\blacksquare,\sigma}$ (\cite{HA}, Corollary 7.3.3.15).

	More generally, for any maps $A \to B \to C$ in $\CAlg_k^{\blacksquare, \sigma}$, there is a cofibre sequence
	\begin{align*}
		C \otimes_B L_{B/A} \to L_{C/A} \to L_{C/B}
	\end{align*}
\end{construction}

\begin{rem}
	If $\sigma < \sigma'$, we can compute the cotangent complex of $A \in \CAlg_k^{\blacksquare, \sigma}$ either in $\CAlg_k^{\blacksquare, \sigma}$ or $\CAlg_k^{\blacksquare, \sigma'}$ by left Kan extending, leading to $L_A^\sigma$ and $L_A^{\sigma'}$. However, since left Kan extension is adjoint to restriction, for $V \in \Solid_{k, \geq 0}^{\sigma'}$.
	\begin{align*}
		\Map_{\Solid_k^{\sigma'}} (L_A^{\sigma'}, V) \simeq  \Map_{\CAlg_{k/ /A}^{\blacksquare, \sigma'}} (A, A \oplus V) \simeq \Map_{\CAlg_{k/ /A}^{\blacksquare, \sigma}} (A, A \oplus V) \simeq \Map_{\Solid_k^{\sigma}} (L_A^{\sigma}, V)
	\end{align*}
	so that $L_A^{\sigma'}$ is the left Kan extension of $L_A^{\sigma}$.

	This means that the cotangent complex does not change as we increase cardinality, so for $A \in \CAlg_k^\blacksquare$, we have a well-defined cotangent complex.
\end{rem}

The cotangent complex is useful for detecting connectivity. The following is a combination of (\cite{HA}, Lemma 7.4.3.17) and the proof of (\cite{HA}, Corollary 7.4.3.2).

\begin{lemma}
	\label{lemma-cotangentcomplexiso}
	Suppose that $f: A \to B$ is an $n$-connective map in $\CAlg_k^{\blacksquare, \cn}$. That is, $\pi_i (A) \to \pi_i (B)$ is an isomorphism for $i < n$ and surjective for $i = n$. Then, $L_{B/A}$ is $n$-connective. The converse holds if $f$ is an isomorphism on $\pi_0$.
\end{lemma}

We recall the following definitions and results from (\cite{HA}, \S 7.4.1).

\begin{defi}
	A map $A \to B$ in $\CAlg_k^\blacksquare$ is a \emph{square-zero extension} if there is a pullback diagram
	\[
		\begin{tikzcd}
			A \arrow{r} \arrow{d} & B \arrow{d}{d_\eta} \\
			B \arrow{r}{d_0} & B \oplus M[1]
		\end{tikzcd}
	\]
	where $M \in \Solid_B$, the maps $d_0$ and $d_\eta$ are maps in $\CAlg^\blacksquare_{k/ /B}$ and the map $d_0$ is classified by the zero map $L_B \to M[1]$ (so it is the $B$-algebra map $B \to B \oplus M[1]$).

	A map $f: A \to B$ is an \emph{$n$-small extension} if
	\begin{enumerate}
		\item $\fib(f) \in \Solid_{k, [n,2n]}$
		\item The multiplication map $\fib(f) \otimes_A \fib(f) \to B$ is null-homotopic.
	\end{enumerate}
\end{defi}

\begin{rem}
	Suppose the map $f: A \to B$ is a square-zero extension. This means we have a commutative diagram
	\[
		\begin{tikzcd}
			A \arrow{r}{f} \arrow[d, swap, "g"] & B \arrow[d, swap, "d_\eta"]  \arrow{ddr}{\id} & \\
			B \arrow{r}{d_0} \arrow[drr, swap, "id"]  & B \oplus M[1] \arrow{dr} & \\
									  & & B
		\end{tikzcd}
	\]
	which implies $g = f$. This means that in the pullback diagram for a square-zero extension both maps $A \to B$ always coincide.
\end{rem}

\begin{prop}[\cite{HA}, Corollary 7.4.1.27]
	\label{prop-nsmallextensions}
	Any $n$-small extension is a square-zero extension.	
\end{prop}

We will later need the following lemma to understand maps with vanishing cotangent complex. The main argument can be found in the proof of (\cite{SAG}, Lemma B.1.2.1).

\begin{lemma}
	\label{lemma-II2}
	Let $f: A \to B$ be a map of connective $\EE_\infty$ ultrasolid $k$-algebras such that
	\begin{enumerate}
		\item $\pi_0 (A) \to \pi_0 (B)$ is surjective.
		\item $\pi_1 L_{B/A} = 0$.
	\end{enumerate}
	Let $I = \ker (\pi_0 (A) \to \pi_0 (B))$. Then, $I = I^2$.
\end{lemma}

\begin{proof}
	Let $R = \pi_0 (A)/I^2$. By \cref{prop-nsmallextensions}, we have a pullback diagram
	\[
		\begin{tikzcd}
			R \arrow{r} \arrow{d} & \pi_0 B \arrow{d}{\eta} \\
			\pi_0 B \arrow{r}{\eta_0} & \pi_0 B \oplus \Sigma (I/I^2)
		\end{tikzcd}
	\]
	where the maps $\eta_0, \eta$ are in $\CAlg^\blacksquare_{k/ /\pi_0 (B)}$ and $\eta_0$ corresponds to the natural inclusion.
	
	Since $L_{B/A}$ is $2$-connective, we get that
	\begin{align*}
		\Map _{\CAlg^\blacksquare_{A/ /\pi_0 (B)}} (B, \pi_0 (B) \oplus \Sigma(I/I^2)) & \simeq \Map _{\Solid_B} (L_{B/A}, \Sigma(I/I^2)) \\
		  & \simeq *
	\end{align*}
	Hence, the map $B \xrightarrow{\tau_{\leq 0}} \pi_0 (B) \xrightarrow{\eta} \pi_0 (B) \oplus \Sigma (I/I^2)$ is the one corresponding to the composition $B \xrightarrow{\tau_{\leq 0}} \pi_0 (B) \xrightarrow{\eta_0} \pi_0 (B) \oplus \Sigma(I/I^2)$ of the truncation and the inclusion. By the universal property of the limit $R$, the truncation map $B \to \pi_0 (B)$ factors through $R$.

	Passing to connected components, the quotient map $\pi_0 (A) /I^2 \to \pi_0 (A)/I$ admits a section $s: \pi_0 (A)/I \to \pi_0 (A)/I^2$ as a map of $\pi_0 (A)$-algebras. We have the triangle
	\[
		\begin{tikzcd}
			\pi_0 (A) \arrow{r} \arrow{dr} & \pi_0 (A)/I \arrow{d}{s} \\
													 & \pi_0 (A)/I^2
		\end{tikzcd}
	\]
	where the maps with source $\pi_0 (A)$ are the natural quotient maps. Hence, the composition $I \xhookrightarrow{} \pi_0 (A) \to \pi_0 (A)/I^2$ is trivial, so $I \subseteq I^2$. The other inclusion is trivial so $I = I^2$.
\end{proof}

We will now show that each step in the Postnikov tower is a square-zero extension. This was first shown for $\EE_\infty$-rings by Basterra \cite{Basterra} and Kriz \cite{kriz}. 

We have the following characterization of truncated objects and the truncation functors. We adapt the proof of (\cite{HA}, Proposition 7.1.3.14).

\begin{prop}
	\label{prop-truncated}
	Let $R \in \CAlg_k^{\blacksquare, \cn}$ and $A \in \CAlg_R^{\blacksquare, \cn}$, with $R$ connective. Then, the following are equivalent:
	\begin{enumerate}
		\item $\pi_i (A) = 0$ for $i > n$.
		\item $A$ is an $n$-truncated object in $\CAlg_R^{\blacksquare, \cn}$. That is, for each $B \in \CAlg_R^{\blacksquare, \cn}$, the space $\Map_{\CAlg_R^\blacksquare} (B,A)$ is $n$-truncated (so its $i$th homotopy group vanishes for $i > n$).
	\end{enumerate}
	Write $\tau_{\leq n} : \CAlg_R^\blacksquare \to \tau_{\leq n}\CAlg_R^{\blacksquare}$ for the left adjoint to the inclusion of $n$-truncated objects. Then, it coincides with truncation on $\Solid_k$. That is, the underlying chain complex of $\tau_{\leq n} A$ is $n$-truncated and the map $A \to \tau_{\leq n} A$ is $n$-connective.
\end{prop}

\begin{proof}
	((1) $\implies$ (2)) Suppose $\pi_i (A) = 0$ for $i > n$. Consider the functor $X: (\CAlg_R^\blacksquare)^{\op} \to \Space$ represented by $A$. We wish to show that for any $B \in \CAlg_R^\blacksquare$, the space $X(B)$ is $n$-truncated.

	The functor $X$ commutes with limits and $n$-truncated spaces are closed under limits (\cite{HTT}, Proposition 5.5.6.5), so we only need to show this for $B$ of the form $\Sym_R^* (R \otimes V)$, where $V \in \Pro(\Vect_k^\omega)$, since these objects generate $\CAlg_R^{\blacksquare, \cn}$ under sifted colimits. Then,
	\begin{align*}
		\Map_{\CAlg_R^\blacksquare} (\Sym_R^* (R \otimes V), A) \simeq \Map_{\Solid_k} (V, A)
	\end{align*}
	Clearly, $A$ is an $n$-truncated object in $\Solid_{k, \geq 0}$ so we are done.

	((2) $\implies$ (1)) If $A$ is an $n$-truncated object in $\CAlg_R^{\blacksquare, \cn}$, we have that for $V \in \Solid_k$
	\begin{align*}
		\Map_{\CAlg_R^\blacksquare} (\Sym_R^* (R \otimes V), A) \simeq \Map_{\Solid_k} (V, A)
	\end{align*}
	It follows that $A$ is an $n$-truncated object in $\Solid_{k, \geq 0}$ so that $\pi_i (A) = 0$ for $i > n$.

	For the last part, we can just apply (\cite{HA}, Proposition 2.2.1.9).
\end{proof}

\begin{cor}[\cite{HA}, Lemma 7.4.1.28]
	\label{cor-EinftyPostnikov}
	Postnikov towers converge in the $\infty$-category $\CAlg_k^\blacksquare$ and each map in the Postnikov tower
	\begin{align*}
		\dots \to \tau_{\leq 2} A \to \tau_{\leq 1} A \to \tau_{\leq 0} A
	\end{align*}
	is a square-zero extension.
\end{cor}

\subsection{Complete profinite $\EE_\infty$ ultrasolid $k$-algebras}
\label{subsection-completeprofiniteEinfty}

\begin{defi}
	We say $R \in \CAlg_{k/ /k}^\blacksquare$ is \emph{complete and profinite} if the underlying chain complex is in $\Solid_{k, \geq 0}^{\aperf}$ (so $\pi_i (R) \in \ProVect$ for all $i \geq 0$) and $\pi_0 (R) \in \CAlg_{k/ /k}^{\blacksquare, \heartsuit}$ is complete and profinite.
\end{defi}

\begin{rem}
	Notice that we are assuming connectivity of complete profinite objects.
\end{rem}

\begin{rem}
	Remember that the $\infty$-category $\CAlg_k^\blacksquare$ are algebras over the $\Sym^*$ monad. We can describe the $\infty$-category of augmented algebras $\CAlg_{k/ /k}^\blacksquare$ as algebras over the augmented monad, which we will write as $\Sym^*_{\nu}$ (equivalently, these are nonunital commutative algebra objects in $\Solid_k$). That is, for $V \in \Solid_k$,
	\begin{align*}
		\Sym^*_{\nu} (V) = \bigoplus_{n > 0} \Sym^n (V)
	\end{align*}
	Notice that the indexing starts at $1$, whereas it starts at $0$ for the $\Sym^*$ monad.

	We will sometimes write $\Sym^*_{\nu} (V)$ for the corresponding free algebra in $\CAlg_{k/ /k}^\blacksquare$, to emphasize that it is equipped with an augmentation map $\Sym^*_{\nu} (V) \to k$.
\end{rem}

\begin{construction}[completed free $\EE_\infty$ ultrasolid $k$-algebras]
	\label{construction-completedfree}
	Write $\Comm^{\nu}$ for the non-unital commutative operad. For $i > 0$, we can obtain the truncation of this operad $\Comm^{\nu} _{\leq i}$ (\cite{gijs}, \S 4.1), which satisfies 
	\begin{align*}
		\Comm^{\nu}_{\leq i} (n) = \begin{cases}
										\Comm^{\nu} (n) \; \; & \text{if } n \leq i \\
										*         \; \; & \text{otherwise}
									\end{cases}
	\end{align*}
	Algebras over this operad are precisely those nonunital $\EE_\infty$ $k$-algebras $R$ equipped with homotopy coherent trivialisations of every multiplication map $R^{\otimes n} \to R$ for $n > i$.

	We then have a sequence of $\infty$-operads
	\begin{align*}
		\Comm^{\nu} \to \dots \to \Comm^{\nu}_{\leq i} \to \dots \to \Comm^{\nu}_{\leq 2} \to \Comm^{\nu}_{\leq 1}
	\end{align*}
	By (\cite{HA}, Example 4.7.3.11), $\Comm^{\nu}_{\leq i}$-algebras in $\Solid_k$ are equivalent to algebras over the free $\Comm^{\nu}_{\leq i}$-algebra monad on $\Solid_k$, which we will write as $\Sym^{ \leq i}_{\nu}$. This monad is given by (\cite{HA}, Proposition 3.1.3.13)
	\begin{align*}
		\Sym^{\leq i}_{\nu} (X) = \bigoplus_{0 < n \leq i} X^{\otimes n}_{h \Sigma_n}
	\end{align*}
	We hence obtain a sequence of monads
	\begin{align*}
		\Sym^*_{\nu} \to \dots \to \Sym^{\leq i}_{\nu} \to \dots \to \Sym^{\leq 2}_{\nu} \to \Sym^{\leq 1}_{\nu}
	\end{align*}
	Define the monad
	\begin{align*}
		\widehat{\Sym}^*_{\nu} := \varprojlim \Sym^{\leq i}_{\nu}
	\end{align*}
	Monads on $\Solid_k$ are $\EE_1$-algebras in $\Fun(\Solid_k, \Solid_k)$, where the symmetric monoidal structure is given by composition. Hence, limits of monads are computed as the limit of the underlying functors (\cite{HA}, Corollary 3.2.2.5), so for an object $X \in \Solid_k$ we have
	\begin{align*}
		\widehat{\Sym}^* _{\nu} (X) \simeq \prod_{i > 0} X^{\otimes n}_{h \Sigma_n}
	\end{align*}
	The previous tower of monads induces a map of monads $\Sym^* _{\nu} \to \widehat{\Sym}^* _{\nu}$, so we get a functor
	\begin{align*}
		\Alg_{\widehat{\Sym}^*_{\nu}} (\Solid_k) \to \CAlg_{k/ /k}^{\blacksquare}
	\end{align*}
	that is the identity on the underlying chain complexes. This means that any $\widehat{\Sym}^* _{\nu}$-algebra has the structure of an augmented $\EE_\infty$ ultrasolid $k$-algebra.
	
	Given $X \in \Solid_k$, we will abuse notation and write $\widehat{\Sym}^*_{\nu}(X)$ for the augmented $\EE_\infty$ ultrasolid $k$-algebra that is the image of the free $\widehat{\Sym}^*_{\nu}$-algebra on $X$. Additionally, we will write $\widehat{\Sym}^* (X)$ for the $\EE_\infty$ ultrasolid $k$-algebra obtained after forgetting the augmentation.
\end{construction}

\begin{rem}
	It is important that we took the inverse limit of the monads $\Sym^{\leq i}_{\nu}$ rather than the inverse limit of the $\infty$-operads $\Comm^{\nu}_{\leq i}$, since the latter would give a different free algebra monad. This is because constructing the free algebra functor of an operad does not necessarily commute with inverse limits.
\end{rem}

We can also describe the completed free algebras in terms of Dyer-Lashof operations. We will first show a few nice properties of the completed free functor.

\begin{lemma}
	\label{lemma-completedfreelimits}
	The functor $\widehat{\Sym}^*: \Solid_{k, \geq 0} \to \Solid_{k, \geq 0}$ preserves the subcategory $\Solid_{k, \geq 0}^{\aperf}$, and commutes with inverse limits and geometric realisations in this subcategory.
\end{lemma}

\begin{proof}
	Each $\Sym^n$ preserves $\Solid_{k, \geq 0}^{\aperf}$ and commutes with geometric realisations and inverse limits in this subcategory (\cref{lemma-Symnlimits}). Then, $\widehat{\Sym}^* \simeq \prod_{n \geq 0} \Sym^n$ preserves $\Solid_{k, \geq 0}^{\aperf}$, since this category is closed under products, and clearly commutes with inverse limits. Finally, we apply \cref{prop-inversegeom} to see that it commutes with geometric realisation in $\Solid_{k, \geq 0}^{\aperf}$. 
\end{proof}

\begin{defi}
	We will write $F: \Gr \ProVect \to \Gr \Solid_k^\heartsuit$ for the free commutative algebra functor, and $\widehat{F}: \Gr \ProVect \to \Gr \Solid_k^\heartsuit$ for its completion (we can proceed in the same way as \cref{construction-completedfree} to obtain a complete free algebra functor).
\end{defi}

We will now use the notation of \cref{construction-DLoperations}.

\begin{prop}
	\label{prop-DLcompleteprofinite}
	Let $k$ be a field of characteristic $p > 0$ and $V \in \ProVect$. Then, there is a natural isomorphism of graded ultrasolid algebras
	\begin{align*}
		\pi_* \widehat{\Sym}^* (V[0]) \cong \widehat{F} (V_0 \otimes^{\Phi} V)
	\end{align*}
	where we have forgotten the polynomial grading on $V_0 \otimes^{\Phi} V$.
\end{prop}

\begin{proof}
	The main strategy is to use the isomorphism in \cref{prop-DLprofinite}. After forgetting polynomial grading,
	\begin{align*}
		\pi_* \Sym^* (V[0]) \cong F(V_0 \otimes^{\Phi} V)
	\end{align*}
	Write $F^{\leq n} : \Gr \Solid_k^\heartsuit \to \Gr \Solid_k^\heartsuit$ for the functor $V \mapsto \bigoplus_{k \leq n} V^{\otimes n}_{\Sigma_n}$. Then, we have that $ \pi_* \Sym^{\leq n} (V[0]) \ncong F^{\leq n} (V_0 \otimes^{\Phi} V)$ due to the natural polynomial grading on $V_0$. 

	To solve this, write $F_{\gr}^{\tau_{\leq n}} : \Gr^2 \Solid_k^\heartsuit \to \Gr \Solid_k^\heartsuit$ for the functor obtained by applying $F_{\gr}$, discarding everything in polynomial grading greater than $n$ and then forgetting polynomial grading. Then, we do have an equivalence
	\begin{align*}
		\pi_* \Sym^{\leq n} (V[0]) \cong F_{\gr}^{\tau_{\leq n}} (V_0 \otimes^{\Phi} V)
	\end{align*}
	Both sides are naturally algebras, and this is an isomorphism of algebras. $F^{\leq n} (V_0 \otimes^{\Phi} V)$ has a natural bigrading, so we also have a map of graded algebras
	\begin{align*}
		F^{\leq n} (V_0 \otimes V) \to F_{\gr}^{\tau_{\leq n}} (V_0 \otimes^{\Phi} V)
	\end{align*}
	discarding any element in polynomial degree greater than $n$. It is clear that this becomes an isomorphism after taking the inverse limit over $n$, since at the level of graded ultrasolid modules both are equal to $\prod_{n \geq 0} (V_0 \otimes^{\Phi} V)^{\otimes n}_{\Sigma_n}$.
\end{proof}

Our goal now will be to prove that the completion map $\Sym^* (V) \to \widehat{\Sym}^*  (V)$ has trivial cotangent complex. This is the spectral ultrasolid analogue of the map $k[x] \to k[[x]]$ having trivial cotangent complex.

We will first reduce the statement to computing a tensor product.

\begin{lemma}
	\label{lemma-idempotent}
	Let $A = \Sym^* (V)$ and $B = \widehat{\Sym}^* (V)$ for $V \in \Solid^{\aperf}_{k, \geq 0}$. If the unit map induces an isomorphism $ k \simeq B \otimes_A k $, then $ L_{B/A} \otimes_B k \simeq 0$.
\end{lemma}

\begin{proof}
	We have that 
	\begin{align*}
		0 \simeq L_{k/k} \simeq L_{B \otimes_A k /k}
	\end{align*}
	Since the cotangent complex commutes with colimits, there is a pushout diagram
	\[
		\begin{tikzcd}
			L_{A/k} \otimes_A k \arrow{r} \arrow{d} & L_{B/k} \otimes_B k \arrow{d} \\
			L_{k/ k} \simeq 0 \arrow{r} & L_{B \otimes_A k/k} \simeq 0
		\end{tikzcd}
	\]
	so the top map must be an equivalence. By the fundamental cofibre sequence, we deduce that $L_{B/A} \otimes_B k \simeq 0$.
\end{proof}

Hence, we are now reduced to computing the tensor product $\widehat{\Sym}^* (V) \otimes_{\Sym^* (V)} k$. Computing this derived tensor product is difficult because there is no obvious resolution of either $k$ or $\widehat{\Sym}^* (V)$ as free $\Sym^* (V)$-modules. However, we can learn a lot about this tensor product from the associated spectral sequence.

We have the following analogue of (\cite{HA}, Proposition 7.2.1.19).

\begin{lemma}
	\label{lemma-spectralsequence}
	Let $A$ be an $\EE_\infty$ ultrasolid $k$-algebra and let $M,N \in \Solid_A$. Then, there is a spectral sequence $\{E_r^{p,q}, d_r\}_{r, \geq 2}$ with $E_2$ page given by $E_2^{p,q} = \Tor_p^{\pi_* A} (\pi_* M, \pi_* N)_q$, converging to $\pi_{p+q} (M \otimes_A N)$.
\end{lemma}

\begin{rem}
	A useful notion is that of \emph{flatness} (\cite{HA}, \S 7.2.2). If $N$ is a flat $A$-module, we can write $\pi_* N \simeq \pi_0 N \otimes^L_{\pi_0 A} \pi_* A$. When this identity holds, by associativity of the tensor product, the $E_2$ term of the above spectral sequence turns into $\Tor_p^{\pi_0 A} (\pi_* M, \pi_0 N)_q$.

	In our case, we are not going to be able to prove that $\widehat{\Sym}^* (V)$ is a flat $\Sym^* (V)$-module, but we will still be able to prove the above identity in order to simplify the spectral sequence.
\end{rem}

\begin{prop}
	\label{prop-tensorpi0}
	Let $A = \Sym^* (V)$ and $B = \widehat{\Sym}^* (V)$ for $V \in \Pro(\Vect_k^\omega)$. Then, there is an equivalence
	\begin{align*}
		\pi_* (B) \simeq \pi_0 (B) \otimes^L_{\pi_0 (A)} \pi_* (A)
	\end{align*}
\end{prop}

\begin{proof}
	By \cref{prop-DLprofinite}, we know that $\pi_* (\Sym^* (V [0] )) \cong F(V_0 \otimes V)$, where $V_0 \in \Gr \Vect_k^\omega$. Additionally, $\pi_* \widehat{\Sym}^* (V \otimes V_0) \cong \widehat{F}(V_0 \otimes V)$ (\cref{prop-DLcompleteprofinite}).

	We now claim that for $U,W \in \Gr\ProVect$, $\widehat{F}(U \oplus W) \cong \widehat{F} (U) \otimes \widehat{F}(W)$. There is a clear map $\widehat{F}(U) \to \widehat{F}(U \oplus W)$ induces by the inclusion, and similar for $W$. Hence, we have a map $\widehat{F}(U) \otimes \widehat{F}(W) \to \widehat{F}(U \oplus W)$. We now check that this is an equivalence on the underlying graded ultrasolid modules. We can explicitly compute 
	\begin{align*}
		\widehat{F}(U \oplus U') \cong \prod_{n \geq 0} (U \oplus U')^{\otimes n}_{\Sigma_n} \cong \prod_{n \geq 0} \prod_{0 \leq i \leq n} U ^{\otimes i}_{\Sigma_i} \otimes U'^{\otimes n-i}_{\Sigma_{n-i}} \cong \prod_{i,j \geq 0} U^{\otimes i}_{\Sigma_i} \otimes U'^{\otimes j}_{\Sigma_j} \cong \widehat{F} (U) \otimes \widehat{F}(U')
	\end{align*}
	where the orbits are taken in the category of graded ultrasolid modules. Here we crucially used that inverse limits of profinite vector spaces commute with the tensor product (\cref{prop-profinite}). 

	Write $W := V_0 \otimes V$ to simplify notation, and let $W_0 \subset W$ be the elements in degree $0$, and $W_{\geq 1} \subset W$ the elements in positive degree, so that $W = W_0 \oplus W_{\geq 1}$. By the above we have that $\widehat{F} (W) \cong \widehat{F} (W_0) \otimes \widehat{F}(W_{\geq 1})$. Since all ultrasolid modules are flat, this is also true of the derived tensor product. Obviously, the same identity is true of the uncompleted free functor, since it sends finite sums to tensor products, so $F(W) \cong F(W_0) \otimes F(W_{\geq 1})$.

	Also, $\widehat{F}(W_{\geq 1}) \cong F(W_{\geq 1})$ since there are only finitely many summands of each homological degree. We can finally conclude that
	\begin{align*}
		\widehat{F} (W) & \simeq \widehat{F}(W_0) \otimes^L F(W_{\geq 1}) \\
							 & \simeq \widehat{F}(W_0) \otimes^L_{F(W_0)} (F(W_0) \otimes^L F(W_{\geq 1})) \\
							 & \simeq\widehat{F} (W_0) \otimes^L_{F(W_0)} F(W)
	\end{align*}
	as required.
\end{proof}

Hence, all that remains to compute in the spectral sequence is $k \otimes^L_{\pi_0 \Sym^* (V)} \pi_0 \widehat{\Sym}^* (V)$. This can be done via Koszul resolutions. For $V \in \Pro(\Vect_k^\omega)$, $\pi_0 \Sym^* (V) = \LSym^* (V)$ and $\pi_0 \widehat{\Sym}^* (V) = \widehat{\LSym}^*(V)$.

\begin{example}
	Suppose $R$ is an ordinary commutative ring and we have an ideal $I = (x_1,\dots , x_n)$ in $R$ generated by a regular sequence. Then, there is an exact sequence
	\begin{align*}
		0 \to \bigwedge\nolimits^n R^n \to \dots \to \bigwedge\nolimits^2 R^n \to R^n \to R/I \to 0
	\end{align*}
	where $\bigwedge^i$ refers to the $i$th exterior power (\cite{weibel}, Corollary 4.5.5). This gives an explicit resolution of $R/I$ by free $R$-modules. In the case $R = k[x_1,\dots , x_n]$, the variables $x_1,\dots , x_n$ form a regular sequence so we always have an exact complex as above. 
\end{example}

We will now generalize these resolutions to the free algebras on a profinite vector space.

\begin{construction}[Exterior powers]
	Let $V \in \Solid_k^\heartsuit$. For $n > 0$, we can consider the action of $\Sigma_n$ on $V^{\otimes n}$, twisted by the sign action. We define the \emph{$n$th exterior power} of $V$ to be
	\begin{align*}
		\bigwedge\nolimits^n V := (V^{\otimes n, \text{sign}})_{\Sigma_n}
	\end{align*}
	where the orbits are taken in the category $\Solid_k^\heartsuit$.

	It is clear from the definition that for $V = \prod_I k$, when $\operatorname{char} k \neq 2$, 
	\begin{align*}
		\bigwedge\nolimits^n V  = \prod_{J \subset I, |J| = n} k
	\end{align*}
\end{construction}

We can now write a Koszul resolution of the free algebras, in a very similar way to (\cite{analyticderham}, Corollary 2.4.4).

\begin{construction}[Koszul resolutions]
	\label{construction-koszul}
	Let $R \in \CAlg_k^{\blacksquare, \heartsuit}$ and $V \in \Solid_k^\heartsuit$ with a map $s: V \to R$. Then, for each $n$, we have $n$ maps $d_1,\dots , d_n$ from $V^{\otimes n} \to V \otimes V^{\otimes n-1}$, by singling out each variable. We can compose them with the projection $V \otimes V^{\otimes n-1} \to V \otimes \bigwedge^{n-1} V \xrightarrow{s} R \otimes \bigwedge^{n-1} V$, to obtain maps $\tilde{d}_1, \dots, \tilde{d}_n: V^{\otimes n} \to R \otimes \bigwedge\nolimits^{n-1}$. The alternating sum $\sum_{i = 1}^n (-1)^i \tilde{d}_i$ is $\Sigma_n$-invariant (twisting $V^{\otimes n}$ by the sign representation), so we get an induced map $\tilde{d}: \bigwedge^n V \to R \otimes \bigwedge^{n-1} V $. Tensoring, we get a map of $R$-modules $R \otimes \bigwedge^n V \to R \otimes \bigwedge^{n-1} V$. One can check that the composition of two such maps is trivial, so we get a Koszul chain complex
	\begin{align*}	
		K(s) := \dots \to R \otimes \bigwedge\nolimits ^i V \to \dots  \to R \otimes \bigwedge\nolimits^1  V \to R \to 0
	\end{align*}
	where $R \otimes \bigwedge\nolimits^i V$ is in degree $i$.
\end{construction}

\begin{lemma}
	\label{lemma-koszul}
	Let $V \in \Solid_k^\heartsuit$ and let $i: V \to \LSym^* (V)$ be the natural inclusion. Then, the augmentation $\LSym^* (V) \to k$ induces an equivalence $K(i) \simeq k[0]$.

	Let $W \in \ProVect$ and let $j : W \to \widehat{\LSym}^* (W)$ be the natural inclusion. Then, the augmentation $\widehat{\LSym}^* (W) \to k$ induces an equivalence $K(j) \simeq k[0] $.
\end{lemma}

\begin{proof}
	The first sequence splits into a sum of sequences.
	\begin{align*}
		0 \to \bigwedge\nolimits ^n  V \to \dots \to V^{\otimes n-1}_{\Sigma_{n-1}}  \otimes \bigwedge \nolimits ^1 V  \to V^{\otimes n}_{\Sigma_n} \to 0
	\end{align*}
	So we only need to show that these are exact. These sequences commute with filtered colimits, which are exact, so we may just show this for $V \in \ProVect$. Given that inverse limits are exact and commute with the tensor product in $\ProVect$ (\cref{prop-profinite}), each term in the sequence commutes with inverse limits, so we can just show the result for $V \in \Vect_k^\omega$. But this follows from writing $k$ as a complete local intersection $k = k[x_1,\dots , x_n] /(x_1,\dots , x_n)$ (\cite{weibel}, Corollary 4.5.5).

	The second sequence splits into a product of the finite sequences, and we are now done by noting that inverse limits are exact in $\ProVect$.
\end{proof}

\begin{theorem}
	\label{theorem-cotangentcomplexcompletion}
	Let $V \in \Solid_{k, \geq 0}^{\aperf}$. Then, the completion map $\Sym^* (V) \to \widehat{\Sym}^* (V)$ has trivial cotangent fibre. That is,
	\begin{align*}
		L_{\widehat{\Sym}^* (V)/\Sym^* (V)} \otimes _{\widehat{\Sym}^* (V)} k \simeq 0
	\end{align*}
\end{theorem}

\begin{proof}
	Both the completed and uncompleted free functor commute with geometric realisations in $\Solid_{k, \geq 0}^{\aperf}$ (\cref{lemma-Symnlimits} and \cref{lemma-completedfreelimits}), and so does the cotangent complex, so it is enough to show this for $V \in \Pro(\Vect_k^\omega)$. For ease of notation, we will write $A = \Sym^* (V)$ and $B = \widehat{\Sym}^* (V)$.

	By \cref{lemma-idempotent}, it is enough to show that $B \otimes_A k \simeq k$. By \cref{lemma-spectralsequence}, there is a spectral sequence $\{E_r^{p,q}, d_r\}_{r, \geq 2}$ with $E_2$-page given by $E_2^{p,q} = \Tor_p^{\pi_* A} (\pi_* B, \pi_* k)_q$ converging to $\pi_{p+q} (B \otimes_A k)$. We can now apply \cref{prop-tensorpi0} to see that $\pi_* B \simeq \pi_* A \otimes_{\pi_0 A}^L \pi_0 B$. The $E_2$-page of the spectral sequence then turns into
	\begin{align*}
		E_2^{p,q} = \Tor_p^{\pi_0 A} (\pi_0 B, \pi_* k)_q = \Tor_p^{ \LSym^* (V)} (\widehat{\LSym}^* (V), \pi_* k)_q
	\end{align*}
	Since $\pi_* k = 0$ for $* > 0$ we see that $E_2^{p,q} = 0$ for $q > 0$. All that remains is to compute $\widehat{\LSym}^* (V) \otimes^L_{\LSym^* (V)} k$. But by \cref{lemma-koszul}, this is equal to $k[0]$. It follows that the spectral sequence degenerates and $B \otimes_A k \simeq k$.
\end{proof}

We will now show that $\widehat{\Sym}^* (V)$ indeed behaves like a free complete algebra.

\begin{prop}
	\label{prop-mapsfromcomplete}
	Suppose that $A \in \CAlg^\blacksquare_{k/ /k}$ is such that $\pi_0 (A) \in \CAlg_{k/ /k}^{\blacksquare, \heartsuit}$ is complete. Let $V \in \Solid_{k, \geq 0}^{\aperf}$. Then, the map $\Sym^* (V) \to \widehat{\Sym}^*  (V)$ induces an equivalence
	\begin{align*}
		\Map_{\CAlg^\blacksquare_{k/ /k}} (\widehat{\Sym}^*  (V), A) \xrightarrow{\simeq} \Map_{\CAlg_{k/ /k}^\blacksquare} (\Sym^* (V), A)
	\end{align*}
\end{prop}

\begin{proof}
	Both $\Sym^*$ and $\widehat{\Sym}^*$ commute with geometric realisations in $\Solid_{k, \geq 0}^{\aperf}$ (\cref{lemma-Symnlimits} and \cref{lemma-completedfreelimits}), so it suffices to show this for $V \in \ProVect$.

	We can write $\Map_{\CAlg^\blacksquare_{k/ /k}} (\widehat{\Sym}^* (V), A) = \varprojlim \Map_{\CAlg^\blacksquare_{k/ /k}} (\widehat{\Sym}^* (V), \tau_{\leq n} A)$. Additionally, define $X_n := \Map_{\CAlg^\blacksquare_{k/ /k}} (\widehat{\Sym}^* (V), \tau_{\leq n} A)$ and $Y_n := \Map_{\CAlg^\blacksquare_{k/ /k}} (\Sym ^* (V), \tau_{\leq n} A)$.

	We will show by induction that the map $X_n \to Y_n$ is an equivalence for all $n$. For the case $n = 0$, we can truncate objects so it remains to show the equivalent result for ultrasolid $k$-algebras. This is true by completeness of $\pi_0$ (\cref{prop-mapsfromcompletediscrete}).

	Each truncation $\tau_{\leq n+1} A \to \tau_{\leq n} A$ fits in a square-zero extension diagram (\cite{HA}, Remark 7.4.1.29). We have the following diagram, where each of the squares is a pullback diagram.
	\[
		\begin{tikzcd}
			\tau_{\leq n+1} A \arrow{r} \arrow{d} & \tau_{\leq n} A \arrow{d} \\
			\tau_{\leq n} A \arrow{r} \arrow{d} & \tau_{\leq n} A \oplus \pi_{n+1} (A) [n+2] \arrow{d} \\
			k \arrow{r} & k \oplus \pi_{n+1} (A)[n+2] 
		\end{tikzcd}
	\]	
	The horizontal map $\tau_{\leq n} A \to \tau_{\leq n} A \oplus \pi_{n+1} (A)[n+2]$ is the inclusion map, and similar for the map $k \to k \oplus \pi_{n+1} (A)[n+2]$. This means the whole rectangle is a pullback diagram.

	Write $M := \pi_{n+1} (A) [n+2]$. Now we have that
	\begin{align*}
		\Map_{\CAlg^\blacksquare_{k/ /k}} (\widehat{\Sym}^* (V) , k \oplus M) & \simeq \Map_{\Solid_k} (L_{\widehat{\Sym}^* (V)} \otimes_{\widehat{\Sym}^* (V)} k , M) \\ 
			& \simeq \Map_{\Solid_k} (V, M ) \\
		& \simeq \Map_{\CAlg^\blacksquare_{k/ /k}} (\Sym^* (V), k \oplus M)
	\end{align*}
	where we have used that $L_{\widehat{\Sym}^* (V)} \otimes_{\widehat{\Sym}^* (V)} k \simeq V$. This is because $L_{\Sym^* (V)/k} \simeq \Sym^* (V) \otimes V$ (\cite{HA}, Proposition 7.4.3.14) and $L_{\widehat{\Sym}^* (V)/\Sym^* (V)} \otimes _{\widehat{\Sym}^* (V)} k \simeq 0$ (\cref{theorem-cotangentcomplexcompletion}), so by the fundamental cofibre sequence $L_{\widehat{\Sym}^* (V)} \otimes_{\widehat{\Sym}^* (V)} k \simeq V$.

	Applying the functor $\Map_{\CAlg^\blacksquare_{k/ /k}} (\widehat{\Sym}^* (V), -)$ to the square-zero extension diagram, we get a pullback diagram
	\[
		\begin{tikzcd}
			X_{n+1} \arrow{r} \arrow{d} & X_n \arrow{d} \\
			* \arrow{r} & \Map_{\CAlg^\blacksquare_{k/ /k}} (\widehat{\Sym}^* (V), k \oplus M)
		\end{tikzcd}
	\]
	and similar for $Y_n$. By induction and the equivalence shown above, the completion map $\Sym^* (V) \to \widehat{\Sym}^* (V)$ induces an equivalence on each element of the pullback, so $X_{n+1} \simeq Y_{n+1}$.

	Hence, the map is an equivalence at each step of the Postnikov tower, so it assembles into the desired equivalence.
\end{proof}

\begin{rem}
	Using the same proof, we may show that if we have a map $A \to B$ with $L_{B/A} \simeq 0$ and the map $\pi_0 (A) \to \pi_0 (B)$ induces an equivalence $\Map (\pi_0 (B), \pi_0(C)) \xrightarrow{\simeq} \Map (\pi_0 (A), \pi_0 (C))$, then $\Map (B,C) \simeq \Map (A,C)$.

	We can think of this as a relative version of the fact that if $f: A \to B$ is a map with $L_{B/A} \simeq 0$ inducing an isomorphism on $\pi_0$ then it is an equivalence.
\end{rem}

This allows us to show an easy characterisation of complete profinite $\EE_\infty$ ultrasolid $k$-algebras.

\begin{prop}
	\label{prop-completemonadic}
	Let $\widehat{\CAlg} _{k/ /k}^\blacksquare$ be the $\infty$-category of augmented complete profinite $\EE_\infty$ ultrasolid $k$-algebras. Then, there is a monadic adjunction
	\begin{align*}
		\widehat{\Sym}^*_{\nu} : \Solid_{k, \geq 0}^{\aperf} \leftrightarrows \widehat{\CAlg}_{k/ /k}^\blacksquare: \forget
	\end{align*}
\end{prop}

\begin{proof}
	By \cref{lemma-completedfreelimits}, the functor $\widehat{\Sym}^*_{\nu}$ preserves $\Solid_{k, \geq 0}^{\aperf}$, so the above functors are well-defined. Remember that there is a monadic adjunction
	\begin{align*}
		\Sym^*_{\nu}: \Solid_{k, \geq 0} \leftrightarrows \CAlg_{k/ /k}^{\blacksquare} : \forget
	\end{align*}
	By \cref{prop-mapsfromcomplete}, this induces an adjunction $\widehat{\Sym}^*_{\nu} : \Solid_{k, \geq 0}^{\aperf} \leftrightarrows \widehat{\CAlg}_{k/ /k}^\blacksquare : \forget$.

	We will now show that $\widehat{\CAlg}_{k/ /k}^\blacksquare$ is closed under geometric realisations. By \cref{prop-complexesprohomology}, $\Solid_{k, \geq 0}^{\aperf}$ is closed under geometric realisations, and sifted colimits of $\EE_\infty$ ultrasolid $k$-algebras are computed on the underlying chain complex, so we only need to check completeness. Let $X_\bullet$ be a simplicial diagram in $\widehat{\CAlg}_{k/ /k}^\blacksquare$. Then,
	\begin{align*}
		\pi_0 |X_\bullet| \cong \pi_0 \colim_{i \leq 1} X_\bullet \cong \coeq (\pi_1 X \rightrightarrows \pi_0 X)
	\end{align*}
	so that $\pi_0 |X_\bullet|$ is a quotient of $\pi_0 X_0$ by a profinite ideal. Since completeness is an inverse limit condition, and inverse limits are exact in $\ProVect$, we see that $\pi_0 |X_\bullet|$ is also complete.

	The functor $\forget$ is conservative and commutes with geometric realisations that are computed in $\widehat{\CAlg}_{k/ /k}^\blacksquare$, so by the Barr-Beck-Lurie Theorem (\cite{HA}, Theorem 4.7.3.5), the adjunction is monadic.
\end{proof}

\begin{rem}
	\label{rem-completeinverse}
	Complete profinite $\EE_\infty$ ultrasolid $k$-algebras are closed under inverse limits. To see this, notice that an inverse limit of profinite chain complexes is still profinite (\cref{prop-complexesprohomology}). Additionally, inverse limits are exact, so $\pi_0$ of an inverse limit is the inverse limit of the $\pi_0$. Given that completion is a limit condition, this inverse limit is still complete.
\end{rem}

The cotangent complex is exceptionally well-behaved in the case of complete profinite $\EE_\infty$ ultrasolid $k$-algebras.

\begin{prop}
	\label{prop-cotinverselimits}
	Consider the functor
	\begin{align*}
		\cot :   \widehat{\CAlg}_{k/ /k}^\blacksquare  & \to \Solid_{k, \geq 0} \\
		R & \mapsto L_{R/k} \otimes_R k
	\end{align*}
	Then, for all $R \in \widehat{\CAlg}_{k/ /k}^\blacksquare$, $\cot (R) \in \Solid_{k, \geq 0}^{\aperf}$, and the functor $\cot$ commutes with inverse limits of complete profinite $\EE_\infty$ ultrasolid $k$-algebras.
\end{prop}

\begin{proof}
	By \cref{prop-completemonadic}, we can use the bar construction to write, for any complete profinite $R$, $R \simeq |\Bar_\bullet (\widehat{\Sym}^*_{\nu}, \widehat{\Sym}^*_{\nu}, R)|$. The functor $\cot$ commutes with sifted colimits, so by \cref{theorem-cotangentcomplexcompletion} we can write $\cot (R) \simeq |\Bar_\bullet (\id, \widehat{\Sym}^*_{\nu}, R)|$, from which it is clear that $\cot (R) \in \Solid_{k, \geq 0}^{\aperf}$.

	By \cref{lemma-completedfreelimits}, $\widehat{\Sym}^*_{\nu}$ commutes with inverse limits in $\Solid_{k, \geq 0}^{\aperf}$. Also, geometric realisations commute with inverse limits in $\Solid_{k, \geq 0}^{\aperf}$ (\cref{prop-inversegeom}). It follows that the $\cot$ functor commutes with inverse limits.
\end{proof}

\begin{prop}
	\label{prop-cotcompleteiso}
	Suppose $f: A \to B$ is a map of complete profinite $\EE_\infty$ ultrasolid $k$-algebras and $L_{B/A} \simeq 0$. Then, $f$ is an equivalence.
\end{prop}

\begin{proof}
	It is enough to show that $f$ is an isomorphism in $\pi_0$ (\cref{lemma-cotangentcomplexiso}). By the fundamental cofibre sequence, the map $k \otimes_A L_{A/k} \to k \otimes_B L_{B/k}$ is an equivalence. We have $\pi_0 (k \otimes_A L_{A/k}) = \mathfrak{m}/\mathfrak{m}^2$, where $\mathfrak{m}$ is the augmentation ideal of $\pi_0 (A)$.

	Let $\mathfrak{n}$ be the augmentation ideal of $\pi_0 (B)$. Then, the map $\mathfrak{m}/\mathfrak{m}^2 \to \mathfrak{n} / \mathfrak{n}^2$ is surjective. In particular, by ultrasolid Nakayama, the map $\mathfrak{m} \to \mathfrak{n}$ is surjective. Hence, the map $\pi_0 (A) \to \pi_0 (B)$ is surjective.

	By \cref{lemma-II2}, we have that $I = I^2$, where $I$ is the kernel of $\pi_0 (A) \to \pi_0 (B)$. Since $I \subseteq \mathfrak{m}$, we have that $I \supseteq \mathfrak{m} I \supseteq I^2 = I$. Hence, $I \mathfrak{m} = I$ and it is profinite, so by ultrasolid Nakayama, $I = 0$. Hence, $\pi_0 (A) \xrightarrow{\simeq} \pi_0 (B)$ is an isomorphism.
\end{proof}

We can actually characterise complete profinite ultrasolid $k$-algebras in a much more classical way.

\begin{defi}
	An augmented $\EE_\infty$ ultrasolid $k$-algebra $R$ is \emph{Artinian} if $R \in \Mod_{k, \geq 0}$, $\pi_0 (R)$ is local and $\pi_* (R)$ is a finite-dimensional vector space. We write $\CAlg_{k/ /k}^{\art}$ for the full subcategory of Artinian $\EE_\infty$ ultrasolid $k$-algebras.
\end{defi}

\begin{rem}
	\label{rem-artinianlimits}
	Equivalently, this is the smallest subcategory of $\CAlg_{k/ /k}^{\blacksquare, \cn}$ closed under finite limits that contains $k \oplus V$ for $V \in \Mod_{k, \geq 0}$ with $\pi_* V$ finite-dimensional (c.f. \cite{DAGX}, Proposition 1.1.11).
\end{rem}

We now reach the main result of this section.

\begin{theorem}
	\label{theorem-proartinian}
	Write $\widehat{\CAlg}_{k/ /k}^\blacksquare$ for the $\infty$-category of complete profinite $\EE_\infty$ ultrasolid $k$-algebras. Then, there is an equivalence
	\begin{align*}
		\Pro(\CAlg_{k/ /k}^{\art}) \simeq \widehat{\CAlg}_{k/ /k}^\blacksquare
	\end{align*}
\end{theorem}

\begin{proof}
	There is a clear fully faithful functor $\CAlg_{k/ /k}^{\art} \to \CAlg_{k/ /k}^\blacksquare$ that extends uniquely to an inverse-limit-preserving functor $F: \Pro(\CAlg_{k/ /k}^{\art}) \to \CAlg_{k/ /k}^\blacksquare$. Given that all Artinian $k$-algebras are complete and profinite, by \cref{rem-completeinverse} this functor lands on $\widehat{\CAlg}_{k/ /k}^\blacksquare$. By (\cite{HTT}, Proposition 5.3.5.11), we only need to show that Artinian $k$-algebras are cocompact in $\widehat{\CAlg}_{k/ /k}^\blacksquare$ and that they generate all of $\widehat{\CAlg}_{k/ /k}^\blacksquare$ under inverse limits.

	We will first show that $k \oplus V$ is cocompact, where $V \in \Mod_{k, \geq 0}$ with $\pi_* V$ finite-dimensional.	Let $\{R_i\}_{i \in I}$ be an inverse system in $\widehat{\CAlg}_{k/ /k}^\blacksquare$. Then,
	\begin{align*}
		\Map_{\CAlg_{k/ /k}^\blacksquare} (\varprojlim R_i, k \oplus V) & \simeq \Map _{\Solid_k} (\cot (\varprojlim R_i), V) \\
																						  & \simeq \Map_{\Solid_k} (\varprojlim \cot (R_i), V) \\
																						  & \simeq \varinjlim \Map (\cot (R_i), V) \\
																						  & \simeq \varinjlim \Map_{\CAlg_{k/ /k}^\blacksquare} (R_i, k \oplus V)
	\end{align*}
	where in the second equivalence we used \cref{prop-cotinverselimits} and in the third equivalence we used that $V$ is cocompact in $\Solid_{k, \geq 0}^{\aperf}$. The latter is true because clearly $V^\vee$ is compact in $\Mod_{k, \leq 0}$ and duality induces an equivalence $\Solid_{k, \geq 0}^{\aperf} \simeq \Mod_{k, \leq 0}^{\op}$.

	Finite limits of cocompact objects are cocompact, so by \cref{rem-artinianlimits}, all Artinian $k$-algebras are cocompact.

	We now must show that Artinian $k$-algebras generate all of $\widehat{\CAlg}_{k/ /k}^{\blacksquare}$ under inverse limits. Given that $\CAlg_{k/ /k}^{\art}$ is closed under finite limits, $\Pro(\CAlg_{k/ /k}^{\art})$ has all limits. Additionally, the functor $\CAlg_{k/ /k} ^{\art} \to \widehat{\CAlg}_{k/ /k}^\blacksquare$ commutes with finite limits. By (\cite{HTT}, Proposition 5.5.1.9), $F$ commutes with all limits. Since it is fully faithful, we only need to show that Artinian $k$-algebras generate all of $\CAlg_{k/ /k}^{\art}$ under limits.

	By taking Postnikov towers, we only need to show that $n$-truncated objects are generated by Artinian objects under limits. Let $R$ be complete and profinite. By \cref{prop-nsmallextensions}, for every $n$, the map $\tau_{\leq n+1} R \to \tau_{\leq n} R$ fits in a diagram
	\[
		\begin{tikzcd}
			\tau_{\leq n+1} R \arrow{r} \arrow{d} & \tau_{\leq n} R \arrow{d} \\
			\tau_{\leq n} R \arrow{d} \arrow{r} & \tau_{\leq n} R \oplus \pi_{n+1} (R)[n+2] \arrow{d} \\
			k \arrow{r} & k \oplus \pi_{n+1} (R)[n+2]
		\end{tikzcd}
	\]
	where the map $\tau_{\leq n} R \to k$ is the augmentation, the map $k \to k \oplus \pi_{n+1} (R) [n+2]$ is the inclusion and the map $\tau_{\leq n } R \oplus \pi_{n+1} (R)[n+2] \to  k \oplus \pi_{n+1} (R)[n+2]$ is induced by the augmentation. Each of the squares is a pullback diagram so the whole rectangle is a pullback diagram.

	The ultrasolid module $\pi_{n+1} (R)[n+2]$ is profinite, so we can write $k \oplus \pi_{n+1} (R)[n+2] \simeq \varprojlim k \oplus V_i [n+2]$ for $\{V_i\}$ an inverse system of finite-dimensional vector spaces. Hence, $k \oplus \pi_{n+1} (R)[n+2]$ can be written as an inverse limit of Artinian algebras.

	Since the image of $F$ is closed under limits, by induction, we only need to show that $0$-truncated objects are inverse limits of Artinian $k$-algebras. Let $R$ be $0$-truncated. Then, $R = \varprojlim R/\mathfrak{m}^i$, where $\mathfrak{m}$ is the augmentation ideal, and each map $R/\mathfrak{m}^{i+1} \to R/\mathfrak{m}^i$ also fits in a square-zero extension diagram (\cref{prop-nsmallextensions})
	\[
		\begin{tikzcd}
			R/\mathfrak{m}^{i+1} \arrow{r} \arrow{d} & R/\mathfrak{m}^i \arrow{d} \\
			k \arrow{r} & k \oplus V[1]
		\end{tikzcd}
	\]
	for some $V \in \Pro(\Vect_k^\omega)$. We may again procced by induction, so we only need to show that $k \oplus V[1]$ is in the image of $F$, for $V \in \Pro(\Vect_k^\omega)$. But this last part is clear because square-zero extensions commute with inverse limits and $k \oplus V[1] \in \CAlg_{k/ /k}^{\art}$ for $V \in \Vect_k^\omega$.
\end{proof}

\section{Animated ultrasolid $k$-algebras}
\label{section-animated}

\subsection{The $\LSym^*$ monad}

The goal is to define animated ultrasolid $k$-algebras as algebras over a sifted-colimit-preserving extension of the monad $\LSym^*$ defined in \cref{rem-ringmonad}. We will first need the following technical lemma.

\begin{lemma}
	\label{lemma-monadextension}
	Left Kan extension gives a symmetric monoidal equivalence of $\infty$-categories
	\begin{align*}
		\End_{\omega} (\Solid_k^\heartsuit) \simeq \End_{\Sigma}^\heartsuit (\Solid_{k, \geq 0})
	\end{align*}
	where $\End_\omega (\Solid_k^\heartsuit)$ is the full subcategory of filtered-colimit-preserving endofunctors and $\End_\Sigma^\heartsuit (\Solid_{k, \geq 0})$ is the full subcategory of sifted-colimit-preserving endofunctors preserving $\Solid_k^\heartsuit$.

	In particular, any filtered-colimit-preserving monad on $\Solid_k^\heartsuit$ may be uniquely extended to a sifted-colimit-preserving monad on $\Solid_{k, \geq 0}$ preserving the heart.
\end{lemma}

\begin{proof}
	By \cref{prop-ultrasolidPSigma}, left Kan extension gives an equivalence
	\begin{align*}
		\Fun(\ProVect, \Solid_k^\heartsuit) \simeq \End'_\Sigma (\Solid_{k, \geq 0})
	\end{align*}
	where $\End'_\Sigma (\Solid_{k, \geq 0})$ is the full subcategory of sifted-colimit-preserving endofunctors sending $\ProVect$ to $\Solid_k^\heartsuit$. Remember that (restricting to a suitable cardinality) $\Solid_k^\heartsuit \simeq \Ind (\ProVect)$ (\cref{lemma-ultrasolidind}), and filtered colimits commute with homology in $\Solid_k$. This gives an equivalence $\End'_\Sigma (\Solid_{k, \geq 0}) \simeq \End^\heartsuit _\Sigma (\Solid_{k, \geq 0})$.

	On the other hand, by \cref{lemma-ultrasolidind}, restriction gives an equivalence
	\begin{align*}
		\End_\omega (\Solid_k^\heartsuit) \simeq \Fun (\ProVect, \Solid_k^\heartsuit)
	\end{align*}
	Now it remains to show that the composition of equivalences
	\begin{align*}
		\End_\omega (\Solid_k^\heartsuit) \simeq \Fun (\ProVect, \Solid_k^\heartsuit) \simeq \End_\Sigma^\heartsuit (\Solid_{k, \geq 0})
	\end{align*}
	is given by left Kan extension. Let $F \in \End_\omega (\Solid_k^\heartsuit)$, and let $\tilde{F} \in \End_\Sigma^\heartsuit (\Solid_{k, \geq 0})$ be the left Kan extension of its restriction to $\ProVect$. Since $\Solid_k^\heartsuit \simeq \Ind(\ProVect)$ and filtered colimits commute with homology, we see that $\tilde{F}_{| \Solid_k^\heartsuit}$ has its image in $\Solid_k^\heartsuit \subset \Solid_{k, \geq 0}$, and since it preserves sifted colimits we get that $F \simeq \tilde{F}_{| \Solid_k^\heartsuit}$. Then, $\tilde{F}$ is the left Kan extension of $\tilde{F}_{| \Solid_k^\heartsuit} \simeq F$ (since it is left Kan extended from a smaller subcategory), which completes the proof of the equivalence.

	Finally, it is clear that left Kan extension gives a symmetric monoidal functor.
\end{proof}

\begin{defi}
	\label{defi-animatedring}
	By \cref{lemma-monadextension}, there is a unique sifted-colimit-preserving monad $\LSym^* : \Solid_{k, \geq 0} \to \Solid_{\geq 0}$ preserving the heart such that its restriction to $\Solid_k^\heartsuit$ agrees with the monad defined in \cref{rem-ringmonad}. That is, for $V \in \Solid_k^\heartsuit$,
	\begin{align*}
		\LSym^* (V) := \bigoplus_{n \geq 0} V^{\otimes n}_{\Sigma_n}
	\end{align*}
	where the orbits are taken in the ordinary category $\Solid_k^\heartsuit$.

	The $\infty$-category of \emph{animated ultrasolid $k$-algebras} is the $\infty$-category of algebras over $\LSym^*$. We will write this as $\CAlg_k^{\blacktriangle}$.

	Alternatively, for an uncountable strong limit cardinal $\sigma$, we can write $\Poly_k^\sigma$ for the full subcategory of ultrasolid $k$-algebras spanned by $\LSym^* (V)$ for $V \in \ProVectsigma$. Then, the $\infty$-category of animated ultrasolid $k$-algebras is the filtered colimit of $\mathcal{P}_\Sigma (\Poly_k^\sigma)$ over all uncountable strong limit cardinals $\sigma$
\end{defi}

\begin{rem}
	The $\LSym^*$ monad admits a natural grading. Write $\LSym^n: \Solid_{k, \geq 0} \to \Solid_{k, \geq 0}$ for the sifted-colimit-preserving extension of the functor $V \mapsto V^{\otimes n}_{\Sigma_n}$ on $\ProVect$. Then, $\LSym^* = \bigoplus_{n \geq 0} \LSym^n $.
\end{rem}

\begin{rem}
	Just as in the spectral analogue, the $\infty$-category of augmented animated ultrasolid $k$-algebras, $\CAlg_{k/ /k}^\blacktriangle$, can be written as algebras over the augmented animated ultrasolid $k$-algebra monad, $\LSym^*_{\nu}$. For $V \in \Solid_{k, \geq 0}$,
	\begin{align*}
		\LSym^*_{\nu} (V) := \bigoplus_{n > 0} \LSym^n (V) 
	\end{align*}
	Notice that the indexing starts at $1$, whereas the indexing for the $\LSym^*$ monad starts at $0$.

	Similarly, if $A \in \CAlg_k^\blacktriangle$, we can write the $\infty$-category $\CAlg_A^\blacktriangle := (\CAlg_k^\blacktriangle)_{A/}$ as algebras over the monad $\LSym_A^*$ on $\Solid_{A, \geq 0}$. This is given by the sifted-colimit-preserving extension of the functor sending $A \otimes V$ to $\LSym^* (V) \otimes A$ for $V \in \ProVect$.
\end{rem}

\begin{construction}[completed free animated ultrasolid $k$-algebras]
	\label{construction-freedeltacomplete}
	Just like in \cref{construction-completedfree}, we obtain a sequence of $\infty$-operads
	\begin{align*}
		\Comm^{\nu} \to \dots \to \Comm^{\nu}_{\leq i} \to \dots \to \Comm^{\nu}_{\leq 2} \to \Comm^{\nu}_{\leq 1}
	\end{align*}
	By (\cite{HA}, Example 4.7.3.11), $\Comm^{\nu}_{\leq i}$-algebras in $\Solid_k^\heartsuit$ are equivalent to algebras over the free $\Comm^{\nu}_{\leq i}$-algebra monad on $\Solid_k^\heartsuit$, which we will call $\LSym^{\leq i}_{\nu}$. By \cref{lemma-monadextension}, we can extend these monads to monads on $\Solid_{k, \geq 0}$, and we get a sequence of monads
	\begin{align*}
		\LSym^* _{\nu} \to \dots \to \LSym^{\leq i}_{\nu} \to \dots \to \LSym^{\leq 2}_{\nu} \to \LSym^{\leq 1}_{\nu}
	\end{align*}
	For $X \in \Solid_{k, \geq 0}$ we have that
	\begin{align*}
		\LSym^{\leq i}_{\nu} (X) \simeq \bigoplus_{0 < n \leq i} \LSym^n (X)
	\end{align*}
	Define the monad
	\begin{align*}
		\widehat{\LSym}^*_{\nu} := \varprojlim \LSym^{ \leq i} _{\nu}
	\end{align*}
	Remember that a limit of monads is taken as the limit of the underlying functors (\cite{HA}, Corollary 3.2.2.5). Hence, the monad $\widehat{\LSym}^*_{\nu}$ is given on $X \in \Solid_{k, \geq 0}$ by
	\begin{align*}
		\widehat{\LSym}^*_{\nu} (X) = \prod_{n > 0} \LSym^n (X)
	\end{align*}
	The above tower of monads induces a map of monads $\LSym^*_{\nu} \to \widehat{\LSym}^*_{\nu}$, so we get a functor
	\begin{align*}
		\Alg_{\widehat{\LSym}^*_{\nu}} (\Solid_{k, \geq 0}) \to \CAlg_{k/ /k}^{\blacktriangle}
	\end{align*}
	that is the identity on the underlying chain complexes. This means that any $\widehat{\LSym}^*_{\nu}$-algebra has the structure of an augmented animated ultrasolid $k$-algebra.

	Given $X \in \Solid_{k, \geq 0}$, we will abuse notation and write $\widehat{\LSym}^*_{\nu} (X)$ for the augmented animated ultrasolid $k$-algebra that is the image of the free $\widehat{\LSym}^*_{\nu}$-algebra on $X$. Additionally, we will write $\widehat{\LSym}^* (X)$ for the animated ultrasolid $k$-algebra obtained after forgetting the augmentation.
\end{construction}

\begin{lemma}
	\label{lemma-freedeltaalgebraslims}
	Let $n \geq 0$.
	\begin{enumerate}
		\item The functors $\LSym^n$ and $\LSym^*$ commute with sifted colimits.
		\item $\LSym^n$ and $\widehat{\LSym}^*$ preserve the subcategory $\Solid_{k, \geq 0}^{\aperf}$.
		\item $\LSym^n$ and $\widehat{\LSym}^*$ commute with inverse limits in $\ProVect$.
		\item If $V \in \Solid_{k, \geq 0}$ is $1$-connective, then $\LSym^n(V)$ is $n$-connective. If $V$ is $m$-connective for $m \geq 2$, $\LSym^n (V)$ is $(m + 2n -2)$-connective.
	\end{enumerate}
\end{lemma}

\begin{proof}
	It is clear that $\LSym^n$ and $\LSym^*$ commute with sifted colimits since they have been defined as sifted-colimit-preserving extensions.

	Profinite modules are closed under tensor products and finite colimits, so $\LSym^n (V) \in \Solid_{k, \geq 0}^{\aperf}$ for $V \in \ProVect$. Given that $\LSym^n$ commutes with geometric realisations, and $\Solid_{k, \geq 0}^{\aperf}$ is closed under geometric realisations (\cref{prop-complexesprohomology}), $\LSym^n$ preserves $\Solid_{k, \geq 0}^{\aperf}$. From this, it follows that $\widehat{\LSym}^*$ also preserves $\Solid_{k, \geq 0}^{\aperf}$.

	By \cref{prop-profinite}, each $\LSym^n$ commutes with inverse limits considered as a functor $\ProVect \to \ProVect$. Since inverse limits commute with homology in $\Solid_{k, \geq 0}^{\aperf}$, each $\LSym^n$ commutes with inverse limits as a functor $\ProVect \to \Solid_{k, \geq 0}^{\aperf}$. The same follows for $\widehat{\LSym}^*$.

	For the last part, $\Solid_{k, \geq m}$ is generated under sifted colimits by objects of the form $V[m]$, for $V \in \ProVect$, so it suffices to show the result for objects of this form. Furthermore, $\LSym^n (\Sigma^m V[0])$ can be written as a geometric realisation of objects of the form $\LSym^n (V^{r_i} [0])$. Given that $\LSym^m$ commutes with inverse limits in $\ProVect$ and preserves the subcategory $\Solid_{k, \geq 0}^{\aperf}$, and inverse limits commute with geometric realisations and homology in $\Solid_{k, \geq 0}^{\aperf}$ (\cref{prop-inversegeom}), we can just show the result for objects of the form $V[m]$ where $V \in \Vect_k^\omega$. But this is (\cite{SAG}, Proposition 25.2.4.1).
\end{proof}

\begin{construction}[Forgetful functor from animated to $\EE_\infty$ ultrasolid $k$-algebras]
	\label{construction-simplicialtoeinfty}
	Given an animated ultrasolid $k$-algebra, we can construct an associated $\EE_\infty$ ultrasolid $k$-algebra. This is given by the unique sifted-colimit-preserving functor $\Theta: \CAlg_k^{\blacktriangle} \to \CAlg_k^\blacksquare$ that maps $\LSym^* (V)$ for $V \in \Solid_k^\heartsuit$ to the associated discrete $\EE_\infty$ ultrasolid $k$-algebra.	

	For $A \in \CAlg_k^{\blacktriangle}$, we will usually write $A^\circ := \Theta (A)$.
\end{construction}

\begin{prop}
	\label{prop-derivedcomparison}
	Consider the functor $\Theta: \CAlg_k^{\blacktriangle} \to \CAlg_k^\blacksquare$ of \cref{construction-simplicialtoeinfty}. Then,
	\begin{enumerate}
		\item $\Theta$ preserves small limits and colimits.
		\item $\Theta$ is conservative.
		\item If $\operatorname{char} (k) = 0$ then it is an equivalence.
	\end{enumerate}
\end{prop}

\begin{proof}
	To show that $\Theta$ preserves small colimits, since it preserves sifted colimits, we only need to show that it preserves finite coproducts when restricted to $\Poly_k$, which is clear.

	Remember that there is a forgetful functor $\forget: \CAlg_k^\blacksquare \to \Solid_{k, \geq 0}$ that is conservative and preserves small limits and sifted colimits. Write $\psi: \CAlg_k^{\blacktriangle} \to \Solid_{k, \geq 0}$ for the functor $A \mapsto \forget (A^\circ)$.

	For objects of the form $\LSym^* (V)$, for $V \in \Pro(\Vect_k)$, the functor $\psi$ coincides with the forgetful functor $\CAlg_k^\blacktriangle \xrightarrow{\forget} \Solid_{k, \geq 0}$. Since both functors commute with sifted colimits, $\psi$ is the forgetful functor, so it preserves limits and is conservative. Given that the forgetful functor on $\EE_\infty$ ultrasolid $k$-algebras preserves limits and is conservative, we can deduce the same thing for $\Theta$.

	For the final part, by (\cite{SAG}, Proposition 25.1.2.2 (3)), there is a map $\Sym^* (V) \to \LSym^* (V)$ that is an equivalence for $V \in \Vect_k^\omega$. Given that both $\Sym^n$ and $\LSym^n$ commute with inverse limits of profinite vector spaces (\cref{lemma-completedfreelimits} and \cref{lemma-freedeltaalgebraslims}), we have a map $\Sym^* (V) \to \LSym^* (V)$ that is an equivalence on profinite vector spaces. By extending along sifted colimits, both monads are equivalent.
\end{proof}

\begin{rem}
	In particular, we see by the Barr-Beck-Lurie Theorem that the functor $\Theta$ is both monadic and comonadic. That is, we can understand animated ultrasolid $k$-algebras as algebras or coalgebras over a certain monad or comonad on $\EE_\infty$ ultrasolid $k$-algebras.
\end{rem}

\begin{defi}
	Given an animated ultrasolid $k$-algebra $A$, we will write $\Solid_A := \Solid_{A^\circ}$ for the $\infty$-category of $A$-modules.
\end{defi}

\subsection{The cotangent complex}

We will now follow the strategy in (\cite{SAG}, \S 25.3) to construct the cotangent complex of animated ultrasolid $k$-algebras.

\begin{construction}
	Write $\Solid^\Delta$ for the $\infty$-category of pairs $(A,M)$ where $A \in \CAlg_k^{\blacktriangle}$ and $M \in \Mod_A$.

	Alternatively, for an uncountable strong limit cardinal $\sigma$, we let $\catC_\sigma$ be the $\infty$-category of pairs $(A,M)$ where $A \in \CAlg_k^{\blacksquare, \heartsuit}$ is of the form $\LSym^*(V)$ for $V \in \Pro(\Vect_k^\omega)_{<\sigma}$ and $M \in \Mod_A$ is of the form $A \otimes W$ for $W \in \Pro(\Vect_k^\omega)_{< \sigma}$. Then, $\Solid^\Delta$ is the filtered colimit of $\mathcal{P}_\Sigma (\catC_\sigma)$ over all uncountable strong limit cardinals $\sigma$.

	Given $V,W \in \Pro(\Vect_k)$, we can form the square-zero extension $\LSym^* (V) \oplus (W \otimes \LSym^*(V)) \in \CAlg_k^{\blacksquare, \heartsuit} \subseteq \CAlg_k^{\blacktriangle}$. Hence, we obtain a sifted-colimit-preserving functor $\Solid^\Delta \to \CAlg_k^{\blacktriangle}$. We will write the image of $(A,M) \in \Solid^\Delta$ as $A \oplus M$ and refer to it as the \emph{trivial square-zero extension of $A$ by $M$}.
\end{construction}

\begin{defi}
	We will say that a map $A \to B$ in $\CAlg_k^\blacktriangle$ is a \emph{square-zero extension} if there is a pullback diagram of the form
	\[
		\begin{tikzcd}
			A \arrow{r} \arrow{d} & B \arrow{d} \\
			B \arrow{r} & B \oplus M[1]
		\end{tikzcd}
	\]
	for some $M \in \Solid_B$, where the horizontal map $B \to B \oplus M[1]$ is the $B$-algebra map. We then say that $A$ is a \emph{square-zero extension of $A$ by $M$}.
\end{defi}

\begin{defi}
	Let $A$ be an animated ultrasolid $k$-algebra and let $M$ be a connective $A$-module. We will write
	\begin{align*}
		\Der (A,M) := \Map_{\CAlg_{k / / A}^{\blacktriangle}} (A, A \oplus M)
	\end{align*}
	for the space of \emph{derivations of $A$ into $M$}.
\end{defi}

\begin{rem}
	The functor $M \mapsto A \oplus M$ commutes with all limits. To see this, we can compose the square-zero extension functor $\Solid_{A, \geq 0} \to \CAlg_k^{\blacktriangle}$ with the forgetful functor to $\Solid_k$, which commutes with all limits and sifted colimits. The functor obtained $\Solid_{A, \geq 0} \to \Solid_{k, \geq 0}$ commutes with sifted colimits and is given by $(A \otimes V)[0] \mapsto A \oplus (A \otimes V)[0]$ for $V \in \Pro(\Vect_k^\omega)$. Hence, it is given by $M \mapsto A \oplus M$ for any $M \in \Solid_{A, \geq 0}$; i.e. the functor $\Solid_{A, \geq 0} \to \Solid_{k, \geq 0}$ is just the sum of ultrasolid modules with $A$. This clearly commutes with all limits. Since the forgetful functor is conservative and commutes with all limits, the statement follows.
\end{rem}

\begin{rem}
	The forgetful functor to $\EE_\infty$ ultrasolid $k$-algebras commutes with square-zero extensions. Since both constructions commute with sifted colimits, we only need to check this on objects of the form $\LSym^* (V) \otimes W$ for $V, W \in \ProVect$, where it is obviously true.

	This means that if $A \in \CAlg_k^\blacktriangle$ and $M \in \Solid_{A, \geq 0} \simeq \Solid_{A^\circ, \geq 0}$, then $(A \oplus M)^\circ \simeq A^\circ \oplus M$.
\end{rem}

\begin{prop}
	\label{prop-cotcorepresentable}
	Let $A \in \CAlg_k^\blacktriangle$. Then, the functor $\Solid_{A, \geq 0} \to \mathcal{S}$ given by $M \mapsto \Der(A,M)$ is corepresentable.
\end{prop}

\begin{proof}
	Square-zero extension commutes with filtered colimits, and since any object is $\kappa$-compact for some cardinal $\kappa$, the functor $M \mapsto \Der (A,M) = \Map_{\CAlg_{k/ /A}^{\blacktriangle}} (A, A \oplus M)$ is accessible.

	Additionally, it preserves all limits. We are now done by (\cite{HTT}, Proposition 5.5.2.7).
\end{proof}

\begin{defi}
	Given $A \in \CAlg_k^{\blacktriangle}$, we will write $L_A$ for the $A$-module corepresenting the functor of \cref{prop-cotcorepresentable}. We will refer to this as the \emph{cotangent complex} of $A$.

	Given a map $A \to B$ in $\CAlg_k^{\blacktriangle}$, we will write $L_{B/A}$ for the cofibre of the map $L_A \otimes_A B \to L_B$, which is the \emph{cotangent complex of the map $A \to B$}.
\end{defi}

\begin{rem}
	By construction, the cotangent complex gives a left adjoint to the square-zero extension functor $\Solid^\Delta \to \CAlg_k^\blacktriangle$ where $R \in \CAlg_k^\blacktriangle$ gets mapped to the pair $(R, L_R)$.

	From the cofibre sequence, the relative cotangent complex gives a colimit-preserving functor $\Fun(\Delta^1, \CAlg_k^\blacktriangle) \to \Solid^\Delta$ that takes a map $A \to B$ to the pair $(B, L_{B/A})$.
\end{rem}

\begin{example}
	The cotangent complex is compatible with base change. That is, if we have a pushout diagram
	\[
		\begin{tikzcd}
			A \arrow{r} \arrow{d} & B \arrow{d} \\
			A' \arrow{r} & B' 
		\end{tikzcd}
	\]
	We can consider the following pushout diagram in $\Fun (\Delta^1, \CAlg_k^\blacktriangle)$
	\[
		\begin{tikzcd}
			(A \to A) \arrow{r} \arrow{d} & (A \to B) \arrow{d} \\
			(A' \to A') \arrow{r} & (A' \to B')
		\end{tikzcd}
	\]
	which leads to the pushout diagram in $\Solid^\Delta$
	\[
		\begin{tikzcd}
			(A,0) \arrow{r} \arrow{d} & (B, L_{B/A}) \arrow{d}  \\
			(A',0) \arrow{r} & (B', L_{B'/A'})
		\end{tikzcd}
	\]
	This implies that $L_{B'/A'} \simeq L_{B/A} \otimes_B B'$.
\end{example}

The following gives us a good approximation to the cotangent complex.

\begin{prop}
	\label{prop-cofibconnective}
	Let $f: A \to B$ be an $n$-connective map in $\CAlg_k^\blacktriangle$ for some $n \geq 0$; that is, $\pi_i (A) \to \pi_i (B)$ is an isomorphism for $i < n$ and a surjection for $i = n$. Let $K = \cofib (A \to B)$. Then, there is an $(n+2)$-connective map $K \otimes_A B \to L_{B/A}$.
\end{prop}

\begin{proof}
	We first construct the map. The identity map classifies a map $d: B \to B \oplus L_{B/A}$. Let $z: B \to B \oplus L_{B/A}$ be the map classified by the zero map. Then, we get a map $d - z : B \to L_{B/A}$. By construction, the composition $A \to B \to L_{B/A}$ is trivial so that it factors through $K \to L_{B/A}$. Using the base change adjunction we get the desired map $K \otimes_A B \to L_{B/A}$.

	It remains to show the result regarding connectivity. For this, we will construct $B$ from $A$ by taking suitable pushouts with symmetric powers. More precisely, we will construct an inductive system of $A$-algebras $\{A_i\}$ such that $B = \colim A_i$ and the map $A_i \to B$ is $(n+i)$-connective, adapting the strategy of (\cite{HA}, Lemma 7.4.3.15) to our setting.

	Set $A_0 := A$ and suppose that $f_i: A_i \to B$ has been constructed. Let $M := \fib (A_i \to B)$. Then, we have a commutative diagram in $\CAlg_k^\blacktriangle$,
	\[
		\begin{tikzcd}
			\LSym^* (M) \arrow{r} \arrow{d} & k \arrow{d} \\
			A_i \arrow{r}{f_i} & B
		\end{tikzcd}
	\]
	where the map $\LSym^* (M) \to k$ is the natural augmentation. Now write $A_{i+1} := A_i \otimes_{\LSym^* (M)} k$. We have a diagram
	\[
		\begin{tikzcd}
			\LSym^* (M) \arrow{r} \arrow{d} & k \arrow{d}  \\
			A_i  \arrow{r} \arrow[dr, "f_i", swap]  & A_{i+1} \arrow{d}{f_{i+1}}  \\
												& B
		\end{tikzcd}
	\]
	Let $N:= \fib (A_i \to A_{i+1})$. The map $f_{i+1} : A_{i+1} \to B$ induces a map of fibre sequences
	\[
		\begin{tikzcd}
			N \arrow{d}{g} \arrow{r} & A_i \arrow{r} \arrow{d}{=} & A_{i+1} \arrow{d}{f_{i+1}} \\
			M \arrow{r} & A_i \arrow{r} & B
		\end{tikzcd}
	\]
	so that $\fib (f_{i+1}) \simeq \fib (g)[1]$. Hence, it suffices to show that $\fib(g)$ is $(n+i-1)$-connective. For this, we will show that the map $g$ has a section. There is a map $s: M \to N \simeq \fib (A_i \to A_{i+1})$ induced by the pushout diagram for $A_{i+1}$. Then, the canonical limit diagram for $M$ factorises through $N$ as follows.
	\[
		\begin{tikzcd}
			M  \arrow{dr} \arrow{drr} \arrow{ddr}  & & & \\
																& N \arrow{r} \arrow{d} & A_i \arrow{d} \arrow{dr} &  \\
																& 0 \arrow{r} & A_{i+1} \arrow{r} & B
		\end{tikzcd}
	\]	
	where the square is the limit diagram for $N$. This implies that $s$ must be a section.

	Hence, $\fib (g)$ is a direct summand of $N$. Let $I:= \fib (\LSym^* (M) \to k) \simeq \bigoplus_{m > 0} \LSym^m (M)$. Then, $N \simeq \fib (A_i \to A_{i+1}) \simeq A_i \otimes_{\LSym^* (M)} I$. Since $M$ was assumed to be $(n+i-1)$-connective, so is $I$ (\cref{lemma-freedeltaalgebraslims}), and thus the tensor product too.

	We will now prove the statement about the cotangent complex. The constructed map $K \otimes_A B \to L_{B/A}$ is compatible with filtered colimits, so we may just show that each map $\cofib (A \to A_i) \otimes_A A_i \to L_{A_i/A}$ is $(n+2)$-connective. Write $M_i$ for the fibre of this map. Then, by the fundamental sequence for the cotangent complex, for each $i$ there is a fibre sequence
	\begin{align*}
		A_{i+1} \otimes_{A_i} M_i \to M_{i+1} \to N_i
	\end{align*}
	where $N_i$ is the fibre of the map $\cofib (A_i \to A_{i+1}) \otimes_{A_i} A_{i+1} \to L_{A_{i+1}/A_i}$. Hence, it suffices to show that each $N_i$ is $(n+2)$-connective. Remember that $A_{i+1} = k \otimes_{\LSym^* (M)} A_i$, where $M \simeq \fib (A_i \to B)$ is $(n+i-1)$-connected. Since our construction is compatible with base change, it suffices to show that the fibre of the map $\cofib (\LSym^* (M) \to k) \otimes_{\LSym^* (M)} k \to L_{k/\LSym^* (M)}$ is $(n+2)$-connective. Let $\tilde{N}_i$ be this fibre. We will compute it explicitly. Firstly, one can use the fundamental fibre sequence for the cotangent complex to obtain $L_{k/\LSym^* (M)} \simeq M[1]$. On the other hand,
	\begin{align*}
		\cofib (\LSym^* (M) \to k) \otimes _{\LSym^* (M)} k & \simeq \cofib (k \to k \otimes _{\LSym^* (M)} k) \\
		& \simeq \cofib (k \to \LSym^* (M[1])) \\
		& \simeq \bigoplus_{m > 0} \LSym^m (M[1])
	\end{align*}
	where in the second equivalence we used that $\LSym^*$ commutes with colimits.

	Hence, $\tilde{N}_i \simeq \fib (\bigoplus_{m > 0} \LSym^* (M[1]) \to M[1])$. The module $\LSym^1 (M[1]) \simeq M[1]$ is mapped isomorphically to the codomain, so it suffices to show that $\LSym^m (M[1])$ is $(n+2)$-connective for $m \geq 2$. But remember that $M[1]$ is $(n+i)$-connective, so the result follows from \cref{lemma-freedeltaalgebraslims}.
\end{proof}

The above allows a comparison between the topological and the algebraic cotangent complex.

\begin{construction}[Map from the topological to the algebraic cotangent complex]
	Let $f: A \to B$ be a map in $\CAlg_k^\blacktriangle$. We then have the identity map $L_{B/A} \to L_{B/A}$ which is classified by a map $B \to B \oplus L_{B/A}$. Forgetting to $\EE_\infty$ ultrasolid $k$-algebras we get a map $B^\circ \to B^\circ \oplus L_{B/A}$, which corresponds to a map $L_{B^\circ/A^\circ} \to L_{B/A}$.	
\end{construction}

\begin{rem}
	The map $L_{B^\circ/A^\circ} \to L_{B/A}$ can also be constructed as follows. Given that the forgetful functor to $\EE_\infty$ ultrasolid $k$-algebras preserves square-zero extensions, for $M \in \Solid_{B, \geq 0} \simeq \Solid_{B^\circ, \geq 0}$, we get a map
	\begin{align*}
		\Map_{\CAlg_{A/ /B}^\blacktriangle} (B, B \oplus M) \to \Map_{\CAlg_{A^\circ/ /B^\circ}^\blacksquare} (B^\circ, B^\circ \oplus M)
	\end{align*}
	By the cotangent complex adjunctions, this is a natural map
	\begin{align*}
		\Map_{\Solid_B} (L_{B/A}, M) \to \Map_{\Solid_B} (L_{B^\circ/A^\circ}, M)
	\end{align*}
	for every $M \in \Solid_{B, \geq 0}$, which translates into a map $L_{B^\circ/A^\circ} \to L_{B/A}$.
\end{rem}

\begin{prop}
	\label{prop-cotcomparisongeneral}
	Let $f: A \to B$ be an $n$-connective map in $\CAlg_k^\blacktriangle$ for some $n \geq 0$. Then, the map $L_{B^\circ/A^\circ} \to L_{B/A}$ is $(n+2)$-connective.
\end{prop}

\begin{proof}
	Let $K := \cofib (A \to B)$. The $(n+2)$-connective map $K \otimes_A B \to L_{B/A}$ from \cref{prop-cofibconnective} factorises through a map
	\begin{align*}
		K \otimes_A B \to L_{B^\circ/A^\circ} \to L_{B/A}
	\end{align*}
	By (\cite{HA}, Proposition 7.4.3.12), the map $K \otimes_A B \to L_{B^\circ/A^\circ}$ is $(2n + 2)$-connective, and $2n + 2 \geq n + 2$ for $n \geq 0$. This implies that the map $L_{B^\circ/A^\circ} \to L_{B/A}$ must also be $(n+2)$-connective.
\end{proof}

\begin{rem}
	Classically, the comparison between the topological and the algebraic cotangent complex of an $n$-connective map is $(n+3)$-connective (\cite{SAG}, Proposition 25.3.5.1), which is slightly sharper than our bound. To mimic that result, one needs to compute the category $\Sp (\CAlg^{\blacktriangle}_{k/ /A})$ for $A \in \CAlg_k^\blacktriangle$. There is an obvious candidate for this category; namely, the author believes that $\Sp (\CAlg_{k/ /A}^\blacktriangle) \simeq \LMod_{A^+}$, where $A^+ \simeq A \otimes (k \otimes_{\mathbb{S}} \ZZ)$ is an $\EE_1$ ultrasolid $k$-algebra (c.f. \cite{SAG}, Corollary 25.3.3.3 and Proposition 25.3.4.2), but the author has not managed to complete this computation.
\end{rem}

This allows us to show many results about the cotangent complex for animated ultrasolid $k$-algebras by comparing with their spectral analogue.

\begin{prop}
	Suppose that $f: A \to B$ is an $n$-connective map in $\CAlg_k^\blacktriangle$. Then, $L_{B/A}$ is $(n+1)$-connective. The converse holds if $f$ is an isomorphism on $\pi_0$.
\end{prop}

\begin{proof}
	This is an immediate consequence of \cref{lemma-cotangentcomplexiso} and \cref{prop-cotcomparisongeneral}.
\end{proof}

Another application of the comparison with the topological cotangent complex is showing that Postnikov towers converge and that each step is a square-zero extension.

Firstly, we need to understand what each truncation looks like. The following has the exact same proof as in \cref{prop-truncated}.

\begin{prop}
	Let $R \in \CAlg_k^\blacktriangle$ and $A \in \CAlg_R^\blacktriangle$. Then, the following are equivalent:
	\begin{enumerate}
		\item $\pi_i (A) = 0$ for $i > n$.
		\item $A$ is an $n$-truncated object in $\CAlg_R^\blacktriangle$. That is, for each $B \in \CAlg_R^\blacktriangle$, the space $\Map_{\CAlg_R^\blacktriangle} (B,A)$ is $n$-truncated (so its $i$th homotopy group vanishes for $i > n$).
	\end{enumerate}
	Write $\tau_{\leq n} : \CAlg_R^\blacktriangle \to \tau_{\leq n} \CAlg _R ^{\blacktriangle}$ for the left adjoint to the inclusion of $n$-truncated objects. Then, it coincides with truncation on $\Solid_k$. That is, the underlying chain complex of $\tau_{\leq n} A$ is $n$-truncated and the map $A \to \tau_{\leq n} A$ is $n$-connective.
\end{prop}

From this, it is clear that Postnikov towers converge. We must now show that each step is a square-zero extension.

\begin{prop}
	Let $f: A \to B$ be a map in $\CAlg_k^\blacktriangle$ such that
	\begin{enumerate}
		\item $\fib (f) \in \Solid_{k, [n, n]}$
		\item The multiplication map $\fib(f) \otimes_A \fib(f)$ is null-homotopic.
	\end{enumerate}
	Then, $f$ is a square-zero extension.
\end{prop}

\begin{proof}
	By \cref{prop-nsmallextensions}, the map $A ^\circ \to B^\circ$ is a square-zero extension. This is classified by a map of $A^\circ$-algebras $\eta^\circ: B ^ \circ \to B^\circ \oplus \Sigma \fib(f)$. Then,
	\begin{align*}
		\Map_{\CAlg_{A/ /B}^\blacktriangle} (B, B \oplus \Sigma \fib(f)) & \simeq \Map_{\Solid_B} (L_{B/A}, \Sigma \fib (f) ) \\
																							  & \simeq \Map_{\Solid_B} (\tau_{\leq n+1} L_{B/A}, \Sigma \fib(f)) \\
			& \simeq \Map_{\Solid_B} (\tau_{\leq n+1} L_{B^\circ/A^\circ}, \Sigma \fib(f)) \\
			& \simeq \Map_{\Solid_B} (L_{B^\circ/A^\circ}, \Sigma \fib(f)) \\
			& \simeq \Map_{\CAlg_{A^\circ/ /B^\circ}^\blacksquare} (B^\circ, B^\circ \oplus \Sigma \fib(f))
	\end{align*}
	where in the third equivalence we used \cref{prop-cotcomparisongeneral}.

	Hence, the square-zero extension diagram in $\CAlg_k^\blacksquare$ lifts to a commutative diagram in $\CAlg_k^\blacktriangle$
	\[
		\begin{tikzcd}
			A \arrow{r}{f} \arrow{d}{f} & B \arrow{d} \\
			B \arrow{r} & B \oplus \Sigma \fib(f)
		\end{tikzcd}
	\]
	which is a pullback diagram when forgetting to $\CAlg_k^\blacksquare$. Since the functor to $\EE_\infty$-rings is conservative and preserves limits, it is also a pullback diagram in $\CAlg_k^\blacktriangle$.
\end{proof}

\begin{cor}
	Postnikov towers converge in $\CAlg_k^\blacktriangle$. Moreover, each step in the Postnikov tower
	\begin{align*}
		\dots \to \tau_{\leq 2} A \to \tau_{\leq 1} A \to \tau_{\leq 0} A
	\end{align*}
	is a square-zero extension.
\end{cor}

\subsection{Complete profinite animated ultrasolid $k$-algebras}

We can reach the same results for complete profinite objects as for the spectral analogue.

\begin{defi}
	An augmented animated ultrasolid $k$-algebra $R$ is \emph{Artinian} if the underlying chain complex is in $\Mod_{k, \geq 0}$, $\pi_0 (R)$ is local and $\pi_* (R)$ is a finite-dimensional vector space. We write $\CAlg_{k/ /k}^{\art, \Delta}$ for the full subcategory of Artinian animated $k$-algebras.

	An augmented animated ultrasolid $k$-algebra $R$ is \emph{complete and profinite} if $\pi_0 (R)$ is complete and the underlying chain complex is in $\Solid_{k, \geq 0}^{\aperf}$. We will write $\widehat{\CAlg}^{\blacktriangle}_{k/ /k}$ for the full subcategory of complete profinite animated ultrasolid $k$-algebras.
\end{defi}

We will not include proofs for the following results, since they are identical to their spectral analogue (\cref{prop-mapsfromcomplete}, \cref{prop-completemonadic}, \cref{prop-cotcompleteiso}, \cref{prop-cotinverselimits} and \cref{theorem-proartinian})

\begin{prop}
	\label{prop-animatedcomplete}
	Let $\widehat{\CAlg}_{k/ /k}^\blacktriangle \subset \CAlg_{k/ /k}^\blacktriangle$ be the full subcategory of complete profinite animated ultrasolid $k$-algebras. Then,
	\begin{enumerate}
		\item For any $R \in \widehat{\CAlg}_{k/ /k}^\blacktriangle$ such that $\pi_0 (R) \in \CAlg_{k/ /k}^\heartsuit$ is complete, and any $V \in \Solid_{k, \geq 0}^{\aperf}$, the completion map induces an equivalence
			\begin{align*}
				\Map_{\CAlg_{k/ /k}^\blacktriangle} (\widehat{\LSym}^* (V), R) \simeq \Map_{\CAlg_{k/ /k}^\blacktriangle} (\LSym^* (V), R)
			\end{align*}
		\item There is a monadic adjunction $\widehat{\LSym}^*_{\nu} : \widehat{\CAlg}_{k/ /k}^\blacktriangle \leftrightarrows \Solid_{k, \geq 0}^{\aperf}: \forget$.
		\item If $f: A \to B$ is a map in $\widehat{\CAlg}_{k/ /k}^\blacktriangle$ and $L_{B/A} \simeq 0$, then it is an equivalence. Additionally, the functor $R \mapsto \cot (R) := L_{R/k} \otimes _R k$ commutes with inverse limits and lands in the category $\Solid_{k, \geq 0}^{\aperf}$.
		\item Let $\CAlg_{k/ /k}^{\art, \Delta} \subset \CAlg_{k/ /k}^\blacktriangle$ be the $\infty$-category of Artinian animated $k$-algebras. Then, there is an equivalence
			\begin{align*}
				\widehat{\CAlg}_{k/ /k}^\blacktriangle \simeq \Pro(\CAlg_{k/ /k}^{\art, \Delta})
			\end{align*}
	\end{enumerate}
\end{prop}

\section{The Lurie-Schlessinger criterion}
\label{section-fmp}

We will finally prove our generalisation of the Lurie-Schlessinger criterion.

\begin{defi}
	Throughout the rest of this section, let $\catC$ be either $\CAlg_{k/ /k}^\blacksquare$ or $\CAlg_{k/ /k}^\blacktriangle$, and let $\catC_{\art}$ be the $\infty$-category of Artinian objects in $\catC$. That is, $\catC_{\art}$ consists of those objects $X \in \catC$ such that the underlying chain complex of $X$ is in $ \Mod_{k, \geq 0}$, $\pi_0 (X)$ is local and $\pi_* (X)$ is finite-dimensional.

	We will also write $\widehat{\catC}$ for the $\infty$-category of complete profinite objects in $\catC$, so that $\Pro(\catC_{\art}) \simeq \widehat{\catC}$.
\end{defi}

\begin{rem}
	Since $\catC_{\art}$ has finite limits, $\widehat{\catC}$ must have all limits. Additionally, these limits are computed in $\catC$. This is because the inclusion $\catC_{\art} \to \catC$ commutes with finite limits, and the equivalence $\widehat{C} \simeq \Pro(\catC_{\art})$ is constructed as the inverse-limit-preserving extension of the inclusion $\catC_{\art} \to \catC$ (see the proof of \cref{theorem-proartinian}). The inclusion $\widehat{\catC} \to \catC$ then commutes with all limits (\cite{HTT}, Proposition 5.3.5.13).
\end{rem}

\begin{defi}
	A \emph{formal moduli problem} is a functor $X: \catC_{\art} \to \Space$ such that $X(k) \simeq *$ and for any map $A \to k \oplus k[n]$ in $\catC_{\art}$ with $n > 0$, the functor $X$ preserves the fibre product $A' := A \times_{k \oplus k[n]} k$. That is, there is a pullback diagram
	\[
		\begin{tikzcd}
			X(A') \arrow{r} \arrow{d} & X(A) \arrow{d} \\
			X(k) \arrow{r} & X(k \oplus k[n])
		\end{tikzcd}
	\]
	We write $\Moduli_k$ for the $\infty$-category of formal moduli problems. Notice that this does not carry any ultrasolid structure.
\end{defi}

\begin{rem}
	By right Kan extending, we can consider any formal moduli problem as an inverse-limit-preserving functor $X: \widehat{\catC} \to \Space$ that preserves the described fibre products in $\catC_{\art} \subset \widehat{\catC}$ and $X(k) \simeq *$.
\end{rem}

\begin{defi}
	A morphism $\phi: A \to B$ in $\catC_{\art}$ is \emph{elementary} if there is a pullback diagram
	\[
		\begin{tikzcd}
			A \arrow{r}{\phi} \arrow{d} & B \arrow{d} \\
			k \arrow{r} & k \oplus k[n]
		\end{tikzcd}
	\]
	for some $n > 0$.

	A map in $\catC_{\art}$ is \emph{small} if it is an equivalence or a finite composition of elementary morphisms.
\end{defi}

We have the following characterisation of small morphisms. We paraphrase the proof for $\EE_\infty$-rings from (\cite{DAGX}, Lemma 1.1.20) and adapt it for animated rings.

\begin{prop}
	A map $f: A \to B$ in $\catC_{\art}$ is small if and only if it induces a surjection $\pi_0 (A) \to \pi_0 (B)$.
\end{prop}

\begin{proof}
	Given that each elementary morphism induces a surjection on $\pi_0$, one direction is clear.

	For the other direction, let $C = \cofib(f)$, so that $\pi_* (C)$ is a finite-dimensional $k$-vector space. We will prove the statement by induction on the dimension of $\pi_* (C)$. If it is zero, then $f$ is an equivalence so we are done.

	Let $n$ be minimal so that $\pi_n (C) \neq 0$. Since $f$ is a surjection on $\pi_0$, we have that $n > 0$. Let $\mathfrak{m}$ be the maximal ideal of $\pi_0 (A)$. By Nakayama's lemma, $\pi_n (C)/\mathfrak{m} \neq 0$. This is a finitely generated $\pi_0 (A) /\mathfrak{m} \cong k$-module, so it is a finite-dimensional vector space. We can then collapse onto a subspace to obtain a surjective $\pi_0 (A)$-module map $\phi: \pi_n(C) \to k$.

	By (\cite{HA}, Theorem 7.4.3.1) in the spectral case and \cref{prop-cofibconnective} in the derived case, there is a $(n+1)$-connective map $C \otimes_A B \to L_{B/A}$. By $n$-connectivity of $C$, $\pi_n (C \otimes_A B) = \Tor^0 _{\pi_0 (A)} (\pi_0 (B), \pi_n (C)) \cong \pi_n (C)/(\ker (\pi_0 (A) \to \pi_0 (B))$, where the last isomorphism is by surjectivity of the map $f$ on $\pi_0$. Then, $\pi_n L_{B/A} \cong \Tor^0_{\pi_0 (A)} (\pi_n (C), \pi_0 (B)) \cong \pi_n (C)/(\ker (\pi_0 (A) \to \pi_0 (B)))$ (here we crucially use that $n > 0$ since the map $C \otimes_A B \to L_{B/A}$ is $(n+1)$-connective). 

	By truncating, giving a map $L_{B/A} \to k[n]$ is equivalent to giving a map $\tau_{\leq n} L_{B/A} \simeq \pi_n L_{B/A} [n] \to k[n]$. Hence, the map $\phi$ determines a map $L_{B/A} \to k[n]$, since $\ker (\pi_0 (A) \to \pi_0 (B)) \subseteq \mathfrak{m}$. This determines a map $ \eta: B \to B \oplus k[n]$ in $\catC_{A/ /B}$. We can then form the following diagram
	\[
		\begin{tikzcd}
			B' \arrow{r}{f'} \arrow{d} & B \arrow{d}{\eta} \\
			B \arrow{r} \arrow{d} & B \oplus k[n] \arrow{d} \\
			k \arrow{r} & k \oplus k[n]
		\end{tikzcd}
	\]
	where each of the squares is a pullback diagram, so the whole rectangle is a pullback diagram. It follows that $B' \simeq B \times_{k \oplus k[n]} k$, so that $f'$ is elementary.

	Since the map $\eta: B \to B \oplus k[n]$ is a map of $A$-algebras, we have a commutative diagram
	\[
		\begin{tikzcd}
			A \arrow{r}{f} \arrow{d}{f} & B \arrow{d}{\eta} \\
			B \arrow{r} & B \oplus k[n]
		\end{tikzcd}
	\]	
	where the horizontal map $B \to B \oplus k[n]$ is the natural inclusion. By the universal property of the pullback, we have a factorisation $A \to B' \xrightarrow{f'} B$. It now remains to show that $f'$ is small but this is true by the inductive hypothesis.
\end{proof}

The following is an immediate consequence of (\cite{DAGX}, Proposition 1.1.15).

\begin{prop}
	A functor $X: \catC_{\art} \to \Space$ is a formal moduli problem if and only if it preserves pullbacks of morphisms that are surjective on $\pi_0$.
\end{prop}

\begin{rem}
	\label{rem-genpullbacks}
	Let $X: \catC_{\art} \to \Space$ be a formal moduli problem, which we can right Kan extend to obtain an inverse-limit-preserving functor $X: \widehat{\catC} \to \Space$. 

	Suppose we have a pullback diagram in $\widehat{\catC}$ of the form
	\[
		\begin{tikzcd}
			R' \arrow{r} \arrow{d} & R \arrow{d}{\eta} \\
			k \arrow{r} & k \oplus V[n]
		\end{tikzcd}
	\]
	where $V \in \Pro(\Vect_k)$ and $n > 0$. Then, the map $\eta$ can be written as an inverse limit of maps of the form $R_i \to k \oplus V_i[n]$ where $V_i$ is finite-dimensional and $R_i \in \catC_{\art}$ (\cite{HTT}, Proposition 5.3.5.15). All the maps $R_i \to k \oplus V_i[n]$ are small because $\pi_0 (k \oplus V_i[n]) = k$. This means that we can write this pullback diagram as an inverse limit of pullback diagrams, all of which are preserved by $X$, so the functor $X$ preserves the above pullback.
\end{rem}

\begin{defi}
	Given a formal moduli problem $X$, we can form a spectrum object $\{X(k \oplus k[n])\}_{n > 0} \in \Sp$. Actually, this can be endowed with the structure of a $k$-module spectrum (\cite{DAGX}, Warning 1.2.9). We call this object, written as $T_X \in \Mod_k$, the \emph{tangent fibre} of $X$.

	We will call a formal moduli problem $X$ \emph{coconnective} if $T_X$ is itself coconnective.
\end{defi}

\begin{rem}
	The objects $k \oplus k[n]$ generate all of $\catC_{\art}$ under pullbacks that are preserved by formal moduli problems. It follows that a map $X \to Y$ of formal moduli problems is an equivalence if and only if the induced map $T_X \to T_Y$ is an equivalence.
\end{rem}

For the following proof, we follow the strategy of (\cite{dag}, Theorem 6.2.13) and adapt it to the ultrasolid setting.

\begin{theorem}
	\label{theorem-schlessinger}
	Let $X: \catC_{\art} \to \Space$ be a formal moduli problem. Then, it is corepresentable by a complete profinite ultrasolid object in $\widehat{\catC}$ if and only if $T_X$ is coconnective.
\end{theorem}

\begin{proof}
	Suppose $X = \Map_{\catC} (R, -)$. Then,
	\begin{align*}
		X(k \oplus k[n]) \simeq \Map _{\catC} (R, k \oplus k[n]) \simeq \Map_{\Solid_k} (\cot (R), k[n]) \simeq \Map_{\Solid_k} (k[-n], \cot^\vee (R))
	\end{align*}
	Here we have used that $\cot (R) \in \Solid_{k, \geq 0}^{\aperf}$ (\cref{prop-cotinverselimits} and \cref{prop-animatedcomplete}) and the contravariant equivalence $\Solid_{k, \geq 0}^{\aperf} \simeq \Mod_{k, \leq 0}^{\op}$ given by duality. Hence, $T_X \simeq \cot^\vee (R)$, which is clearly coconnective.

	Suppose now $T_X$ is coconnective. We fix the following notation. For $R \in \widehat{\catC}$, we will write $F_R$ for the formal moduli problem corepresented by $R$, and we will write $T_R := T_{F_R}$. Given a map $X \to Y$ of formal moduli problems, we will write $T_{Y/X}$ for the fibre of the map $T_X \to T_Y$.

	We will construct by induction an inverse system $\{R^i\}$ in $\widehat{\catC}$ with compatible maps $F_{R^i} \to X$ such that $R^{i+1} \to R^i$ is $(i+1)$-connective and $T_{X/R^i} \in \Mod_{k, \leq -i-2}$.

	Let $R^{-1} = k$, so $T_{X/R^{-1}} = T_X [-1]$ and the result is clear. Suppose that $\phi_i: F_{R^i} \to X$ has been constructed, so that $T_{X/R^i} \in \Mod_{k, \leq -i-2}$. We will now construct an inverse system $\{R_j^{i+1}\}$ of objects in $\widehat{\catC}$, where every map in the inverse system is $(i+1)$-connective, together with compatible maps $\phi_j^{i+1} : R_j^{i+1} \to X$. We begin by setting $R_0^{i+1} = R^i$.

	Now suppose $\phi_j^{i+1} : R_j^{i+1} \to X$ has been constructed. Let $V = \pi_{-i-2} T_{X/R^{i+1}_j}$, so we have a map $V[-i-2] \to T_{X/R^{i+1}_j}$. This is classified by a homotopy commutative diagram
	\[
		\begin{tikzcd}
			V[-i-2] \arrow{r}{\eta} \arrow{d} & T_{R^{i+1}_j} \arrow{d}{\phi_j^{i+1}} \\
			0 \arrow{r} & T_X
		\end{tikzcd}
	\]
	together with a homotopy between the two compositions. We also have
	\[
		\begin{tikzcd}
			\Map_{\Mod_k} (V[-i-2], T_{R^{i+1}_j}) = \Map_{\catC} (R^{i+1}_j, k \oplus V^\vee [i+2])
		\end{tikzcd}
	\]
	so that $\eta$ classifies a map $\eta: R^{i+1}_j \to k \oplus V^\vee[i+2]$. Hence, we can rewrite the above diagram as
	\[
		\begin{tikzcd}
			F_{k \oplus V^\vee [i+2]} \arrow{r}{\eta} \arrow{d} & F_{R^{i+1}_j} \arrow{d}{\phi_j^{i+1}} \\
			F_k \arrow{r} & X
		\end{tikzcd}
	\]
	with a $2$-cell exhibiting the commutativity of the diagram.

	Define $R^{i+1}_{j+1} = R_j^{i+1} \times_{k \oplus V^\vee[i+2]} k$, and this pullback is preserved by $X$ by \cref{rem-genpullbacks}. Remember that for $R \in \widehat{\catC}$, by the Yoneda embedding we have that $X(R) \simeq \Map (F_R, X)$. This means that the maps $F_k \to X$ and $F_{R_j^{i+1}} \to X$ factor through $F_{R^{i+1}_{j+1}}$. Let $\phi_{j+1}^{i+1} : F_{R^{i+1}_{j+1}} \to X$ be the induced map. Clearly, the map $R_{j+1}^{i+1} \to R_j^{i+1}$ is $(i+1)$-connective, and since $\hat{\mathcal{C}}$ is closed under limits, $R_{j+1}^{i+1} \in \hat{\mathcal{C}}$.

	We now claim that the induced map $\pi_{-i-2} T_{R_j^{i+1}/X} \to \pi_{-i-2} T_{R_{j+1}^{i+1}/X}$ is zero. It suffices to check that the map $V[-i-2] \to  T_{R^{i+1}_j/X} \to T_{R^{i+1}_{j+1}/X}$ vanishes. This map is induced by the commutative diagram
	\[
		\begin{tikzcd}
			F_{k \oplus V^\vee [i+2]} \arrow{r} \arrow{d}  & F_{R^{i+1}_j} \arrow{d} \arrow{ddr} & \\
			F_k \arrow{r} \arrow{drr} & F_{R^{i+1}_{j+1}} \arrow{dr} & \\
							  & & X
		\end{tikzcd}
	\]
	At the level of tangent complexes, we can omit $R_j^{i+1}$, identify maps from $F_{k \oplus V^\vee [i+2]}$ as maps from $V [-i-2]$ into the tangent complex, and write this as
	\[
		\begin{tikzcd}
			V[-i-2] \arrow{r} \arrow{d} & T_{R^{i+1}_{j+1}} \arrow{d} \\
			T_k \simeq 0 \arrow{ur} \arrow{r} & T_X
		\end{tikzcd}
	\]
	To check that the map $V[-i-2] \to T_{R^{i+1}_{j+1}/X} \simeq \fib (T_{R^{i+1}_{j+1}} \to T_X)$ is trivial we need to check that both the map $V[-i-2] \to T_X$ and the homotopy between the compositions $V[-i-2] \to T_X$ are trivial. These conditions are both satisfied because both the map and the homotopy factor through $T_k \simeq 0$.

	We have now constructed $R^{i+1}_{j+1}$. Set $R^{i+1} = \varprojlim R^{i+1}_j$. Given that Artinian objects are cocompact in $\hat{\mathcal{C}}$ we see that $F_{R^{i+1}} = \colim F_{R^{i+1}_j}$. By the universal property of colimits, there is a factorization $F_{R^i} \to F_{R^{i+1}} \to X$. The map $R^{i+1} \to R^i$ is clearly $(i+1)$-connective by construction. Given that this is also true of cotangent complexes, we see that the fibre of the map $T_{R^i/X} \to T_{R^{i+1}/X}$ is in $\Mod_{k, \leq -i-2}$. Additionally, $\pi_{-i-2} T_{R^{i+1}/X} = \colim \pi_{-i-2} T_{R^{i+1}_j/X}$. All of the maps in this filtered system vanish, so this colimit is zero. Hence, $\pi_{-i-2} T_{R^{i+1}/X} = 0$.

	We can now set $R = \varprojlim R^i$, so $R \in \hat{\mathcal{C}}$. This clearly induces a map $F_R \to X$ that is an equivalence on tangent complexes so $X \simeq F_R$.
\end{proof}

\begin{rem}
	The covariant Yoneda $R \mapsto \Map (R, -)$ gives a functor $\widehat{\catC} \to \Moduli_k^{\op}$. This coincides with the obvious fully faithful embedding $\widehat{\catC} \simeq \Pro(\catC_{\art}) \to \Moduli_k^{\op}$. This implies that $\widehat{\catC}^{\op}$ embeds fully faithfully into formal moduli problems.
\end{rem}

In \cite{partition} they classify formal moduli problems in terms of \emph{partition Lie algebras}, algebras over a monad $\Lie^\pi_k$ on $\Mod_k$. For $V \in \Mod_{k, \leq 0}^{\ft}$ (that is, $\pi_i (V)$ is finite-dimensional for every $i$), we have that $\Lie^\pi_k (V) \simeq \cot^\vee (k \oplus V^\vee)$. One can extend this to a sifted-colimit-preserving monad on all of $\Mod_k$ through Goodwillie calculus (\cite{partition}, Theorem 4.20 (2)).

Using the equivalence between partition Lie algebras and formal moduli problems, we can actually understand the objects in $\widehat{\catC}$ in terms of partition Lie algebras.

\begin{theorem}
	\label{theorem-partitionmonadicadj}
	There is a monadic adjunction
	\begin{align*}
		k \oplus (-^\vee) : \Mod_{k, \leq 0} \leftrightarrows \widehat{\catC}^{\op} : \cot^\vee
	\end{align*}
	The induced monad coincides with the partition Lie algebra monad defined in \cite{partition}.
\end{theorem}

\begin{proof}
	There is a fully faithful contravariant embedding from $\widehat{\catC}$ into formal moduli problems. This is because $\widehat{\catC} \simeq \Pro(\catC_{\art})$, which can be described as the opposite category of finite-limit-preserving functors $\catC_{\art} \to \Space$ (\cite{HTT}, Corollary 5.3.5.4). Since formal moduli-problems are functors $\catC_{\art} \to \Space$ preserving some finite limits, we get the fully faithful embedding $\widehat{\catC}^{\op} \xhookrightarrow{} \Moduli_k$.

	By \cref{theorem-schlessinger}, the image is precisely those formal moduli problems with coconnective tangent fibre. Write $\Moduli_{k, \leq 0}$ for the full subcategory of formal moduli problems with coconnective tangent fibre. By taking tangent fibres, we have a functor
	\begin{align*}
		\widehat{\catC}^\op \simeq \Moduli_{k, \leq 0} \to \Mod_{k, \leq 0}
	\end{align*}
	which corresponds to the functor $\cot^\vee$.

	By (\cite{partition}, Theorem 1.12), the tangent fibre functor is precisely the forgetful monadic functor from coconnective partition Lie algebras to chain complexes. Since the partition Lie algebra monad preserves coconnective complexes (\cite{partition}, Proposition 5.35 and Proposition 5.49), the functor $\Moduli_{k, \leq 0} \to \Mod_{k, \leq 0}$ is monadic, so we are done.
\end{proof}

\begin{rem}
	Note that one of the features of this construction is obtaining the partition Lie algebra monad on $\Mod_{k, \leq 0}$ as part of an adjunction (c.f. \cite{partition}, Theorem 4.20 (2)).
\end{rem}

\begin{rem}
	Remember that an augmented $\EE_\infty$ $k$-algebra or animated $k$-algebra $R$ is \emph{complete local Noetherian} if $\pi_0 (R)$ is complete, local and Noetherian and $\pi_i (R)$ is a finitely generated $\pi_0 (R)$-module for every $i$.

	Combining \cref{theorem-partitionmonadicadj} and (\cite{partition}, Proposition 5.31 and Corollary 5.46), we get a fully faithful embedding from complete local Noetherian $\EE_\infty$ or animated $k$-algebras into $\Moduli_{k, \leq 0} \simeq \widehat{\catC}^{\op}$. We will write $\widehat{\catC}^{\cN}$ for the essential image of this embedding. The category $\widehat{\catC}^{\cN}$ is precisely the $R \in \widehat{\catC}$ such that $\cot (R) \in \Mod_{k, \geq 0}^{\ft}$. That is, $\pi_i (\cot (R))$ is a finite-dimensional vector space for all $i$.
\end{rem}

\appendix

\section{Hypersheaves on the site of profinite sets}

In \cref{subsection-Einftyfree} we described the homology of free $\EE_\infty$ ultrasolid $k$-algebras via the homology of condensed anima. Here we recall some well-known results about condensed anima whose proof we recollect for completeness.

\begin{defi}
	The category of \emph{profinite sets} $\ProFin$ is the category of topological spaces that can be written as an inverse limit of finite discrete spaces. Equivalently, these are compact, Hausdorff spaces that are totally disconnected (\cite[\href{https://stacks.math.columbia.edu/tag/08ZY}{Tag 08ZY}]{stacks-project}). They form a site where coverings are finite families of jointly surjective maps.

	For a cardinal $\kappa$, we write $\ProFin_\kappa$ for the site of $\kappa$-small profinite sets. Given an $\infty$-category $\catC$ with small limits, we will write $\Shv^{\hyper} (\ProFin_\kappa, \catC)$ for the $\infty$-category of hypersheaves.

	A space is \emph{extremally disconnected} if it is compact, Hausdorff, and the closure of every open is open. We write $\ED$ for the category of extremally disconnected spaces, and $\ED_\kappa$ for the subcategory of $\kappa$-small extremally disconnected spaces.
\end{defi}

\begin{rem}
	By (\cite{Gleason}, Theorem 1.2), extremally disconnected spaces are projective in compact Hausdorff spaces. That means that if $S$ is extremally disconnected, for any solid diagram of compact Hausdorff spaces with $X \twoheadrightarrow Y$ surjective,
	\[
		\begin{tikzcd}
			& X  \arrow[d, twoheadrightarrow] \\
			S \arrow[ur, dashed] \arrow{r} & Y
		\end{tikzcd}
	\]
	there is a dotted arrow making the diagram commute.
\end{rem}

\begin{prop}
	\label{prop-condhyper}
	Let $\kappa$ be an uncountable strong limit cardinal and let $\catC$ be an $\infty$-category with all small limits and colimits, where finite products commute with filtered colimits. Then, restriction gives an equivalence
	\begin{align*}
		\Shv^{\hyper} (\ProFin_\kappa, \catC) \xrightarrow{\simeq} \Fun_{\sqcup} (\ED_\kappa^{\op}, \catC)
	\end{align*}
	where $\Fun_{\sqcup} (\ED_\kappa^{\op}, \catC) \subset \Fun (\ED_\kappa^{\op}, \catC)$ is the full subcategory of finite-product-preserving functors. The inverse is given by the sheafification of the left Kan extension.
\end{prop}

\begin{proof}
	Let $S \in \ProFin_\kappa$. Then, by taking disjoint unions, any hypercover of $S$ can be refined into a hypercover by single objects $T_\bullet \to S$.

	Let $X: \ED_\kappa^{\op} \to \catC$ be a finite-product-preserving functor, and let $\tilde{X}$ be its left Kan extension to all of $\ProFin_\kappa^{\op}$. We will now show that $\tilde{X}$ still preserves finite products. Let $T, T' \in \ProFin_\kappa$. Then, we have that
	\begin{align*}
		\tilde{X} (T \sqcup T') = \colim _{T \sqcup T' \to S} X(S)
	\end{align*}
	where the colimit runs over all $\kappa$-small extremally disconnected $S$ with a map $T \sqcup T' \to S$. Factoring each of these maps to the codiagonal $S \sqcup S \to S$, we see that (the opposite of) this category is cofinal to the category of pairs of maps $T \to S$ and $T' \to S$ for $S \in \ED_\kappa$. Then,
	\begin{align*}
		\tilde{X} (T \sqcup T') = \colim _{T \to S, T' \to S} X(S \sqcup S) = \colim _{T \to S , T' \to S} X(S) \times X(S) = \tilde{X} (T) \times \tilde{X}(T')
	\end{align*}
	where we used that finite products commute with filtered colimits in $\catC$.

	Let $S \in \ED_\kappa$. Any hypercover $T_\bullet \to S$ can be extended to a split diagram since $S$ is extremally disconnected and hence projective in compact Hausdorff spaces. Write $\tilde{X}^\vee$ for the sheafification. Then,
	\begin{align*}
		\tilde{X}^\vee (S) = \colim_{T_\bullet \to S} \Tot(\tilde{X} (T_\bullet)) \simeq \tilde{X} (S) \simeq X(S)
	\end{align*}
	where the colimit runs over all hypercovers of $S$. Here we used that $\tilde{X}$ will necessarily preserve split diagrams.

	This shows that the composition
	\begin{align*}
		\Fun_{\sqcup} (\ED_\kappa^{\op}, \catC) \to \Shv^{\hyper} (\ProFin_\kappa, \catC) \to \Fun_{\sqcup} (\ED_\kappa^{\op}, \catC)
	\end{align*}
	is an equivalence.

	Given a $\kappa$-small profinite set $S$, we can choose a hypercover $T_\bullet \to S$ by $\kappa$-small extremally disconnected spaces. This is clearly true if we don't restrict cardinalities by taking Stone-\v{C}ech compactifications. That is, for every profinite set $S$, we can cover $S$ by $\beta S^\delta$, where $S^\delta$ refers to taking the underlying discrete set. Now it is critical that we chose an uncountable strong limit cardinal, because if $S$ was $\kappa$-small, so will $\beta S^\delta$.

	Hence we see that any hypercomplete sheaf on $\ProFin_\kappa$ is determined by its values on $\ED_\kappa$. We just showed that left Kan extension doesn't change the values on extremally disconnected spaces so the composition
	\begin{align*}
		\Shv^{\hyper} (\ProFin_\kappa, \catC) \to \Fun_{\sqcup} (\ED_\kappa^{\op} , \catC) \to \Shv^{\hyper} (\ProFin_\kappa, \catC)	
	\end{align*}
	is also an equivalence.
\end{proof}

\begin{defi}
	\label{rem-cond}
	For an $\infty$-category $\catC$ and an uncountable strong limit cardinal $\kappa$, the $\infty$-category of \emph{$\kappa$-condensed objects of $\catC$}, $\Cond_\kappa(\catC)$, is the $\infty$-category of hypercomplete sheaves $\ProFin_\kappa^{\op} \to \catC$. Then, the $\infty$-category of \emph{condensed objects of $\catC$} is $\Cond(\catC) := \colim_\kappa \Cond_\kappa (\catC)$, where the colimit is taken over all uncountable strong limit cardinals.	

	By \cref{prop-condhyper}, $\Cond_\kappa(\catC)$ is the $\infty$-category of finite-product-preserving functors $\ED_\kappa^{\op} \to \catC$. We saw above that if $\ED_\kappa^{\op} \to \catC$ preserves finite products, so does its left Kan extension. Hence, $\Cond(\catC)$ can be described as the $\infty$-category of finite-product-preserving functor $\ED^{\op} \to \catC$ that are left Kan extended from $\ED_{\kappa}^{\op}$ for some $\kappa$.
\end{defi}

\begin{example}
	Let $R$ be a commutative ring. Then, $\Cond(\Mod_R^\heartsuit)$ is the category of condensed $R$-modules. We can form its derived category $D(\Cond(\Mod_R^\heartsuit))$ (as an $\infty$-category). By (\cite{condensed}, Warning 2.8), we see that $D(\Cond(\Mod_R^\heartsuit)) \simeq \Cond(\Mod_R)$.

	Under this correspondence, given a chain complex $V \in D(\Cond(\Mod_R^\heartsuit))$, interpreted as a hypercomplete $\Mod_R$-valued sheaf on $\ProFin$, and $S \in \ProFin$, we can associate $V(S)$ with $\RHom (R[S], V)$, where $R[S]$ is the free condensed $R$-module on $S$ (\cite{condensed}, Lecture II).
\end{example}

\begin{lemma}
	\label{lemma-hypercoverfree}
	Let $R$ be a ring, let $S$ be profinite and let $T_\bullet \to S$ be a hypercover by extremally disconnected spaces. Then, the complex of condensed $R$-modules
	\begin{align*}
		\dots \to R[T_1] \to R[T_0] \to R[S] \to 0
	\end{align*}
	is exact.
\end{lemma}

\begin{proof}
	This is \cite[\href{https://stacks.math.columbia.edu/tag/0517}{Tag 01GF}]{stacks-project} 
\end{proof}

\begin{example}
	\label{example-condani}
	The $\infty$-category of \emph{condensed anima} consists of finite-product-preserving functors $\ED^{\op} \to \Space$ that are left Kan extended from $\ED_\kappa^{\op}$ for some $\kappa$ (\cite{catrinmair}). By taking filtered colimits in (\cite{HTT}, Theorem 5.5.8.22), we see that if $\catC$ admits sifted colimits, then restriction gives an equivalence
	\begin{align*}
		\Fun_\Sigma (\Cond(\Space), \catC) \xrightarrow{\simeq} \Fun (\ED, \catC)
	\end{align*}
	where $\Fun_\Sigma (\Cond(\Space), \catC)$ is the full subcategory of sifted-colimit-preserving functors.

	Condensed sets are finite-product-preserving functors $\ED^{\op} \to \Set$ that are left Kan extended from $\ED_\kappa^{\op}$ for some $\kappa$, so we see that there is a fully faithful embedding $\Cond(\Set) \to \Cond(\Space)$. In particular, by (\cite{condensed}, Proposition 1.7), there is a fully faithful embedding from the $1$-category of compactly generated topological spaces into condensed anima.

	Also, there is a fully faithful embedding from finite sets into extremally disconnected spaces, and the latter are projective in $\Cond(\Space)$, so by (\cite{HTT}, Theorem 5.5.8.22), there is a fully faithful embedding $\Space \to \Cond(\Space)$ (here we use that $\Space$ is the sifted cocompletion of finite sets).

	Hence, condensed anima form a category that fits both topological phenomena from the $1$-category of compactly generated spaces and homotopical phenomena from the $\infty$-category of spaces.
\end{example}

\newpage
\bibliographystyle{alpha}
\bibliography{references}

\end{document}